\documentclass[reqno, 11pt]{amsart}
\usepackage[utf8]{inputenc}
\usepackage{amssymb,amsxtra,latexsym}
\usepackage{ragged2e}\justifying\let\raggedright\justifying% ????
\usepackage{mathrsfs}
\usepackage{hyperref}
\usepackage{textcomp}
\usepackage[scr=rsfs]{mathalfa}
\usepackage{epstopdf}
\usepackage{tikz}
\usepackage{pgfplots}
\usepackage{tikz-cd}
\usepackage[text={145mm,235mm},centering]{geometry}
\usepackage[mathscr]{eucal}
\usepackage[all]{xy}
\usepackage{mathrsfs}
\usepackage{xypic}
\usepackage{amsfonts}
\usepackage{amsmath}
\usepackage{amsthm,bm}
\usepackage{amssymb,tikz-cd}
\usepackage{latexsym}
\usepackage{tabularx}
\usepackage{graphicx}
\usepackage{pict2e}
\usepackage{tikz}
\usepackage{textcomp}
\usepackage[scr=rsfs]{mathalfa}
\usepackage[active]{srcltx}
\newtheorem{thm}{Theorem}[section]
\newtheorem{cor}[thm]{Corollary}
\newtheorem{lem}[thm]{Lemma}
\newtheorem{prop}[thm]{Proposition}

\newtheorem{prob}[thm]{Problem}
\theoremstyle{definition}
\newtheorem{defn}[thm]{Definition}
\theoremstyle{property}
\newtheorem{prope}[thm]{Property}
\theoremstyle{remark}
\newtheorem{rem}[thm]{Remark}

\numberwithin{equation}{section}
\definecolor{ceruleanblue}{rgb}{0.16, 0.32, 0.75}
\hypersetup{colorlinks=true,allcolors=ceruleanblue}
\allowdisplaybreaks
\numberwithin{equation}{section}
\makeatletter
\@namedef{subjclassname@2020}{%
  \textup{2020} Mathematics Subject Classification}
\makeatother

\makeatletter
\def\@tocline#1#2#3#4#5#6#7{\relax
  \ifnum #1>\c@tocdepth % then omit
  \else
    \par \addpenalty\@secpenalty\addvspace{#2}%
    \begingroup \hyphenpenalty\@M
    \@ifempty{#4}{%
      \@tempdima\csname r@tocindent\number#1\endcsname\relax
    }{%
      \@tempdima#4\relax
    }%
    \parindent\z@ \leftskip#3\relax \advance\leftskip\@tempdima\relax
    \rightskip\@pnumwidth plus4em \parfillskip-\@pnumwidth
    #5\leavevmode\hskip-\@tempdima
      \ifcase #1
       \or\or \hskip 1em \or \hskip 2em \else \hskip 3em \fi%
      #6\nobreak\relax
    \dotfill\hbox to\@pnumwidth{\@tocpagenum{#7}}\par
    \nobreak
    \endgroup
  \fi}
\makeatother

\usepackage[inline]{enumitem}
\usepackage{perpage} % Needed for \MakePerPage{footnote}  % Make the footnote recount per page.
\MakePerPage{footnote}  % Make the footnote recount per page.

\begin{document}

\title[Comparison of K\"{a}hler quotients of torus actions]
{Comparison of K\"{a}hler quotients of torus actions}

\author{Xiangsheng Wang}
\address{School of Mathematics, Shandong University, Jinan 250100, China}
\email{xiangsheng@sdu.edu.cn}
\author{Xiangdong Yang}
\address{Department of Mathematics, Lanzhou University, Lanzhou 730000, China}
\email{yangxd@lzu.edu.cn}


\date{\today}

% -----------------------------------------------------------

\begin{abstract}
  Let $T$ be a torus with the complexification $T^{\mathbb{C}}$ and $(X, ds^{2})$ a compact K\"{a}hler Hamiltonian $T$-manifold with the moment map $\Phi$ such that $T^{\mathbb{C}}$ acts on $X$ holomorphically.
  For each $\alpha$ in the moment body $\Phi(X)$, the K\"{a}hler quotient $X_{\alpha}=\Phi^{-1}(\alpha)/T$ is a reduced normal complex analytic
  space admitting a unique K\"{a}hler structure $\kappa_{\alpha}$ induced from $ds^{2}$.
  Inspired by the theory of variation of Geometric Invariant Theory, when $\alpha$ moves from a subpolytope (a connected component of the set of regular values of $\Phi$) to another one in the interior of $\Phi(X)$, we show that the quotient $X_{\alpha}$ undergoes a bimeromorphic transformation, and this enables us to compare the K\"{a}hler classes of the different quotients.
  In particular, as applications, we prove that each nondegenerate singular K\"{a}hler quotient has a partial and rational desingularisation which is obtained by shifting the moment map; moreover, we obtain a formula on the Riemann--Roch numbers of singular K\"{a}hler quotients.
\end{abstract}

% -----------------------------------------------------------
\maketitle
% -----------------------------------------------------------

\tableofcontents
%==================================
\section{Introduction}
\subsection{Background}
Let $(X, ds^{2})$ be a compact K\"{a}hler manifold and $T$ a torus with the Lie algebra $\mathfrak{t}$.
Suppose the complexification of the torus $T^{\mathbb{C}}$ acts on $X$ holomorphically and the K\"{a}hler metric $ds^{2}$ is $T$-invariant.
The K\"{a}hler form $\omega=-\mathrm{Im}\,ds^{2}$ gives rise to a symplectic structure on $X$.
If the $T$-action on $(X, \omega)$ is Hamiltonian, then we have the equivariant moment map $\Phi:X\rightarrow\mathfrak{t}^{\ast}$ satisfies the property
$$
d\Phi^{\xi}=\iota_{\xi_{X}}\omega
$$
for all $\xi\in\mathfrak{t}$.
Here $\Phi^{\xi}$ is the $\xi^{th}$-component of $\Phi$ defined by $\Phi^{\xi}(x)=\langle\Phi(x), \xi\rangle$ and $\xi_{X}$ is the vector field on $X$ induced by $\xi$.
Then the abelian convexity theorems \cite{At82, GS82} say that the moment body $\Delta=\Phi(X)$ is a convex polytope which is a union of convex subpolytopes of the same dimension.
To be more specific, the interior of the subpolytopes are disjoint and constitute the regular values of $\Phi$, and the boundary of the subpolytopes consist of the critical values of $\Phi$.

For any point $\alpha$ of $\Delta$, put $Z_{\alpha}=\Phi^{-1}(\alpha)$ and $X_{\alpha}=Z_{\alpha}/T$.
Then there is a natural commutative square of continuous maps
\begin{equation*}
\vcenter{
\xymatrix@C=1cm{
  Z_{\alpha} \,\ar[d]_{\pi_{\alpha}} \ar[r]^{\imath_{\alpha}} & X \ar[d]^{\pi} \\
  X_{\alpha}\, \ar[r]^{} & X/T}
  }
\end{equation*}
with respect to the induced sub-topology on $Z_{\alpha}$ and quotient topologies on $X/T$ and $X_{\alpha}$ respectively.
According to the Marsden--Weinstein reduction theorem \cite{MW74}, if $\alpha$ is a regular value of $\Phi$, i.e., an interior point of a subpolytope, then $\omega$ determines a unique orbifold symplectic form $\omega_{\alpha}$ on $X_{\alpha}$ such that $\pi_{\alpha}^{\ast}\omega_{\alpha}=\imath_{\alpha}^{\ast}\omega$.
In general, the level set $Z_{\alpha}$ has quadratic singularities and the reduced space $X_{\alpha}$ is seriously  singular.
Starting from the fundamental paper \cite{MW74}, the problem of understanding the structure of $X_{\alpha}$ when it acquires singularities, and how the geometry and topology of $X_{\alpha}$ varies as $\alpha$ varies has been one of central problems in symplectic geometry.

With different emphasises, extensive works on the structure of singular symplectic reduced spaces have been done by many symplectic and algebraic geometers, including \cite{ACG89,AJG90,SL91,Pf01}.
The first result on the variation of symplectic quotients can be traced back to
Duistermaat--Heckman \cite{DH82}.
They proved that if $\alpha$ varies in the interior of a subpolytope the diffeomorphism type of the regular symplectic quotient $X_{\alpha}$ is constant and the cohomology class $[\omega_{\alpha}]$ of the induced orbifold symplectic form on $X_{\alpha}$ varies linearly with $\alpha$.
Subsequently, Guillemin--Sternberg \cite{GS89} considered the \emph{wall-crossing problem} of $X_{\alpha}$ for the quasi-free Hamiltonian $T$-action.
More precisely, they showed that as $\alpha$ crosses a boundary (wall) separating two subpolytopes the diffeomorphism type of $X_{\alpha}$ undergoes a composition of a symplectic blow-up and a symplectic blow-down, and explained how the cohomology class of $[\omega_{\alpha}]$ changes.
In the paper \cite{Go01}, Godinho generalized Guillemin--Sternberg's wall-crossing theorem to general Hamiltonian $T$-action.
Recently, using the constructible property of the moment map, Mol \cite{Mol25} extended the linear variation theorem of Duistermaat--Heckman to singular symplectic reduced spaces.

From the viewpoint of complex geometry, for each $\alpha\in\Delta$, there is a canonical way to construct a pure dimensional, reduced and normal complex analytic space $X^{ss}(\Phi_{\alpha})/\!\!/T^{\mathbb{C}}$ called the K\"{a}hler quotient, which is canonically homeomorphic to the symplectic reduced space $X_{\alpha}$, see for example \cite{HHL94, HL94, Sj95, HS20}.
In particular, the usual K\"{a}hler metric $ds^{2}$ on $X$ induces a unique singular K\"{a}hler structure $\kappa_{\alpha}$ on $X^{ss}(\Phi_{\alpha})/\!\!/T^{\mathbb{C}}$.
Observe that the construction of the K\"{a}hler quotients $X^{ss}(\Phi_{\alpha})/\!\!/T^{\mathbb{C}}$ also depends on the choice of $\alpha\in\Delta$.
Hence a natural problem that arises now is as follows:
\begin{prob}\label{prob}
How the K\"{a}hler quotient $X^{ss}(\Phi_{\alpha})/\!\!/T^{\mathbb{C}}$, as a reduced normal K\"{a}hler space, varies as $\alpha$ varies in $\Delta$; especially, how to describe explicitly the degeneration from a regular K\"{a}hler quotient (as a K\"{a}hler orbifold) to a singular K\"{a}hler quotient?
\end{prob}

The K\"{a}hler quotient is a crossing-point of symplectic geometry and algebraic geometry, and the above problem is closely related to the one in the Geometric Invariant Theory (GIT): \emph{how the GIT-quotients vary with the linearisation of the group action.}
In 1986, Goresky--MacPherson \cite{GM86} first addressed the fundamental comparison problem of GIT-quotients and pioneered the natural morphisms among different quotients.
Suppose $X$ is a projective variety over an algebraically closed field on which the algebraic torus $(\mathbb{C}^{\ast})^{n}$ acts, and $X$ is equipped with an ample $(\mathbb{C}^{\ast})^{n}$-linear line bundle, Brion--Procesi \cite{BP90} studied the relationships among GIT-quotients of $X$.

If $X\subset\mathbb{CP}^{N}$ is a projective manifold and the action of the algebraic torus $(\mathbb{C}^{\ast})^{n}$ on $X$ can extends to a linear action on $\mathbb{CP}^{N}$.
Along the same line in \cite{GM86}, Hu \cite{Hu92} studied the geometry and topology of the (semi-)geometric quotients of torus actions, and obtained inductive formulae for the dimensions of intersection homology groups of the quotient.
In general, for the actions of any reductive algebraic groups on nonsingular complex projective varieties, using different methods,
Dolgachev--Hu \cite{DH98} and Thaddeus \cite{Th96} independently proved that the space of all possible linearizations is divided into finitely many chambers within which the quotient is constant, and when the linearization moves from one chamber to another by crossing a wall the corresponding quotient undergoes a birational transformation similar to Mori's flip.
Using equivariant cohomology theory, Fujiki \cite{Fu96} generalized some results of \cite{DH98} and \cite{Th96} to the general K\"{a}hler quotients.

\subsection{Summary of the results}
In general, let $K$ be a compact connected Lie group and $G=K^{\mathbb{C}}$ its complexification.
Let $G=K^{\mathbb{C}}$ act holomorphically on a possibly non-projective K\"{a}hler manifold $(X, ds^{2})$, and assume that the $K$-action on $(X, \omega)$ is Hamiltonian, then we get a K\"{a}hler Hamiltonian $K$-manifold.
Owing to the Holomorphic Slice Theorem (cf. \cite[Theorem 1.12]{Sj95} or \cite[(2.7) Theorem]{HL94}), the holomorphic $G$-action on $X$ is locally algebraic.
This implies that many results on the action of a reductive algebraic group on a projective algebraic variety have analytic counterparts for K\"{a}hler Hamiltonian $K$-manifolds in the analytic category.

This work is motivated by Problem \ref{prob}.
Consider a compact K\"{a}hler Hamiltonian $T$-manifold $(X,ds^{2},T^{\mathbb{C}},\Phi)$.
For simplicity, we fix some notations for later use:
\begin{itemize}
\item[--] $\mathrm{int}\,\Delta$= the interior of $\Delta=\Phi(X)$;

\item[--] $\partial\,\Delta$= the boundary of $\Delta$;

\item[--] $\mathrm{int}\,\Delta_{i}$= the interior of a subpolytope $\Delta_{i}$ of $\Delta$, which is a connected component of the set of regular values of $\Phi$;
\item[--] $\partial\,\Delta_{i}$= the boundary of a subpolytope $\Delta_{i}$ of
          $\Delta$, which consists of critical values of $\Phi$;
\item[--] $\mathrm{int}\,F$= the relative interior of a wall $F$ of $\Delta$, where a wall is a face of a subpolytope in $\Delta$\footnote{In the terminology of \cite{DH98}, the interior of a wall is a cell.
    Be careful that our definition of walls does not coincide with that defined in \cite{DH98}.
    In general, a wall as defined in \cite{DH98} consists of several walls in our sense.}.
\end{itemize}
It is well known to experts that $\Delta$ has a decomposition of subpolytope $\Delta_i$, see Proposition~\ref{prop:decomp} for a proof.

Inspired by the theory of Variation of Geometric Invariant Theory (VGIT) for projective varieties, we show the following:
\begin{thm}[=Theorem \ref{m1}]\label{main-thm-1}
Let $\Delta_{i}$ be a subpolytope of $\Delta=\Phi(X)$ and $\epsilon$ be a critical value of $\Phi$ lying in the intersection $\mathrm{int}\,\Delta\cap\partial\,\Delta_{i}$.
Then for any $\xi$ in $\mathrm{int}\,\Delta_{i}$, there exists a natural proper modification from the regular K\"{a}hler quotient $X^{s}(\Phi_{\xi})/\!\!/T^{\mathbb{C}}$ to the singular K\"{a}hler quotient $X^{ss}(\Phi_{\epsilon})/\!\!/T^{\mathbb{C}}$.
\end{thm}

Especially, we consider the wall-crossing of K\"{a}hler quotients.
Let $\Delta_{-}$ and $\Delta_{+}$ be two subpolytopes of $\Phi(X)$ separated by a wall $F$.

\begin{thm}[=Theorem \ref{m2}]\label{main-thm-2}
Let $\xi$, $\epsilon$ and $\zeta$ be arbitrary points lying in $\mathrm{int}\,\Delta_{-}$, $\mathrm{int}\,F$ and $\mathrm{int}\,\Delta_{+}$ respectively.
Then we have two natural proper modifications
\begin{equation*}
f_{\xi, \epsilon}:X^{ss}(\Phi_{\xi})/\!\!/T^{\mathbb{C}}\longrightarrow X^{ss}(\Phi_{\epsilon})/\!\!/T^{\mathbb{C}}
\end{equation*}
and
\begin{equation*}
f_{\zeta, \epsilon}:X^{ss}(\Phi_{\zeta})/\!\!/T^{\mathbb{C}}\longrightarrow X^{ss}(\Phi_{\epsilon})/\!\!/T^{\mathbb{C}}
\end{equation*}
which have the same center in $X^{ss}(\Phi_{\epsilon})/\!\!/T^{\mathbb{C}}$;
moreover, each fiber of $f_{\xi,\epsilon}$ (resp, $f_{\zeta,\epsilon}$) over the center is biholomorphic to the quotient of a weighted projective space by a finite group.
\end{thm}

A brief review of concepts in bimeromorphic geometry can be found in Appendix \ref{mod-bim}.
It is noteworthy that both Theorem \ref{main-thm-1} and Theorem \ref{main-thm-2} have the definite prototypes in the theory of VGIT, and we would like to add a comment about their relation with the known results in the literature.

The first difference is that, certainly, we deal with the general K\"ahler manifolds rather than only the projective ones.
As expected, many results can be extends to this larger category by using the more or less of the same methods coming from the algebraic case.
In \cite{Fu96}, Fujiki have made some exploration in this direction.
Theorem \ref{main-thm-1} is another example of such relatively straightforward generalizations.
In fact, we suspect that Theorem \ref{main-thm-1} is well known to experts.
But we can not find a suitable reference in the literature, and therefore we decide to give it a detailed proof.

As a contrast, Theorem \ref{main-thm-2}, although is still very similar to its VGIT counterpart, has crucial differences with the classical results both in its statement and proof.
Unlike the tradition in the theory of VGIT, where the K\"ahler form and the moment map value vary simultaneously as in \cite{DH98,Th96,Fu96}, we fix the K\"ahler form and only varies the values of the moment map.
Considering the explicit description of the image of the moment map under the torus action, we hope that in this smaller parameter space, the whole theory can become more concrete and transparent.
We can say such an object is fulfilled in some sense as we successfully remove a delicate concept, truly faithful cell defined in \cite{DH98}, from the statement.

Moreover, to investigate the complex geometry properties for the varying K\"{a}hler quotients, this small parameter space is also very useful since many interesting phenomena have already appeared in this range.
As to the proof of Theorem \ref{main-thm-2}, we remark that, even useful for the projective case, Theorem \ref{main-thm-2} still seems not to be a very trivial corollary of the classical result in \cite{DH98,Th96}.
The reason lies in the fact that the condition of truly faithful cells is not very straightforward to check.
Actually, a key point in proof of Theorem \ref{main-thm-2} is the fact that many subsets of the critical value of the moment map are automatically truly faithful, see Proposition \ref{G-b} and Remark \ref{rk:tf}.

Based upon Theorem \ref{main-thm-2}, we prove the following result which describes the wall-crossing of K\"{a}hler quotients by a quasi-free torus action in the complex analytic category.
\begin{thm}[=Theorem \ref{qusi-free-wc}]\label{main-thm-3}
If the $T$-action is quasi-free and the common center $B$ of the modifications $f_{\xi,\epsilon}$ and $f_{\zeta,\epsilon}$ is a finite set, then the quotient $X^{ss}(\Phi_{\zeta})/\!\!/T^{\mathbb{C}}$ can be obtained from the quotient $X^{ss}(\Phi_{\xi})/\!\!/T^{\mathbb{C}}$ by a blow-up at an explicit smooth center followed a blow-down at an explicit smooth center.
\end{thm}

In general, from Theorem \ref{main-thm-2}, a wall-crossing of K\"{a}hler quotients corresponds to a canonical bimeromorphic transformation.
Particularly, in the case of quasi-free torus action, Theorem \ref{main-thm-3} provides an explicit example of bimeromorphic map between compact K\"{a}hler manifolds for which the Strong Factorization Conjecture holds.

There is a subsidiary result of Theorem \ref{main-thm-2}.
Note that the induced K\"{a}hler structure $\kappa_{\alpha}$ on the quotient $X^{ss}(\Phi_{\alpha})/\!\!/T^{\mathbb{C}}$ determines a cohomology class $c_{1}(\kappa_{\alpha})$ in the second \v{C}ech cohomology group of the sheaf of locally constant $\mathbb{R}$-valued functions on $X^{ss}(\Phi_{\alpha})/\!\!/T^{\mathbb{C}}$.
Denote by $V$ the fiber product of $X^{s}(\Phi_{\xi})/\!\!/T^{\mathbb{C}}$ and $X^{s}(\Phi_{\zeta})/\!\!/T^{\mathbb{C}}$ over $X^{s}(\Phi_{\epsilon})/\!\!/T^{\mathbb{C}}$.
Based on the Duistermaat--Heckman Theorem \cite{DH82}, we obtain a natural comparison of K\"{a}hler classes $c_{1}(\kappa_{\alpha})$ which contains the critical value on the wall.
\begin{thm}[=Theorem \ref{comp-cla}]\label{thm1.4}
Let $l$ be a line segment with the ending points $\xi$ in $\mathrm{int}\,\Delta_{-}$ and $\zeta$ in $\mathrm{int}\,\Delta_{+}$ such that $l$ intersects with the wall $F$ at $\epsilon$ in $\mathrm{int}\,F$.
As $\alpha$ crossing the wall $F$ along $l$, the variation of the K\"{a}hler classes $c_{1}(\kappa_{\alpha})$ gives rise to a broken line segment in the second \v{C}ech cohomology group of the sheaf $\underline{\mathbb{R}}_{V}$.
\end{thm}

As applications, we show that all nondegenerate K\"{a}hler quotients have the same algebraic dimension (Theorem \ref{inv-alg-dim}), and each singular nondegenerate K\"{a}hler quotient admits a natural shift desingularisation which is partial and rational (Theorem \ref{par-des}).
This enables us to prove the invariance of Riemann--Roch numbers of nondegenerate K\"{a}hler quotients (Corollary \ref{inv-euler}).
Particularly, inspired by multiplicity formulae of singular symplectic quotients in \cite{Sj95}, we apply the main theorems to the computation of the Riemann--Roch numbers of singular K\"{a}hler quotients in the integral case (Theorem \ref{L}).

\subsection{Outline of the paper}
We devote Section \ref{pre} to a brief review on the reductive group actions on K\"{a}hler manifolds.
In Section \ref{nat-map}, we construct a natural morphism from a regular K\"{a}hler quotient to a singular one from the bimeromorphic geometry point of view.
In Section \ref{var-quo}, we study the variation of reduced spaces.
In Section \ref{bim-equ}, we study wall-crossing of the K\"{a}hler quotients by a quasi-free torus action.
In Section \ref{var-cla}, we study the change of the K\"{a}hler structures on quotients in an appropriate \v{C}ech cohomology group.
In Section \ref{apps}, we list several applications of the main result.
Appendix \ref{kah-spa} devotes to the definition of K\"{a}hler metric on a singular complex analytic space.
The final Appendix \ref{mod-bim} contains some basic notions in bimeromorphic geometry used in the paper.

\subsection*{Convention}
By a convex polytope (or a polytope for short) $E$ in a finitely dimensional linear space $V$ in this paper, we mean that $E$ is a bounded subset which is a finite intersection of closed half spaces in $V$.
Let $L \supseteq E$ be the support hyperplane of $E$, i.e., the hyperplane containing $E$ of the minimal dimension.
The relative interior of a convex subset $E$ is defined to be the topological interior of $E$ as a subset of $L$.
%==================================
\subsection*{Acknowledgements}
All the authors would like to express their gratitude to Professor Reyer Sjamaar and Professor Weiping Zhang for useful comments, and Song Yang for helpful suggestions.
The second author would like to thank the School of Mathematics of Sichuan University, Tianyuan Mathemtical Center in Southwest China and the Chern Institute of Mathematics (CIM) for hosting his research visits when he was working on this work.
X.\ Wang was partially supported by the National Natural Science Foundation of China (Grant No. 12471049, 12101361), the Project of Young Scholars of Shandong University.
X.\ Yang was partially supported by the National Nature Science Foundation of China (Grant No.12271225).
%=====================================
\section{Reductive group actions on K\"{a}hler manifolds}\label{pre}

Throughout of this paper, we denote by $K$ a compact connected Lie group with the Lie algebra
$\frak{k}$ and $G=K^{\mathbb{C}}$ its complexification.
A \emph{K\"{a}hler Hamiltonian $K$-manifold} is defined to be a tuple $(X, ds^{2}, G, \Phi)$
satisfying the following conditions:
\begin{itemize}
  \item [(i)] $ds^{2}$ is a $K$-invariant K\"{a}hler metric on $X$;
  \item [(ii)] $G$ acts on $X$ holomorphically;
  \item [(iii)] the $K$-action $(X, -\mathrm{Im}\,ds^{2})$ is Hamiltonian and $\Phi:
      X\rightarrow\mathfrak{k}^{\ast}$ is the moment map.
\end{itemize}
Throughout the paper, we assume that the generic infinitesimal isotropic subgroup of the $K$-action is trivial, i.e., the regular value of $\Phi$ is non-empty.

Consider a K\"{a}hler Hamiltonian $K$-manifold $(X, ds^{2}, G,\Phi)$.
Put $Z=\Phi^{-1}(0)$ and $X_{0}=Z/K$ equipped with the induced sub-topology and the quotient
topology respectively.
Then we have the \emph{inclusion-quotient diagram:}
$$
\xymatrix@=1cm{
  Z \ar[d]_{\pi} \ar[r]^{\imath} &   X     \\
  X_{0}                     }
$$
where $\pi$ is the orbit map and $\imath$ is the inclusion map.
There is a natural way to endow a structure sheaf on $X_{0}$ (cf. \cite{HS20}).
For any open subset $U$ of $Z$, a continuous function $f: U\rightarrow\mathbb{C}$ is
\emph{holomorphic} if it can extend to a holomorphic function on an open neighborhood of $U$ in $X$.
We denote by $\mathscr{O}_{Z}$ the sheaf of holomorphic functions on $Z$
and the structure sheaf of $X_{0}$ is defined to the $K$-invariant direct image of
$\mathscr{O}_{Z}$ along $\pi$, i.e.,
$\mathscr{O}_{X_{0}}=\pi^{K}_{\ast}\mathscr{O}_{Z}$.
\begin{thm}[{\cite[Theorem 1]{HS20}}]\label{reduced-normal}
The ringed space $(X_{0}, \mathscr{O}_{X_{0}})$ is a reduced normal complex analytic space.
\end{thm}
\begin{rem}
Actually, in \cite{HS20}, the authors proved the above theorem in a more general setting: the group $K$ is a closed Lie subgroup of the group of holomorphic isometries of the K\"{a}hler structure $ds^{2}$, and hence the $K$-action on $X$ is not necessarily given by restricting a holomorphic action of the complexified group $G=K^{\mathbb{C}}$ on $X$.
\end{rem}

\begin{defn}\label{stability}
Let $(X, ds^{2}, G,\Phi)$ a K\"{a}hler Hamiltonian $K$-manifold, a point $x$ in $X$ is called
\begin{itemize}
  \item [(i)] \emph{$\Phi$-unstable} if the intersection of the closure of the orbit $\overline{G\cdot x}$  and the zero level set $\Phi^{-1}(0)$ is empty;
  \item [(ii)]\emph{$\Phi$-semistable} if the closure of the orbit $\overline{G\cdot x}$ intersects the zero
      level set $\Phi^{-1}(0)$;
  \item [(iii)]\emph{$\Phi$-stable} if the orbit $G\cdot x$ intersects the zero level set
      $\Phi^{-1}(0)$ and the stabilizer $G_{x}$ is finite.
\end{itemize}
\end{defn}

We denote by $X^{us}(\Phi)$, $X^{ss}(\Phi)$, and $X^{s}(\Phi)$ the sets of
\emph{$\Phi$-unstable}, \emph{$\Phi$-semistable}, and \emph{$\Phi$-stable} points,
respectively.
By definition, we have  $X^{us}(\Phi)\cap X^{ss}(\Phi)=\emptyset$ and $X^{s}(\Phi)\subset
X^{ss}(\Phi)$.
The complement of $X^{s}(\Phi)$ in $X^{ss}(\Phi)$, denoted by $X^{sss}(\Phi)$, is called the set
of $\Phi$-\emph{strictly semistable} points.
%In particular, there exists a decomposition
%\begin{equation}
%X=X^{us}(\Phi)\cup X^{ss}(\Phi).
%\end{equation}
For any $\Phi$-semistable points $x$ and $x^{\prime}$, we say that $x$ and $x^{\prime}$ are
\emph{related} if and only if
$$
\overline{G\cdot x}\cap\overline{G\cdot
x^{\prime}}\cap X^{ss}(\Phi)\neq\emptyset.
$$
The relation $\sim$ is an equivalence relation and we denote the quotient by
$\Pi: X^{ss}(\Phi)\rightarrow X^{ss}(\Phi)/\!\!/G$.
Because each fiber $\Pi^{-1}(p)$ contains a unique closed $G$-orbit $G\cdot x$ for some $x$ in $\Phi^{-1}(0)$ (cf. \cite[Proposition 2.4]{Sj95}), we say that $p$ is of \emph{$G$-orbit type $(H)$} is the stabilizer $G_{x}$ is conjugate to $(K_{x})^{\mathbb{C}}$ in $G$.

Recall the definition of \emph{analytic Hilbert quotient} in the
analytic category which serves as the \emph{good quotient} in the Geometric Invariant Theory.
Suppose that $V$ is a complex space on which $G$ acts holomorphically.
The analytic Hilbert quotient of $V$ by the action of $G$ is a complex space $W$ together with a $G$-invariant holomorphic surjection $\pi: V\rightarrow W$ satisfying the following two conditions:
\begin{itemize}
  \item [(i)] $\pi$ is a locally Stein map, which means that there exists an open covering of $W$ by
      Stein subsets $\{U_{\alpha}\}$ such that $\pi^{-1}(U_{\alpha})$ is a Stein subset of $V$ for
      all $\alpha$,
  \item [(ii)] $\mathscr{O}_{W}$ is equal to the sheaf $\pi^{G}_{\ast}\mathscr{O}_{V}$.
\end{itemize}
If the analytic Hilbert quotient of a holomorphic $G$-space $V$ exists then it is unique
up to bihomomorphism and we will denote it by $V/\!\!/G$.
In summary, we have the following theorem which collects the main results on the quotients of K\"{a}hler Hamiltonian $K$-manifolds.
\begin{thm}\label{main-k-q}
Let $(X, ds^{2}, G, \Phi)$ be a compact K\"{a}hler Hamiltonian $K$-manifold.
Then we have:
\begin{itemize}
  \item [(i)] $X^{ss}(\Phi)$ is the smallest $G$-invariant open subset containing $\Phi^{-1}(0)$ such that the analytic Hilbert quotient $\Pi: X^{ss}(\Phi)\rightarrow X^{ss}(\Phi)/\!\!/G$ exists.
  \item [(ii)] The inclusion $\imath:\Phi^{-1}(0)\hookrightarrow X^{ss}(\Phi)$ induces a biholomorphic map
                     $$\tilde{\imath}:(X_{0}, \mathscr{O}_{X_{0}})\stackrel{\simeq}\longrightarrow
                     (X^{ss}(\Phi)/\!\!/G, \mathscr{O}_{X^{ss}(\Phi)/\!\!/G}).$$
  \item [(iii)] If $0$ is a regular value of $\Phi$, then $X^{ss}(\Phi)$ coincides with $X^{s}(\Phi)$.
  \item [(iv)] The stratification of $X_{0}$ by $K$-orbit types is identical with the
  stratification of $X^{ss}(\Phi)/\!\!/G$ by $G$-orbit types under the homeomorphism $\tilde{\imath}$.
  \item [(v)]   Each stratum of $X^{ss}(\Phi)/\!\!/G$ is a complex manifold whose closure is an complex-analytic subvariety of $X^{ss}(\Phi)/\!\!/G$.
  \item [(vi)] The K\"{a}hler metric $ds^{2}$ induces a unique K\"{a}hler structure on
                    $X_{0}$ which restricts to smooth K\"{a}hler metric on each stratum.
\end{itemize}
\end{thm}

The assertion $(i)$ comes from \cite[Proposition 2.4]{Sj95} and see \cite{HL94} for a proof in a more general setting.
The assertion $(ii)$ is a direct consequence of \cite[Theorem 2]{HS20}.
The assertions $(iii)$-$(v)$ come from \cite[Theorems 2.9-2.10]{Sj95}.
For a proof of $(vi)$, we refer to \cite{HHL94}.

From now on, we will identify $(X_{0}, \mathscr{O}_{X_{0}})$ and $(X^{ss}(\Phi)/\!\!/G, \mathscr{O}_{X^{ss}(\Phi)/\!\!/G})$ as complex analytic spaces.
Particularly, via refining the decomposition in (iv), we may assume that every stratum is connected.
By Theorem \ref{reduced-normal}, the K\"{a}hler quotient $X^{ss}(\Phi)/\!\!/G$ is normal.
Note that $X^{ss}(\Phi)/\!\!/G$ is connected.
Because every connected reduced normal complex space is irreducible (cf. \cite[Theorem, Page 168]{GR84}), we are led to the following result:
\begin{prop}\label{irreducible}
As a compact complex analytic space, the K\"{a}hler quotient $X^{ss}(\Phi)/\!\!/G$ is irreducible.
\end{prop}

Let $\mathfrak{X}_{\ast}(G)$ be the set of group homomorphisms
$\rho:\mathbb{C}^{\ast}\rightarrow G$, i.e. the set of one-parameter subgroup.
Then $\mathfrak{X}_{\ast}(G)$ can be identified with a subset of the Lie algebra $\mathfrak{k}$.
Given a subset $A\subset\mathfrak{k}^{\ast}$, we denote by $d_{\rho}(0, A)$ the signed distance
from $0$ to the boundary of the projection of the set $A$ onto the positive ray spanned by
$\rho^{\ast}$.

\begin{defn}[cf. {\cite[\S\,2.5]{DH98} or \cite[\S\,5.1]{Wa21}}]\label{hm-nf}
The \emph{Hilbert--Munford numerical function} on $X$ is defined to be
\begin{eqnarray*}
  \mathrm{M}^{\Phi}: X &\longrightarrow& \mathbb{R} \\
  x &\longmapsto& \sup_{\rho\in\mathfrak{X}_{\ast}(G)}d_{\rho}
  (0,\Phi(\overline{\rho(\mathbb{C}^{\ast})\cdot x})).
\end{eqnarray*}
\end{defn}

Using the Hilbert--Munford numerical function, we have the following:
\begin{prop}\label{prop-ss}
The sets unstable $X^{us}(\Phi)$, $X^{ss}(\Phi)$, and $X^{s}(\Phi)$ can be described as:
\begin{eqnarray*}
  X^{us}(\Phi) &=& \{x\in X\,|\, \mathrm{M}^{\Phi}(x)>0\},\\
  X^{ss}(\Phi) &=& \{x\in X\,|\, \mathrm{M}^{\Phi}(x)\leq0\} ,\\
  X^{s}(\Phi) &=& \{x\in X\,|\, \mathrm{M}^{\Phi}(x)<0\}.
\end{eqnarray*}
\end{prop}
\begin{proof}
See for example \cite[Corollary 12.7]{GRS21} or \cite[Proposition 5.1]{Wa21}.
\end{proof}

%=============================
\section{Natural morphisms among K\"{a}hler quotients}\label{nat-map}
The purpose of this section is to study K\"{a}hler quotients from the bimeromorphic geometry point of view.
From now on, we assume that $K=T$ is a torus of dimension $d$ and then $G=T^{\mathbb{C}}$.
We also assume that the Lie algebra $\mathfrak{t}$ is equipped with an invariant inner product.
Consider a compact K\"{a}hler Hamiltonian $T$-manifold $(X, ds^{2}, T^{\mathbb{C}}, \Phi)$.
Set $\omega=-\mathrm{Im}\,ds^{2}$.
The set $\Phi(X^{T})$ is finite and we call the points in $\Phi(X^{T})$ the \emph{vertices} of $\Phi$.
By the convexity theorem \cite{At82,GS82}, the moment body $\Delta=\Phi(X)$ is the convex hull of $\Phi(X^{T})$, which is the union of some subpolytopes of dimension $d$.
Moreover, the interiors of these subpolytopes are disjoint and consist of regular values of $\Phi$.
We call a face of a $d$-dimensional subpolytope a \emph{wall} which consists of critical values of $\Phi$.

As we have said in the introduction, our definition of a wall is different from the one in the sense of Dolgachev--Hu, compare to \cite[Definition 3.3.1]{DH98}.
By definition, a point $\gamma\in\mathfrak{t}^{\ast}$ lies in a wall of $\Delta$ if and only if there exists $x \in X$ such that $\Phi(x)=\gamma$ and $x$ is a critical point of $\Phi$, which is also equivalent to there exists a point $x$ in $X$ with $\mathrm{dim}\,T_{x}>0$ such that $\mathrm{M}^{\Phi-{\gamma}}(x)=0$.
In fact, fixed a wall $F$, in Proposition \ref{prop:dh-wall}, we can further show that for any $\gamma\in F$, we can choose the same point $x$ such that $\mathrm{dim}\,T_{x}>0$ and $\mathrm{M}^{\Phi-{\gamma}}(x)=0$ hold.
In other words, the wall defined in this paper is a subset of the wall defined in \cite{DH98} in general.
%==========
\subsection{Properties of the numerical function}
%==========
For any element $\alpha\in\mathfrak{t}^{\ast}$, we define the map $\Phi_{\alpha}:
X\rightarrow\mathfrak{t}^{\ast}$ by setting $\Phi_{\alpha}(x)=\Phi(x)-\alpha$, which is a moment map for the Hamiltonian $T$-action on $(X,\omega)$.
Let $r$ be an arbitrary positive real number.
Note that $r\omega$ is a symplectic form on $X$, and the $T$-action on $(X, r\omega)$ is Hamiltonian with the moment map $r\Phi$.
For any $\alpha,\,\beta\in\mathfrak{t}^{\ast}$, let $\gamma=(1-r)\alpha-r\beta$ for $r\in[0, 1]$, then we get $(1-r)\Phi_{\alpha}+r\Phi_{\beta}=\Phi_{\gamma}$.
Note that $X$ is compact and $G=T^{\mathbb{C}}$ is abelian.
The Hilbert--Mumford numerical function $\mathrm{M}^{\Phi_{\alpha}}(x)$ can be reformulated as
\begin{equation}\label{re-hm}
\mathrm{M}^{\Phi_{\alpha}}(x)=-\inf_{v\in\mathfrak{t}\setminus\{0\}}
\biggl\{\lim_{t\rightarrow+\infty}\bigl\langle\Phi_{\alpha}(\exp(\textbf{i}vt)x), v/|v|\bigr\rangle\biggr\}.
\end{equation}
The existences of the limit and the inf on the right-hand side of \eqref{re-hm} are guaranteed by \cite[Lemma 5.4]{GRS21} and \cite[Theorem 12.1]{GRS21}, respectively.
Consequently, for every $x\in X$, we can define a function on $\mathfrak{t}^{\ast}\cong\mathbb{R}^{d}$ by setting
\begin{eqnarray*}
  f^{HM}_{x}:\mathfrak{t}^{\ast} &\longrightarrow& \mathbb{R} \\
  \alpha&\longmapsto& \mathrm{M}^{\Phi_{\alpha}}(x).
\end{eqnarray*}

\begin{prop}\label{prop-convex}
With the same notions as above, we have:
\begin{itemize}
  \item [(i)] $\mathrm{M}^{r\Phi_{\alpha}}(x)=r\mathrm{M}^{\Phi_{\alpha}}(x)$, for $\alpha\in\mathfrak{t}^{\ast}$ and $r\geq0$;
  \item [(ii)]$\mathrm{M}^{\Phi_{\alpha}+\Phi_{\beta}}(x)\leq
      \mathrm{M}^{\Phi_{\alpha}}(x)+\mathrm{M}^{\Phi_{\beta}}(x)$, for $\alpha,\,\beta\in\mathfrak{t}^{\ast}$;
  \item [(iii)] the function $f^{HM}_{x}$ is a continuous convex function on $\mathfrak{t}^{\ast}$, for any $x\in X$.
\end{itemize}
\end{prop}
\begin{proof}
The assertion (i) is a direct consequence of the reformulation \eqref{re-hm}.
Also from \eqref{re-hm}, we have
\begin{eqnarray*}
\mathrm{M}^{\Phi_{\alpha}+\Phi_{\beta}}(x)&=&-\inf_{v\in\mathfrak{t}\setminus\{0\}}
\biggl\{\lim_{t\rightarrow+\infty}\bigl\langle\Phi_{\alpha}(\exp(\textbf{i}vt)x), v/|v|\bigr\rangle\\
& &\quad\quad\quad\,\,
+\lim_{t\rightarrow+\infty}\bigl\langle\Phi_{\beta}(\exp(\textbf{i}vt)x), v/| v|\bigr\rangle\biggr\}\\
&\leq&-\inf_{v\in\mathfrak{t}\setminus\{0\}}
\biggl\{\lim_{t\rightarrow+\infty}\bigl\langle\Phi_{\alpha}(\exp(\textbf{i}vt)x), v/|v|\bigr\rangle\biggr\}\\
& &-\inf_{v\in\mathfrak{t}\setminus\{0\}}
\biggl\{\lim_{t\rightarrow+\infty}\bigl\langle\Phi_{\beta}(\exp(\textbf{i}vt)x), v/|v|\bigr\rangle\biggr\}\\
&=&\mathrm{M}^{\Phi_{\alpha}}(x)+\mathrm{M}^{\Phi_{\beta}}(x).
\end{eqnarray*}
Consider the third assertion.
By the assertions (i) and (ii), we get
\begin{eqnarray*}
  \mathrm{M}^{(1-r)\Phi_{\alpha}+r\Phi_{\beta}}(x) &\leq& \mathrm{M}^{(1-r)\Phi_{\alpha}}(x)+\mathrm{M}^{r\Phi_{\beta}}(x) \\
   &=& (1-r)\mathrm{M}^{\Phi_{\alpha}}(x)+r\mathrm{M}^{\Phi_{\beta}}(x)
\end{eqnarray*}
for any $\alpha,\,\beta\in\mathfrak{t}^{\ast}$ and any $r\in[0, 1]$.
Observe that $(1-r)\Phi_{\alpha}+r\Phi_{\beta}=\Phi_{\gamma}$, where $\gamma=(1-r) \alpha+r\beta$.
This implies
$$
f^{HM}_{x}((1-r)\alpha+r\beta)\leq(1-r)f^{HM}_{x}(\alpha)+rf^{HM}_{x}(\beta)
$$
and therefore $f^{HM}_{x}$ is a convex function on $\mathfrak{t}^{\ast}$.
It follows from a standard result in convex geometry (cf. \cite[Theorem 2.2]{Gr07}) that $f^{HM}_{x}$ is continuous on $\mathfrak{t}^{\ast}$ and this completes the proof.
\end{proof}
%=====
\subsection{Quotients inside a subpolytope}
%=====

Consider the K\"{a}hler Hamiltonian $T$-manifold $(X, ds^{2}, T^{\mathbb{C}}, \Phi)$.
Given a connected component $D$ of the regular values set of $\Phi$, we have the following result.
\begin{lem}\label{lem-xi-zeta}
If $\alpha$ and $\beta$ lie in $D$, then
$X^{s}(\Phi_{\alpha})=X^{s}(\Phi_{\beta})$, and hence $X^{s}(\Phi_{\alpha})/\!\!/G=X^{s}(\Phi_{\beta})/\!\!/G$ as complex analytic spaces.
\end{lem}
\begin{proof}
We first show $X^{s}(\Phi_{\alpha})\subset X^{s}(\Phi_{\beta})$.
Since $D$ is a connected open subset of $\mathfrak{t}^*$, there exists a continuous path $l$ connecting $\alpha$ and $\beta$ lying in $D$.
By Proposition \ref{prop-ss}, for any $x\in X^{s}(\Phi_{\alpha})$, we have
$f^{HM}_{x}(\alpha)=\mathrm{M}^{\Phi_{\alpha}}(x)<0$.
If $f^{HM}_{x}(\beta)=\mathrm{M}^{\Phi_{\beta}}(x)=0$ then $X^{ss}(\Phi_{\beta})\setminus X^{s}(\Phi_{\beta})\neq\emptyset$ which leads to a contradiction with the fact that $\beta$ is a regular value of $\Phi$.
Due to the assertion (iii) in Lemma \ref{prop-convex}, the function $f^{HM}_{x}$ is continuous on the path $l$.
If $f^{HM}_{x}(\beta)=\mathrm{M}^{\Phi_{\beta}}(x)>0$, by continuity there exists a point $\gamma\in l$ such that $f^{HM}_{x}(\gamma)=\mathrm{M}^{\Phi_{\gamma}}(x)=0$ and therefore $X^{ss}(\Phi_{\gamma})\setminus X^{s}(\Phi_{\gamma})\neq\emptyset$, also a contradiction since $\gamma$ is a regular value of $\Phi$.
As a result, we obtain $f^{HM}_{x}(\beta)=\mathrm{M}^{\Phi_{\beta}}(x)<0$ which means $x\in X^{s}(\Phi_{\beta})$ and thus $X^{s}(\Phi_{\alpha})\subset X^{s}(\Phi_{\beta})$.
Likewise, we can show $X^{s}(\Phi_{\beta})\subset X^{s}(\Phi_{\alpha})$.
As a direct consequence of Theorem \ref{main-k-q}, the K\"{a}hler quotients $X^{s}(\Phi_{\alpha})/\!\!/G$ and $X^{s}(\Phi_{\beta})/\!\!/G$ are identical as reduced normal complex analytic spaces.
\end{proof}

The proof of Lemma~\ref{lem-xi-zeta} shows that the moment body can be decomposed into a union of convex subpolytopes.
Although this is a well known results for experts, we provide a proof here for the convenience of readers.
\begin{prop}\label{prop:decomp}
There exists finitely many top-dimensional subpolytopes $\Delta_i$ of the moment body $\Delta = \Phi(M)$ such that $\Delta = \cup_i \Delta_i$ and $\cup_i \mathrm{int}\,\Delta_{i}$ is the regular value set of the moment map $\Phi$.
\end{prop}
\begin{proof}
By the proof of the lemma above, any connected component $D_{i}$ of the regular values set of $\Phi$ can be reformulated as
\begin{equation*}
D_{i}=\bigcap_{x\in X^{s}(\Phi_{\alpha})}\{\gamma\in\mathfrak{t}^{\ast}\,|\, f^{HM}_{x}(\gamma)<0\},
\end{equation*}
where $\alpha$ is an arbitrary point in $D_{i}$.
Therefore, by the convexity of $f^{HM}_{x}$, the set $D_{i}$ is a convex open subset of $\mathfrak{t}^*$.
Moreover, by Atiyah's orbit closure convexity theorem, the critical value set of $\Phi$ is a union of finitely many polytopes in $\mathfrak{t}^*$, and none of these polytopes are of top dimension.
Denote by $\Delta_i$ the closure of $D_{i}$ in $\mathfrak{t}^{\ast}$.
As a result, $\Delta_i$ must be a top-dimensional convex subpolytope of $\Delta$; moreover, the moment body $\Delta$ is finite union of such $\Delta_i$.
\end{proof}

With the same notations in Proposition \ref{prop:decomp}, suppose
$$
\Delta=\Phi(X)=\bigcup_{i=1}^{k}\Delta_{i}.
$$
Let $\epsilon$ be a critical value of $\Phi$ which lies in $\mathrm{int}\,\Delta$.
Define a set
$$
\mathrm{I}(\epsilon)=\{j\,|\, \epsilon\in\Delta_{j}\}\subset\{1,2,\cdots, k\}.
$$
Then we have the following result.
\begin{lem}\label{ep-in}
The $\Phi_{\epsilon}$-stable set $X^{s}(\Phi_{\epsilon})$ is nonempty and is equal to the intersection
$$
X^{s}(\Phi_{\epsilon})=\bigcap_{j\in \mathrm{I}(\epsilon)}X^{s}(\Phi_{\xi_{j}}),
$$
where $\xi_{j}\in\mathrm{int}\,\Delta_{j}$ for every $j\in\mathrm{I}(\epsilon)$.
\end{lem}
\begin{proof}
Let $x$ be an arbitrary point in $X^{s}(\Phi_{\epsilon})$, then we have
$f^{HM}_{x}(\epsilon)=\mathrm{M}^{\epsilon}(x)<0$.
Due to the continuity of $f^{HM}_{x}$, there exists an open neighborhood $U_{\epsilon}$ of $\epsilon$ in $\mathfrak{t}^{\ast}\cong\mathbb{R}^{d}$ satisfying $U_{\epsilon}\subset\mathrm{int}\,\Delta$ and $f^{HM}_{x}(\epsilon^{\prime})<0$ for any $\epsilon^{\prime}$ in $U_{\epsilon}$.
According to Lemma \ref{lem-xi-zeta}, the stable set $X^{s}(\Phi_{\xi_{j}})$ is
constant as $\xi_{j}$ varies in $\mathrm{int}\,\Delta_{j}$.
Without loss of generality, we may choose $\xi_{j}$ lies in $\mathrm{int}\,\Delta_{j}\cap U_{\epsilon}$.
Then we obtain $f^{HM}_{x}(\xi_{j})<0$, i.e., $x\in X^{s}(\Phi_{\xi_{j}})$ and therefore
\begin{equation}\label{alf-in-bj}
X^{s}(\Phi_{\epsilon})\subset \bigcap_{j\in\mathrm{I}(\epsilon)}X^{s}(\Phi_{\xi_{j}}).
\end{equation}
Observe that each subpolytope $\Delta_{i}$ is of $d$-dimension.
Particularly, the interiors of these subpolytopes are disjoint and the union of all $\Delta_{i}$ equals $\Delta=\Phi(X)$ which is the convex hull of the set $\Phi(X^{T})$.
Note that there always exists a line segment through $\epsilon$ such that the ending points, denoted by $\xi_{l}$ and $\xi_{m}$, lie in the interiors of the subpolytopes $\Delta_{l}$ and $\Delta_{m}$ for some $l$ and $m$ in $\mathrm{I}(\epsilon)$.
As a result, we get $\epsilon=(1-r)\xi_{l}+r\xi_{m}$ for some $r\in(0,1)$.
For any point $x$ in $\bigcap_{j\in\mathrm{I}(\epsilon)}X^{s}(\Phi_{\xi_{j}})$, we have
$$
x\in \bigcap_{j\in\mathrm{I}(\epsilon)}X^{s}(\Phi_{\xi_{j}})
\subset X^{s}(\Phi_{\xi_{l}})\cap X^{s}(\Phi_{\xi_{m}}).
$$
This implies
$f^{HM}_{x}(\xi_{l})<0$ and $f^{HM}_{x}(\xi_{m})<0$.
On account of Proposition \ref{prop-convex},
we obtain
$$
f^{HM}_{x}(\epsilon)\leq
(1-r)f^{HM}_{x}(\xi_{l})+rf^{HM}_{x}(\xi_{m})<0
$$
which means that $x$ is $\Phi_{\epsilon}$-stable and thus
\begin{equation}\label{bj-in-alf}
\bigcap_{j\in\mathrm{I}(\epsilon)}X^{s}(\Phi_{\xi_{j}})\subset X^{s}(\Phi_{\epsilon}).
\end{equation}
Since the complements of $X^{s}(\Phi_{\xi_{j}})$ are complex analytic subsets of $X$, the stable point sets $X^{s}(\Phi_{\xi_{j}})$ are open and dense, and hence their intersection is non-empty.
Combining \eqref{alf-in-bj} with \eqref{bj-in-alf} completes the proof.
\end{proof}

Choose an arbitrary critical value $\epsilon$ of $\Phi$ in $\mathrm{int}\,\Delta$.
Suppose that $\epsilon$ lies in the boundary of a subpolytope $\Delta_{j}$.
For each $\xi\in\mathrm{int}\,\Delta_{j}$, we have $X^{ss}(\Phi_{\xi})=X^{s}(\Phi_{\xi})$ since $\xi$ is a regular value of $\Phi$.
By the continuity of the Hilbert--Munford numerical function, there is a natural inclusion map
$$i_{\xi, \epsilon}: X^{s}(\Phi_{\xi})\hookrightarrow X^{ss}(\Phi_{\epsilon})$$
which induces a natural continuous map between the analytic Hilbert quotients under the quotient topologies
\begin{equation*}
f_{\xi, \epsilon}:X^{s}(\Phi_{\xi})/\!\!/G\longrightarrow X^{ss}(\Phi_{\epsilon})/\!\!/G
\end{equation*}
such that the following diagram is commutative.
\begin{equation}\label{i-f-}
\vcenter{
\xymatrix@=1cm{
  X^{s}(\Phi_{\xi}) \ar[d]_{\Pi_{\xi}}
  \ar[r]^{i_{\xi, \epsilon}} & X^{ss}(\Phi_{\epsilon}) \ar[d]^{\Pi_{\epsilon}} \\
  X^{s}(\Phi_{\xi}) /\!\!/G
  \ar[r]^{f_{\xi, \epsilon}} & X^{ss}(\Phi_{\epsilon})/\!\!/G}
  }
\end{equation}
Geometrically, for any point $x$ of $X^{s}_{\xi}$, denote by $[G\cdot x]$ the equivalence class in the regular K\"{a}hler quotient $X^{s}(\Phi_{\xi}) /\!\!/G$, then $f_{\xi, \epsilon}([G\cdot x])$ is given by
$$
f_{\xi, \epsilon}([G\cdot x])=\Pi_{\epsilon}(\overline{G\cdot x}),
$$
where $\overline{G\cdot x}$ is the closure of the orbit in $X^{ss}(\Phi_{\epsilon})$.

Consider the map $f_{\xi, \epsilon}$.
For the sake of simplicity, we write $X^{s}_{\xi}=X^{s}(\Phi_{\xi})$ and
$X^{ss}_{\epsilon}=X^{ss}(\Phi_{\epsilon})$.
Let $V$ be an arbitrary open subset of $X^{ss}_{\epsilon}/\!\!/G$.
By definition, we have
$$\mathscr{O}_{X^{ss}_{\epsilon}/\!\!/G}(V)=\mathscr{O}_{X}(\Pi^{-1}_{\epsilon}(V))^{G}.$$
Due to the commutativity of \eqref{i-f-}, we get
$$
\Pi^{-1}_{\xi}\bigl(f_{\xi, \epsilon}^{-1}(V)\bigr)=\Pi^{-1}_{\epsilon}(V)\cap X^{s}_{\xi};
$$
moreover, for any section $h\in\mathscr{O}_{X}(\Pi^{-1}_{\epsilon}(V))^{G}$,   the restriction of $h$ to the
intersection $\Pi^{-1}_{\epsilon}(V)\cap X^{s}_{\xi}$ descends to a holomorphic function on
$f_{\xi, \epsilon}^{-1}(V)$.
As a result, the map $f_{\xi, \epsilon}$ determines a natural morphism of sheaves
$$
f_{\xi, \epsilon}^{\natural}:\mathscr{O}_{X^{ss}_{\epsilon}/\!\!/G}\longrightarrow
(f_{\xi, \epsilon})_{\ast}\mathscr{O}_{X^{s}_{\xi}/\!\!/G}.
$$
This implies that
\begin{equation}\label{nat-map-0}
f_{\xi, \epsilon}:(X^{s}_{\xi}/\!\!/G, \mathscr{O}_{X^{s}_{\xi}/\!\!/G})\longrightarrow
(X^{ss}_{\epsilon}/\!\!/G, \mathscr{O}_{X^{ss}_{\epsilon}/\!\!/G})
\end{equation}
is a morphism of complex spaces.
Note that both $X^{s}_{\xi}/\!\!/G$ and $X^{ss}_{\epsilon}/\!\!/G$ are compact.
The map $f_{\xi, \epsilon}$ is proper necessarily.

Put $B=\Pi_{\epsilon}(X^{sss}_{\epsilon})$, where $X^{sss}_{\epsilon}$ is the set of $\Phi_{\epsilon}$-strictly
semistable points, i.e., the complement of $X^{s}_{\epsilon}$ in $X^{ss}_{\epsilon}$.
We claim the following result
\begin{lem}\label{B-dense}
$B$ is a nowhere dense analytic subset in $X^{ss}_{\epsilon}/\!\!/G$ and $X^{sss}_{\epsilon}$ is saturated with respect to $\Pi_{\epsilon}$, i.e., $\Pi^{-1}_{\epsilon}(B)=X^{sss}_{\epsilon}$.
\end{lem}
\begin{proof}
On account of Theorem \ref{main-k-q}, there is a natural stratification of $X^{ss}_{\epsilon}/\!\!/G$ by
$G$-orbit types and each stratum is a complex manifold whose closure is an analytic subvariety of
$X^{ss}_{\epsilon}/\!\!/G$.
Note that $B$ is equal to the union of strata with positive-dimensional stabilizers.
It follows $\Pi^{-1}_{\epsilon}(B)\subset X^{sss}_{\epsilon}$ and thus $\Pi^{-1}_{\epsilon}(B)=X^{sss}_{\epsilon}$ since $B=\Pi_{\epsilon}(X^{sss}_{\epsilon})$.
Because the union of strata with finite stabilizers is equal to $X^{s}_{\epsilon}/\!\!/G$, the subset $B$ is identical with the complement of $X^{s}_{\epsilon}/\!\!/G$ in $X^{ss}_{\epsilon}/\!\!/G$.
Let $Q_{\mathrm{max}}$ be the maximal stratum of $X^{ss}_{\epsilon}/\!\!/G$ whose stabilizer is positive-dimensional.
The frontier condition of the stratification implies that the closure of $Q_{\mathrm{max}}$ is the
union of all strata with positive-dimensional stabilizers, namely,
$B=\overline{Q}_{\mathrm{max}}$.
Particularly, as a closed complex analytic subspace of $X^{ss}_{\epsilon}/\!\!/G$, the
codimension of $B$ in $X^{ss}_{\epsilon}/\!\!/G$ is greater than or equal to 1 and therefore $B$ is nowhere dense.
\end{proof}

We are ready to prove the main theorem of this section.
\begin{thm}\label{m1}
The natural map \eqref{nat-map-0} is a proper modification.
\end{thm}

\begin{proof}
Observe that the stable quotient $X^{s}_{\epsilon}/\!\!/G$ is open and dense in $X^{ss}_{\epsilon}/\!\!/G$ since $X^{ss}_{\epsilon}/\!\!/G=(X^{s}_{\epsilon}/\!\!/G)\coprod B$ and $B$ is a closed complex analytic subset of $X^{ss}_{\epsilon}/\!\!/G$, see Lemma \ref{B-dense}.
Consider the restriction map
\begin{equation}\label{res-f-}
f_{\xi, \epsilon}: W:=f_{\xi, \epsilon}^{-1}(X^{s}_{\epsilon}/\!\!/G)\longrightarrow X^{s}_{\epsilon}/\!\!/G.
\end{equation}
According to Lemma \ref{ep-in}, as a subset of $X^{s}_{\xi}$, the set of
$\Phi_{\epsilon}$-stable points $X^{s}_{\epsilon}$ is nonempty and hence we get a natural inclusion
$j: X^{s}_{\epsilon}\hookrightarrow X^{s}_{\xi}$ satisfying $i_{\xi,\epsilon}\circ j=\mathrm{id}_{X^{s}_{\epsilon}}$ and $j\circ(i_{\xi,\epsilon}|_{X^{s}_{\epsilon}})=\mathrm{id}_{X^{s}_{\epsilon}}$
which induces a holomorphic mapping $\underline{j}:X^{s}_{\epsilon}/\!\!/G\rightarrow
X^{s}_{\xi}/\!\!/G$.
From definition, we have $(f_{\xi, \epsilon})\circ\underline{j}=\mathrm{id}_{X^{s}_{\epsilon}/\!\!/G}$ and $\underline{j}\circ(f_{\xi, \epsilon}|_{W})=\mathrm{id}_{W}$.
This implies that \eqref{res-f-} is a biholomorphic mapping.
Since $X^{s}_{\xi}/\!\!/G$ is compact, its image under $f_{\xi,\epsilon}$ is a compact subset of $X^{ss}_{\epsilon}/\!\!/G$.
Note that $X^{ss}_{\epsilon}/\!\!/G$ is Hausdorff.
It follows that $f_{\xi,\epsilon}(X^{s}_{\xi}/\!\!/G)$ is closed in $X^{ss}_{\epsilon}/\!\!/G$, and we have
$$
X^{s}_{\epsilon}/\!\!/G\subset f_{\xi,\epsilon}(X^{s}_{\xi}/\!\!/G)
\subset X^{ss}_{\epsilon}/\!\!/G.
$$
Because $X^{s}_{\epsilon}/\!\!/G$ is dense in $X^{ss}_{\epsilon}/\!\!/G$, we obtain
$$
X^{ss}_{\epsilon}/\!\!/G=\overline{X^{s}_{\epsilon}/\!\!/G}\subset f_{\xi,\epsilon}(X^{s}_{\xi}/\!\!/G)
\subset X^{ss}_{\epsilon}/\!\!/G
$$
which means that $f_{\xi, \epsilon}$ is surjective.

Let $A$ be the inverse image of $B$ under $f_{\xi, \epsilon}$, which is a closed complex subspace of $X^{s}_{\xi}/\!\!/G$.
As $X^{s}_{\xi}/\!\!/G$ is irreducible, $A$ is a nowhere dense analytic subset of $X^{s}_{\xi}/\!\!/G$; moreover, the restriction
$$
f_{\xi, \epsilon}: (X^{s}_{\xi}/\!\!/G)\setminus A\longrightarrow (X^{ss}_{\epsilon}/\!\!/G)\setminus B.
$$
is identical with the biholomorphic map \eqref{res-f-}.
Consequently, we are led to the conclusion that $f_{\xi, \epsilon}$ is a proper modification in the sense of Definition \ref{mod}.
\end{proof}

%=======
\subsection{Quotients on the wall}
%=======
To study the geometry of quotients corresponding to the critical values on a wall, we need the
following result.
\begin{prop}\label{prop-x}
For any $x\in X$, if $\mathrm{dim}\,T_{x}>0$ then there exists a point $x^{\prime}\in X$ such that $\mathrm{dim}\,T_{x^{\prime}}=0$ and $x\in\overline{G\cdot x^{\prime}}$.
\end{prop}
\begin{proof}
Recall that $T=(\mathbb{S}^{1})^{d}$ and $G=T^{\mathbb{C}}=(\mathbb{C}^{\ast})^{d}$.
The assertion is proved by induction, which is a direct consequence of the plus and the minus
decompositions by Carrell--Sommese \cite[Proposition II]{CS78} for $(\mathbb{C}^{\ast})^{d}$-action.
If $d=1$, then $x$ is fixed by the $\mathbb{C}^{\ast}$-action.
Let $F$ be the connected component of the $\mathbb{C}^{\ast}$-fixed point set in which $x$ lies.
Then there exist a $\mathbb{C}^{\ast}$-invariant Zariski open subset $V$ of $X$ together with a $\mathbb{C}^{\ast}$-invariant maximal rank holomorphic surjection
$\pi:V\rightarrow F$ with vector space fibres.
In particular, the projection $\pi$ is defined by setting
$\pi(v)=\lim_{t\rightarrow0}t\cdot v$ for any $v\in V$,
where $t\in\mathbb{C}^{\ast}$.
This implies that the assertion is valid for $d=1$.
Put $H=(\mathbb{S}^{1})^{k-1}\times\{1\}$ and $L=\{1\}\times \mathbb{S}^{1}$.
Then we have $H^{\mathbb{C}}=(\mathbb{C}^{\ast})^{k-1}\times\{1\}$ and $L^{\mathbb{C}}=\{1\}\times\mathbb{C}^{\ast}$.
Suppose that the assertion holds for the $H^{\mathbb{C}}$-action.
Since $(\mathbb{C}^{\ast})^{k}$ is commutative, for the action of
$H^{\mathbb{C}}$ there exists a point $y\in X$ such that
$\mathrm{dim}\,H_{y}=0$ and $x\in\overline{H^{\mathbb{C}}\cdot y}$.
If $\mathrm{dim}\,L_{y}=0$, then we have $\mathrm{dim}\,G_{y}=0$ and
$
x\in \overline{H^{\mathbb{C}}\cdot y}\subset\overline{G\cdot y}.
$
Otherwise, $y$ is fixed by the $L^{\mathbb{C}}$-action.
By \cite[Proposition II]{CS78} again, there is a $L^{\mathbb{C}}$-invariant holomorphic
projection $\pi^{\prime}$ from a Zariski open subset $V^{\prime}$ to the connected component
of the
$L^{\mathbb{C}}$-fixed point set $F^{\prime}$ which contains the point $y$.
Consequently, there is a point $x^{\prime}\in V^{\prime}$ satisfying $\mathrm{dim}\,L_{x^{\prime}}=0$ and
$\pi^{\prime}(x^{\prime})=y$, i.e., $y\in\overline{L^{\mathbb{C}}\cdot x^{\prime}}$.
Note that $H^{\mathbb{C}}$ leaves the $L^{\mathbb{C}}$-fixed set invariant and so does the plus and the minus decompositions with respect to the $L^{\mathbb{C}}$-action.
It follows that $H_{x^{\prime}}$ is a subgroup of $H_{y}$ and therefore
$\mathrm{dim}\,H_{x^{\prime}}=0$.
As a result, we get $\mathrm{dim}\,T_{x^{\prime}}=0$ and
$
x\in \overline{H^{\mathbb{C}}\cdot y}
\subset \overline{H^{\mathbb{C}}\cdot(\overline{L^{\mathbb{C}}\cdot x^{\prime}})}
\subset\overline{G\cdot x^{\prime}}. $
\end{proof}

Using Proposition \ref{prop-x}, we obtain the following

\begin{lem}\label{F-gx}
Let $F$ be a wall of $\Delta$.
For any $x\in X$, if there exists an element $\lambda\in\mathrm{int}\,F$ such that $\lambda\in\Phi(\overline{G\cdot x})$ then $F\subset\Phi(\overline{G\cdot x})$.
\end{lem}
\begin{proof}
Suppose $F$ is a face of a $d$-dimensional subpolytope $\Delta_{i}$.
Put $P=\mathrm{int}\,\Delta_{i}$.
We divide the proof into two cases.

$\textbf{Case 1.}$ Assume that $\mathrm{dim}\,T_{x}=0$.
In this case, the image of the orbit closure $\Phi(\overline{G\cdot x})$ is a $d$-dimensional polytope.
Denote by $\Pi_{F}$ the hyperplane in $\mathfrak{t}^{\ast}\cong\mathbb{R}^{d}$ in which the wall $F$ lies.
Assume that $\Pi_{F}$ is defined by the equation $\langle v,\lambda\rangle=c$ for some $v\in\mathfrak{t}$ and  $c\in\mathbb{R}$.
Then we get two open half-planes in $\mathfrak{t}^{\ast}$ denoted by $H^{-}:=\{\lambda\in\mathfrak{t}^{\ast}\,|\,\langle v,\lambda\rangle<c\}$ and
$H^{+}:=\{\lambda\in\mathfrak{t}^{\ast}\,|\,\langle v,\lambda\rangle>c\}$.
Suppose $P\cap\Phi(\overline{G\cdot x})=\emptyset$, then we get $\overline{P}\cap\Phi(G\cdot x)=\emptyset$ and therefore $P$ and $\Phi(G\cdot x)$ lie on the opposite sides of $\Pi_{F}$, respectively.
Without loss of generality, we assume that $P$ lies in $H^{-}$ and $\Phi(G\cdot x)$ lies in $H^{+}$ (see Figure \ref{fig-0} below).
Since $\lambda\in \mathrm{int}\,F\cap\Phi(\overline{G\cdot x})$, we have $\overline{P}\cap\Phi(\overline{G\cdot x})\neq\emptyset$.
Suppose $F$ does not lie in a $(d-1)$-dimensional face of $\Phi(\overline{G\cdot x})$, then there exists a face $F^{\prime}$ of $\Phi(\overline{G\cdot x})$ such that $\mathrm{dim}\,F^{\prime}\leq d-2$ and $\overline{P}\cap F^{\prime}\neq\emptyset$.
Let $\eta\in\overline{P}\cap F^{\prime}$ and choose a point $y\in\overline{G\cdot x}$ such that $\Phi(y)=\eta$.
Let $\Pi_{F^{\prime}}$ be the plane in $\mathfrak{t}^{\ast}$ satisfying $F^{\prime}\subset \Pi_{F^{\prime}}$ and $\mathrm{dim}\,F^{\prime}=\mathrm{dim}\, \Pi_{F^{\prime}}$.
Then we get $\Pi_{F^{\prime}}\subset\Pi_{F}$.

We claim that  there is an open neighborhood $U$ of $y$ in $X$ such that $\Phi(U)\subseteq\Pi_{F}\cup H^{+}$.
Let $\mathfrak{t}_{y}$ be the Lie algebra of the stabilizer of $y$ and $\imath_{y}:\mathfrak{t}_{y}\hookrightarrow\mathfrak{t}$ the inclusion of Lie algebras.
Then $(X, \omega)$ become a Hamiltonian $T_{y}$-manifold with the moment map
$$
\Phi_{y}:=\imath^{\ast}_{y}\circ\Phi:X\longrightarrow\mathfrak{t}^{\ast}_{y},
$$
where $\imath^{\ast}_{y}$ is the adjoint map of $\imath_{y}$.
Choose an open neighborhood $\tilde{U}$ of $y$ in $X$ and an open neighborhood $W$ of $\eta$ in $\mathfrak{t}^{\ast}$ such that $\Phi(\tilde{U})\subseteq W$.
Note that $(\Pi_{F^{\prime}})^{\perp}\cong\mathfrak{t}^{\ast}_{y}$ with respect to the invariant inner product on $\mathfrak{t}^{\ast}$.
So that $\mathrm{dim}\,T_{y}\geq2$ and in this case the set $\Phi_{y}(\tilde{U})$ belongs to the convex hull of the critical values of $\Phi_{y}|_{\tilde{U}}$.
Denote by $\mathrm{Cri}(\Phi|_{\tilde{U}})$ (resp. $\mathrm{Cri}(\Phi_{y}|_{\tilde{U}})$) the set of critical values of $\Phi|_{\tilde{U}}$ (resp. $\Phi_{y}|_{\tilde{U}}$).
Then we have
\begin{eqnarray*}
  \mathrm{Cri}(\Phi_{y}|_{\tilde{U}}) &\subseteq&
\imath^{\ast}_{y}\bigl(\mathrm{Cri}(\Phi|_{\tilde{U}})\bigr) \\
   &\subseteq& \imath^{\ast}_{y}\bigl(\mathrm{Cri}(\Phi)\cap W\bigr)\\
   &\subseteq& \imath^{\ast}_{y}\bigl((F\cup H^{+})\cap W\bigr)\\
   &\subseteq& \imath^{\ast}_{y}\bigl(\Pi_{F}\cup H^{+}\bigr).
\end{eqnarray*}
It follows that $\Phi_{y}(\tilde{U})$ lies in $\imath^{\ast}_{y}\bigl(\Pi_{F}\cup H^{+}\bigr)$.
According to the local convexity theorem \cite[Theorem 3]{GS82}, there exists an open neighborhood $U$ of $y$ in $X$ and an open neighborhood $V$ of $\Phi_{y}(y)=\imath^{\ast}_{y}(\eta)$ in $\mathfrak{t}^{\ast}_{y}$ such that $U\subset\tilde{U}$ and
\begin{equation*}
\Phi_{y}(U)=V\cap\bigl(\imath^{\ast}_{y}(\eta)+S(\alpha_{1},\cdots,\alpha_{l})\bigr)
\subset\imath^{\ast}_{y}\bigl(\Pi_{F}\cup H^{+}\bigr),
\end{equation*}
where $\alpha_{1},\cdots,\alpha_{l}$ are the weights of the linear isotropy representation of $T_{y}$ on the tangent space $T_{y}X$ and
$$
S(\alpha_{1},\cdots,\alpha_{l})
=\{s_{1}\alpha_{1}+\cdots+s_{l}\alpha_{l}\,|\,s_{1},\cdots,s_{l}\geq0\}.
$$
Since $(\Pi_{F^{\prime}})^{\perp}\cong\mathfrak{t}^{\ast}_{y}$ the image of $\Pi_{F}$ under $\imath^{\ast}_{y}$ is a hyperplane in $\mathfrak{t}^{\ast}_{y}$ and therefore $\Phi_{y}(U)=\imath^{\ast}_{y}(\Phi(U))$ lies in a closed half-plane in $\mathfrak{t}^{\ast}_{y}$.
As a result, we are led to the conclusion $\Phi(U)\subseteq\Pi_{F}\cup H^{+}$.
%Since $P\cap\Pi_{F}=\emptyset$ and $P\cap H^{+}=\emptyset$ we are led to the conclusion $\Phi(U)\cap P=\emptyset$.

Recall that
$\Pi_{F}=\{\lambda\in\mathfrak{t}^{\ast}\,|\, \langle v,\lambda\rangle=c\}.$
Consider the $v$-component of the moment map $\Phi^{v}=\langle\Phi,v \rangle$.
It follows from \cite[Lemma 2.2]{At82} that $\Phi^{v}$ is a Morse function (in the sense of Bott) on $X$ which has only critical manifolds of even indexes, and therefore it has a unique local minimum and a unique local maximum.
Since $\Phi(U)\subseteq\Pi_{F}\cup H^{+}$ and $\Phi(y)=\eta$ belongs to $\Pi_{F}$, we get
$\Phi^{v}(x)\geq\langle\eta,v\rangle$
for any $x\in U$, i.e., $\langle\eta,v\rangle$ is a local minimum of $\Phi^{v}$.
The uniqueness of the local minimum for $\Phi^{v}$ implies $\Phi^{v}(x)\geq\langle\eta,v\rangle$
for all $x\in X$.
This contradicts with the fact $P\subset H^{-}$.
Consequently, we have $P\cap\Phi(\overline{G\cdot x})\neq\emptyset$, i.e., there exists an element $\lambda\in P$ such that $\lambda$ lies in $\Phi(\overline{G\cdot x})$, which means $x\in X^{ss}(\Phi_{\lambda})$.
Note that $P$ only consists of the regular values of $\Phi$.
This implies $X^{ss}(\Phi_{\alpha})=X^{s}(\Phi_{\alpha})$ and $X^{s}(\Phi_{\alpha})=X^{s}(\Phi_{\lambda})$ for any $\alpha\in P$, see Lemma \ref{lem-xi-zeta}.
Since $\alpha\in\Phi(\overline{G\cdot x})$ if and only if $x\in X^{s}(\Phi_{\alpha})$,
we get $\alpha\in\Phi(\overline{G\cdot x})$ for each $\alpha\in P$, i.e., $P\subset\Phi(\overline{G\cdot x})$.
As a result, we obtain
$
F\subset\overline{P}\subset\Phi(\overline{G\cdot x}).
$

$\textbf{Case 2.}$ Assume that $\mathrm{dim}\,T_{x}>0$.
According to Proposition \ref{prop-x}, there is a point $x^{\prime}\in X$ such that $\mathrm{dim}\,T_{x^{\prime}}=0$ and $x\in\overline{G\cdot x^{\prime}}$.
This implies $\overline{G\cdot x}\subset\overline{G\cdot x^{\prime}}$ and thus $\lambda\in\Phi(\overline{G\cdot x})\subset\Phi(\overline{G\cdot x^{\prime}})$.
It follows from the result in the first case that the wall $F$ lies in a $(d-1)$-dimensional face $\tilde{F}$ of the convex polytope $\Phi(\overline{G\cdot x^{\prime}})$.
In particular, $\lambda$ is an interior point of $\tilde{F}$ since it lies in $\mathrm{int}\,F$.
A result by Atiyah \cite[Theorem 2, (b)]{At82} shows that $\Phi(\overline{G\cdot x})$ is a face of $\Phi(\overline{G\cdot x^{\prime}})$ in which $\lambda$ lies.
This implies that $\tilde{F}$ is identical with $\Phi(\overline{G\cdot x})$ and thus we get $F\subset\Phi(\overline{G\cdot x})$.
\end{proof}

\begin{figure}[h]
\centering
\begin{tikzpicture}[scale=0.5]
  \shade[ball color=black] (1,5) circle (0.1);
  \shade[ball color=black] (3,2) circle (0.1);
  \shade[ball color=black] (8,2) circle (0.1);
  \shade[ball color=black] (10,5) circle (0.1);
  \shade[ball color=black] (8,9) circle (0.1);
  \shade[ball color=black] (3,9) circle (0.1);
  \shade[ball color=black] (9,7) circle (0.1);
  \shade[ball color=black] (12,11) circle (0.1);
  \shade[ball color=black] (14,6) circle (0.1);
  \draw  (15,8.5)node[right] {$\Phi(\overline{G\cdot x})$};
  \draw  (7.5,6.5)node[right] {$\eta$};
  \draw  (4.5,5.2)node[right] {$P$};
  \draw  (10.5,2)node[right] {$\Pi_{F}$};
  \draw[dashed,black] (9,7) circle (1.5cm);
  \fill[dashed,red, opacity=0.2] (9,7) circle (1.5cm);
  \draw  (13,4)node[right] {$W$};
  \draw[blue] (11,3)--(7,11);
  \draw(9,7)--(12,11)--(14,6)--(9,7);
  \draw[-latex,blue] (12,8)--(15,8.5);
  \draw[-latex,blue] (9.5,6.5)--(13,4);
  \fill[dashed,blue, opacity=0.1] (9,7)--(12,11)--(14,6)--(9,7);
  \draw (1,5) --(3,2)--(8,2)--(10,5)--(8,9)--(3,9)--(1,5);
  \fill[dashed,blue, opacity=0.1] (1,5) --(3,2)--(8,2)--(10,5)--(8,9)--(3,9)--(1,5);
\end{tikzpicture}
\caption{}
\label{fig-0}
\end{figure}

The following lemma is useful in the study of variation of the quotients.
\begin{lem}\label{sta}
Let $(X, ds^{2}, T^{\mathbb{C}}, \Phi)$ be a compact K\"{a}hler Hamiltonian $T$-manifold.
\begin{itemize}
  \item [(i)] If $\alpha$ lies in the boundary of the moment body $\Delta=\Phi(X)$, then $X^{s}(\Phi_{\alpha})=\emptyset$.
  \item [(ii)] If $\alpha$ and $\beta$ lie in the relative interior of a wall $F$, then
      $X^{ss}(\Phi_{\alpha})=X^{ss}(\Phi_{\beta})$.
\end{itemize}
\end{lem}
\begin{proof}
At first, for any $x\in X$, we claim $f^{HM}_{x}(\alpha)=\mathrm{M}^{\Phi_{\alpha}}(x)>0$ if $\alpha\in\mathfrak{t}^{\ast}\setminus\Delta$.
If the assertion does not hold, then there exists a point $x\in X$ such that
$f^{HM}_{x}(\alpha)=\mathrm{M}^{\Phi_{\alpha}}(x)\leq0$ and therefore $X^{ss}(\Phi_{\alpha})\neq\emptyset$.
As a result, we get
$$
\Phi^{-1}(\alpha)/T=(\Phi_{\alpha})^{-1}(0)/T=X^{ss}(\Phi_{\alpha})/\!\!/G\neq\emptyset
$$
which leads to a contradiction with the fact $\Phi^{-1}(\alpha)=\emptyset$.
Suppose $X^{s}(\Phi_{\alpha})\neq\emptyset$, then $f^{HM}_{x}(\alpha)=\mathrm{M}^{\Phi_{\alpha}}(x)<0$ for some $x\in X$.
By continuity of $f^{HM}_{x}$, there is an open neighborhood $U_{\alpha}$ of $\alpha$ in $\mathfrak{t}^{\ast}$ such that
$f^{HM}_{x}(\tau)=\mathrm{M}^{\Phi_{\tau}}(x)<0$ for all $\tau\in U_{\alpha}$.
Since $f^{HM}_{x}(\tau)>0$ when $\tau$ is outside of $\Delta$, we have
$U_{\alpha}\subset\Delta$ and this contradicts the assumption that $\alpha$ lies in the boundary of
$\Delta$.

It remains to prove the second assertion.
Recall that
$$
X^{ss}(\Phi_{\alpha})=\{x\in X\,|\,\overline{G\cdot x}\cap\Phi^{-1}(\alpha)\neq\emptyset\}
=\{x\in X\,|\,\alpha\in\Phi(\overline{G\cdot x})\}
$$
for any $\alpha\in\Delta=\Phi(X)$.
Let $\alpha$ be an arbitrary point in $\mathrm{int}\,F$, then we have
$\alpha\in\Phi(\overline{G\cdot x})$ for any $x\in X^{ss}(\Phi_{\alpha})$.
Due to Lemma \ref{F-gx}, the wall $F$ lies in $\Phi(\overline{G\cdot x})$.
As a result, for any $\beta\in \mathrm{int}\,F$, we obtain
$$
\beta\in\bigcap_{x\in X^{ss}(\Phi_{\alpha})}\Phi(\overline{G\cdot x}).
$$
This implies $X^{ss}(\Phi_{\alpha})\subseteq X^{ss}(\Phi_{\beta})$.
Similarly, using the same argument, we can prove $X^{ss}(\Phi_{\beta})\subseteq X^{ss}(\Phi_{\alpha})$ and this completes the proof.
\end{proof}

As another application of Lemma \ref{F-gx}, the following result explains the
relation between the two concepts of walls in this paper and \cite{DH98} as we
have mentioned at the beginning of this section.

\begin{prop}\label{prop:dh-wall}
Let $F$ be a wall of $\Delta$.
For any $x \in X$, if $x$ is a critical point of $\Phi$ and $\Phi(x)$ lying in the relative interior of $F$, then $\mathrm{dim}\,T_{x}>0$ and $\mathrm{M}^{\Phi_{\gamma}}(x)=0$ for any $\gamma \in F$.
\end{prop}
\begin{proof}
Since $x$ is a critical point of $\Phi$, we have $\mathrm{dim}\,T_{x}>0$ by the gradient flow interpretation of complex orbits.

To see another conclusion, we first notice that for any $\gamma \in \mathfrak{t}^*$, the numerical function $\mathrm{M}^{\Phi_{\gamma}}(x)$ is equal to 0 if and only if $\gamma \in \partial (\Phi(\overline{G\cdot x}))$.\footnote{Although $\Phi(\overline{G\cdot x})$ is a convex polytope, we use $\partial (\Phi(\overline{G\cdot x}))$ to denote the topological boundary of it rather than the polytope boundary.}
Let $\Phi(x) = \gamma_0$.
Since $x$ is a critical point of $\Phi$, we get $\dim \Phi(\overline{G\cdot x}) < \dim \mathfrak{t}^*$.
As a result, we have
\begin{equation}\label{eq:bd-gx}
\Phi(\overline{G\cdot x}) =  \partial (\Phi(\overline{G\cdot x})).
\end{equation}
Moreover, by assumption, the element $\gamma_0$ lies in the relative interior of $F$.
From Lemma \ref{F-gx}, we get $F \subseteq \Phi(\overline{G\cdot x})$.
Therefore, (\ref{eq:bd-gx}) implies that for any $\gamma \in F$, we obtain $\gamma \in \partial (\Phi(\overline{G\cdot x}))$, from which $\mathrm{M}^{\Phi_{\gamma}}(x)=0$ follows.
\end{proof}

%======================
\section{Variation of K\"{a}hler quotients}\label{var-quo}

Henceforth we will make the following assumptions (see Figure \ref{fig-1}).
\begin{itemize}
  \item [(\textbf{A1})] $\Delta_{-}$ and $\Delta_{+}$ are two subpolytopes of the moment body $\Delta$
                        separated by a wall $F$.
  \item [(\textbf{A2})] $l$ is a line segment with end-points $\xi$ and $\zeta$ which
                        lie in the interiors $\mathrm{int}\,\Delta_{-}$ and $\mathrm{int}\,\Delta_{+}$ respectively.
  \item [(\textbf{A3})] $l$ intersects the wall at a relative interior point $\epsilon$.
\end{itemize}
\begin{figure}[h]
\centering
\begin{tikzpicture}[scale=0.4]
  \shade[ball color=black] (3,8) circle (0.1);
    \shade[ball color=black] (11,8) circle (0.1);
     \shade[ball color=black] (7,8) circle (0.1);
\draw  (1.2,7.2)node[right] {$\xi$};
\draw  (10.8,7.2)node[right] {$\zeta$};
\draw  (5.8,7.2)node[right] {$\epsilon$};
\draw  (3.5,5.2)node[right] {$\Delta_{-}$};
\draw  (8,5.2)node[right] {$\Delta_{+}$};
\draw[-latex,blue] (3,8) --(7,8);
\draw[-latex,blue] (11,8)--(7,8);
\draw(7,2)--(7,12);
\draw (0,7) --(7,12)--(14,7)--(7,2)--(0,7);
\fill[dashed,blue, opacity=0.1] (0,7) --(7,12)--(14,7)--(7,2)--(0,7);
\end{tikzpicture}
\caption{}
\label{fig-1}
\end{figure}
Because of Lemma \ref{ep-in}, the equality
$
X^{s}_{\epsilon}=X^{s}_{\xi}\cap X^{s}_{\zeta}
$
is valid.
Particularly, from Theorem \ref{m1}, there exist two natural proper modifications
\begin{equation*}
f_{\xi, \epsilon}:X^{s}_{\xi}/\!\!/G\longrightarrow X^{ss}_{\epsilon}/\!\!/G
\end{equation*}
and
\begin{equation*}
f_{\zeta, \epsilon}:X^{s}_{\zeta}/\!\!/G\longrightarrow X^{ss}_{\epsilon}/\!\!/G.
\end{equation*}
To describe $f_{\xi, \epsilon}$ and $f_{\zeta, \epsilon}$ more precisely, we consider their fibers over $B$ as reduced closed complex subspaces of $X^{s}_{\xi}/\!\!/G$ and $X^{s}_{\zeta}/\!\!/G$.
Note that $X^{s}_{\xi}/\!\!/G$, $X^{s}_{\zeta}/\!\!/G$ and $X^{ss}_{\epsilon}/\!\!/G$ are compact, connected, reduced normal complex spaces.
Combining the Stein factorization of $f_{\xi, \epsilon}$ (resp. $f_{\zeta, \epsilon}$) with Zariski's Main Theorem derives the following result, see \cite[Corollary 1.12]{Ue75}.
\begin{prope}\label{fiber-conn}
All fibers of $f_{\xi, \epsilon}$ (resp. $f_{\zeta, \epsilon}$) are connected.
\end{prope}
We divide this section into two parts: first we consider the case of $\mathbb{C}^{\ast}$-actions, and then the general torus actions.

\subsection{$\mathbb{C}^{\ast}$-actions}
Assume that $G=\mathbb{C}^{\ast}$, with the same setting as in Theorem \ref{m1}, then we have the following:
\begin{thm}\label{c-action}
For each $b\in B$ the fiber  $f_{\xi, \epsilon}^{-1}(b)$ (resp. $f_{\zeta, \epsilon}^{-1}(b))$ is biholomorphic to a weighted projective space satisfying:
\begin{equation}\label{dim-fiber}
\mathrm{dim}_{\mathbb{C}}\,f_{\xi, \epsilon}^{-1}(b)+
\mathrm{dim}_{\mathbb{C}}\,f_{\zeta, \epsilon}^{-1}(b)=
\mathrm{codim}_{\mathbb{C}}\,(B_{i}, X^{ss}_{\epsilon}/\!\!/\mathbb{C}^{\ast})-1,
\end{equation}
where $B_{i}$ is the connected component  of $B$ in which $b$ lies.
\end{thm}
\begin{proof}
Observe that the critical set of the moment map $\Phi: X\rightarrow\mathrm{Lie}^{\ast}(\mathbb{S}^{1})\cong\mathbb{R}^{1}$ is precisely the fixed point set $X^{\mathbb{S}^{1}}=X^{\mathbb{C}^{\ast}}$.
It follows from the construction of $f_{\zeta, \epsilon}$ that $B$ is identical with the union of the connected components of $X^{\mathbb{C}^{\ast}}$ which lie in $\Phi^{-1}(\epsilon)$.
Let $B_{1}, B_{2}, \cdots, B_{k}$ be the connected components of $X^{\mathbb{C}^{\ast}}$ and put
$$
V^{+}_{i}:=\{x\in X\,|\,\lim_{\lambda\rightarrow0}\lambda\cdot x\in B_{i}\},\quad\quad
V^{-}_{i}:=\{x\in X\,|\,\lim_{\lambda\rightarrow\infty}\lambda\cdot x\in B_{i}\},
$$
where $\lambda\in\mathbb{C}^{\ast}$.
Denote by $\pi^{+}_{i}$ (resp. $\pi^{-}_{i}$) the map from $V^{+}_{i}$ (resp. $V^{-}_{i}$) to $B_{i}$ by setting $\pi^{+}_{i}(x)=\lim_{\lambda\rightarrow0}\lambda\cdot x$ (resp. $\pi^{-}_{i}(x)=\lim_{\lambda\rightarrow\infty}\lambda\cdot x$).
According to a result by Carrell--Sommese \cite[Proposition II]{CS78}, the maps $\pi^{+}_{i}$ and $\pi^{-}_{i}$ are $\mathbb{C}^{\ast}$-invariant maximal rank holomorphic surjections with vector space fibers; moreover, we have the plus and the minus decompositions of locally closed analytic submanifolds:
\begin{equation}\label{plus-minus-decom}
X=\bigsqcup_{i=1}^{k}V^{+}_{i}=\bigsqcup_{i=1}^{k}V^{-}_{i}.
\end{equation}

Suppose $b$ lies in a connected component $B_{i}\subset\Phi^{-1}(\epsilon)$, consider the commutative diagram:
\begin{equation*}
\vcenter{
\xymatrix{
  X^{s}_{\xi} \ar[d]_{\Pi_{\xi}} \ar[r]^{i_{\xi, \epsilon}} &
 X^{ss}_{\epsilon} \ar[d]_{\Pi_{\epsilon}}
  &X^{s}_{\zeta}\ar[l]_{i_{\zeta, \epsilon}}\ar[d]_{\Pi_{\zeta}} \\
  X^{s}_{\xi} /\!\!/\mathbb{C}^{\ast} \ar[r]^{f_{\xi, \epsilon}} & X^{ss}_{\epsilon}/\!\!/\mathbb{C}^{\ast}  & X^{s}_{\zeta}/\!\!/\mathbb{C}^{\ast} \ar[l]_{f_{\zeta, \epsilon}}  }
  }
\end{equation*}
From definition, the fibers $f_{\xi, \epsilon}^{-1}(b)$ and $f_{\zeta, \epsilon}^{-1}(b)$ are equal to the respective quotients of the spaces
$$
V_{b}^{\xi}:=\{x\in X^{s}_{\xi}\,|\, b\in\overline{\mathbb{C}^{\ast}\cdot x} \},\quad
V_{b}^{\zeta}:=\{x\in X^{s}_{\zeta}\,|\, b\in\overline{\mathbb{C}^{\ast}\cdot x} \}
$$
by the $\mathbb{C}^{\ast}$-action, where $\overline{\mathbb{C}^{\ast}\cdot x}$ is the closure of the orbit $\mathbb{C}^{\ast}\cdot x$ in $X$.
Put $W^{+}_{b}=(\pi^{+}_{i})^{-1}(b)$ and $W^{-}_{b}=(\pi^{-}_{i})^{-1}(b)$.
Note that $b\in\overline{\mathbb{C}^{\ast}\cdot x}$ if and only if $\lim_{\lambda\rightarrow0}\lambda\cdot x=b$ or $\lim_{\lambda\rightarrow\infty}\lambda\cdot x=b$, i.e., $x\in W^{+}_{b}\setminus\{b\}$ or $x\in W^{-}_{b}\setminus\{b\}$.
Observe that all fibers of $f_{\xi,\epsilon}$ and $f_{\zeta,\epsilon}$ are connected.
Since $W^{+}_{b}\cap W^{-}_{b}=\{b\}$, without loss of generality, we may assume that $V^{\xi}_{b}$ equals $X^{s}_{\xi}\cap W^{+}_{b}$ and $V^{\zeta}_{b}$ equals $X^{s}_{\zeta}\cap W^{-}_{b}$.
Consequently, $V^{\xi}_{b}$ and $V^{\zeta}_{b}$ are $\mathbb{C}^{\ast}$-invariant open subsets of $W^{+}_{b}\setminus\{b\}$ and $W^{-}_{b}\setminus\{b\}$, respectively.

Observe that $W^{+}_{b}$ is a complex vector space on which $\mathbb{C}^{\ast}$ acts linearly, and $b$ is the unique fixed point of the $\mathbb{C}^{\ast}$-action.
This implies that the quotient space $(W^{+}_{b}\setminus\{b\})/\mathbb{C}^{\ast}$ is a weighted projective space, and $f_{\xi, \epsilon}^{-1}(b)=V^{\xi}_{b}/\mathbb{C}^{\ast}$ is a non-empty open subset in $(W^{+}_{b}\setminus\{b\})/\mathbb{C}^{\ast}$.
Note that $f_{\xi, \epsilon}^{-1}(b)$ is compact in $X^{s}_{\xi}/\!\!/\mathbb{C}^{\ast}$, and so is in $(W^{+}_{b}\setminus\{b\})/\mathbb{C}^{\ast}$.
It follows that $f_{\xi, \epsilon}^{-1}(b)$ is a closed subset of $(W^{+}_{b}\setminus\{b\})/\mathbb{C}^{\ast}$.
Consequently, we are led to the conclusion that $f_{\xi, \epsilon}^{-1}(b)$ is open and closed in $(W^{+}_{b}\setminus\{b\})/\mathbb{C}^{\ast}$, which means
$$
f_{\xi, \epsilon}^{-1}(b)=(W^{+}_{b}\setminus\{b\})/\mathbb{C}^{\ast}.
$$
By the same argument, we can show that $f_{\zeta, \epsilon}^{-1}(b)$ is also biholomorphic to the weighted projective space $(W^{-}_{b}\setminus\{b\})/\mathbb{C}^{\ast}$.
According to \cite[Proposition II]{CS78}, we have:
$$
\mathrm{dim}_{\mathbb{C}}\,W^{+}_{b}+
\mathrm{dim}_{\mathbb{C}}\,W^{-}_{b}=
\mathrm{codim}_{\mathbb{C}}\,(B_{i}, X).
$$
Note that $B_{i}$ can be considered as a subspace of $X^{ss}_{\epsilon}/\!\!/\mathbb{C}^{\ast}$ under the quotient map, and we have
$$
\mathrm{codim}_{\mathbb{C}}(B_{i}, X^{ss}_{\epsilon}/\!\!/\mathbb{C}^{\ast})=
\mathrm{codim}_{\mathbb{C}}(B_{i}, X)-1.
$$
As a result, the equality \eqref{dim-fiber} is valid and this completes the proof.
\end{proof}

%=====
\subsection{($\mathbb{C}^{\ast})^{d}$-actions}
Consider the $(\mathbb{C}^{\ast})^{d}$-action with $d>1$.
For every $x\in X$, the \emph{dimension} of the convex polytope $\Phi(\overline{G\cdot x})\subset\mathfrak{t}^{\ast}\cong\mathbb{R}^{d}$ is defined to be the minimal dimension of the affine subspace of $\mathfrak{t}^{\ast}\cong\mathbb{R}^{d}$ which contains $\Phi(\overline{G\cdot x})$.
We consider the fibers of $f_{\xi, \epsilon}$ (resp. $f_{\xi, \epsilon}$) over $B$.
Note that $\epsilon$ lies in the relative interior of the wall $F$, and for any $\bar{b}\in B$, there exists a unique $b\in\Phi^{-1}(\epsilon)$ with $\mathrm{dim}\,T_{b}>0$, such that the orbit $G\cdot b$ is closed in $X^{ss}_{\epsilon}$ and $\bar{b}=\Pi_{\epsilon}(G\cdot b)$.
Denote by $[G\cdot b]=\Pi_{\epsilon}(G\cdot b)$.

\begin{prop}\label{G-b}
Let $b\in\Phi^{-1}(\epsilon)$ with $\mathrm{dim}\,T_{b}>0$ satisfying $[G\cdot b]\in B$, then we have $\mathrm{dim}_{\mathbb{C}}\,G_{b}=1$.
\end{prop}
\begin{proof}
Recall that $G$ is the complexification of the $d$-torus $T=(\mathbb{S}^{1})^{d}$.
According to \cite[Proposition 1.6]{Sj95}, the complex stabilizer of $b$ is equal to the complexification of the compact stabilizer $T_{b}$, i.e., $G_{b}=(T_{b})^{\mathbb{C}}$.
To end the proof, it suffices to show $\mathrm{dim}\,T_{b}=1$.
Consider the moment image of the orbit closure $\overline{G\cdot b}$, which is a convex
polytopy in $\mathfrak{t}^{\ast}$.
By definition, we have
$$
\Phi(\overline{G\cdot b})=\{\lambda\in\mathfrak{t}^{\ast}\,|\,
b\in X^{sss}(\Phi_{\lambda})\}.
$$
Observe that $\epsilon\in\Phi(\overline{G\cdot b})$ since $\Phi(b)=\epsilon$.
Due to Lemma \ref{sta}, we have $X^{sss}(\Phi_{\epsilon})=X^{sss}(\Phi_{\lambda})$
for any $\lambda\in\mathrm{int}\,F$.
This implies $b\in X^{sss}(\Phi_{\lambda})$ and hence
$\lambda\in\Phi(\overline{G\cdot b})$, i.e., $\mathrm{int}\,F$ is an open subset of
$\Phi(\overline{G\cdot b})$.
Consequently, the dimension of the convex polytope $\Phi(\overline{G\cdot b})$ is equal to
$\mathrm{dim}\,F=d-1$.
On account of \cite[Theorem 2, (c)]{At82}, we have
$\mathrm{dim}\,\Phi(\overline{G\cdot b})=\mathrm{dim}\,(T\cdot b)$, and therefore we are led to the conclusion $\mathrm{dim}\,T_{b}=1$.
\end{proof}

\begin{rem}\label{rk:tf}
  Two comments on Proposition~\ref{G-b}.

  \begin{enumerate}[leftmargin=0pt, itemindent=3em, labelsep=0.5em]
    \item Note that $B$ itself is singular in general.
    However, the proposition above shows that each connected component of $B$ is a closed complex subspaces of $X^{ss}_{\epsilon}/\!\!/G$ which has finite quotient singularities at most.

    \item By Proposition \ref{G-b}, in the language of \cite{DH98}, we can say that the the relative interior of a wall $F$ is a truly faithful cell.
    In the projective case, \cite[Corollary~4.1.10]{DH98} shows a closely related result that all codimension $1$ wall (in their sense) are truly faithful for the torus action.
    However, there are two problems to deduce Proposition \ref{G-b} from this result.
    First, the parameter space in \cite{DH98} is larger than $\mathfrak{t}^*$ here.
    As a result, it seems not clear why the relative interior of $F$ is always contained in a codimension $1$ wall.
    Even so, since not all cells contained in a wall (in the sense of \cite{DH98}) are necessarily truly faithful, we still do not know whether $F$ is a subset of a truly faithful cell.
    Therefore, we think that Proposition \ref{G-b} may be also useful for the projective case.\footnote{Note that in~\cite[Example~3.3.23]{DH98}, the authors discuss a special case, $X= \mathbb{CP}^n$, $G = (\mathbb{C}^*)^n$. In this case, the large parameter space is a cone over the moment body and Proposition~\ref{G-b} does follow from~\cite[Corollary~4.1.10]{DH98}.}
  \end{enumerate}
\end{rem}

For each $[G\cdot b]\in B$ with $b\in\Phi^{-1}(\epsilon)$, it follows from Proposition \ref{G-b} that the identity component of the stabilizer $G_{b}$, denoted by $G^{0}_{b}$, is isomorphic to $\mathbb{C}^{\ast}$.
Without loss of generality, we may assume that $G^{0}_{b}=\mathbb{C}^{\ast}$ is the first component of $G=(\mathbb{C}^{\ast})^{d}$ and $G_{b}=\mathbb{C}^{\ast}\times\Gamma$,
where $\Gamma$ is a finite subgroup of $(\mathbb{C})^{d-1}$.
Let $\{F_{i}\,|\,1\leq i\leq r\}$ be the connected components of the fixed point set of the $G^{0}_{b}$-action on $X$.
Then we have the plus and the minus decompositions \eqref{plus-minus-decom} corresponding to the $G^{0}_{b}$-action.
Define the respective subspaces $V_{b}^{\xi}\subset X^{s}_{\xi}$ and $V_{b}^{\zeta}\subset X^{s}_{\zeta}$ by setting:
$$
V_{b}^{\xi}=\{x\in X^{s}_{\xi}\,|\, G\cdot b\subset\overline{G\cdot x} \},\quad
V_{b}^{\zeta}=\{x\in X^{s}_{\zeta}\,|\, G\cdot b\subset\overline{G\cdot x} \},
$$
where $\overline{G\cdot x}$ is the closure of the orbit $G\cdot x$ in $X$.
Then we have the following results.
\begin{prop}\label{v=gw}
With the same notions as above, we have $V_{b}^{\xi}=G(X^{s}_{\xi}\cap W^{+}_{b})$ and
$V_{b}^{\zeta}=G(X^{s}_{\zeta}\cap W^{-}_{b})$.
\end{prop}
\begin{proof}
We shall prove only the identity $V_{b}^{\xi}=G(X^{s}_{\xi}\cap W^{+}_{b})$ and the argument for $V_{b}^{\zeta}=G(X^{s}_{\zeta}\cap W^{-}_{b})$ is similar.
For any $x\in X^{s}_{\xi}\cap W^{+}_{b}$, we get $\lim_{\lambda\rightarrow0}\lambda\cdot x=b$ and hence $b\in\overline{G^{0}_{b}\cdot x}$,
where $\lambda\in G^{0}_{b}$ and $\overline{G^{0}_{b}\cdot x}$ is the closure of the orbit $G^{0}_{b}\cdot x$ in $X$.
Consequently, we obtain
$$
G\cdot b\subset G\cdot\overline{(G^{0}_{b}\cdot x)}\subset \overline{G\cdot (G^{0}_{b}\cdot x)}.
$$
Since $G\cdot (G^{0}_{b}\cdot x)=G\cdot x$, we get $G\cdot b\subset\overline{G\cdot x}$ and thus $x\in V^{\xi}_{b}$.
It follows that $X^{s}_{\xi}\cap W^{+}_{b}$ is a subspace of $V_{b}^{\xi}$.

We claim that $G\cdot(X^{s}_{\xi}\cap W^{+}_{b})$ belongs to $V_{b}^{\xi}$.
Let $y$ be an arbitrary point in $X^{s}_{\xi}\cap W^{+}_{b}$.
For any $g\in G$, note that
$$
\lim_{\lambda\rightarrow0}\lambda\cdot (g\cdot y)=
g(\lim_{\lambda\rightarrow0}\lambda\cdot  y)=
g\cdot b.
$$
This implies $g\cdot b\in\overline{G^{0}_{b}\cdot(g\cdot y)}$ and therefore we get
$$
G\cdot b\subset G\cdot\overline{G^{0}_{b}\cdot(g\cdot y)}
\subset \overline{G\cdot G^{0}_{b}\cdot(g\cdot y)}=
\overline{G\cdot(g\cdot y)}
$$
which means that $g\cdot y$ lies in $V^{\xi}_{b}$, i.e., $G(X^{s}_{\xi}\cap W^{+}_{b})\subset V^{\xi}_{b}$.

It remains to verify $V^{\xi}_{b}\subset G(X^{s}_{\xi}\cap W^{+}_{b})$.
For any $z\in V^{\xi}_{b}\setminus(X^{s}_{\xi}\cap W^{+}_{b})$, we have
$G\cdot b\subset\overline{G\cdot z}$.
Without loss of generality, we may identity $G^{0}_{b}$ with the first component of $G=(\mathbb{C}^{\ast})^{d}$.
Set $G_{1}=\mathbb{C}^{\ast}\times\{1\}$ and $G_{2}=\{1\}\times(\mathbb{C}^{\ast})^{d-1}$.
By the commutativity of $G$, we get
$$
G\cdot b=G_{1}\cdot G_{2}\cdot b=G_{2}\cdot G_{1}\cdot b=G_{2}\cdot b,
$$
and hence $G\cdot b\subset\overline{G\cdot z}$ is equal to
\begin{equation}\label{g2b}
G_{2}\cdot b\subset\overline{G_{1}\cdot(G_{2}\cdot z)}.
\end{equation}
Note that the point $z$ is $\Phi_{\xi}$-stable.
Due to \cite[Lemma 5.2]{Wa21}, we get that $z$ is stable for the $G_{2}$-action with respect to the restricted moment map, and thus the orbit $G_{2}\cdot z$ is closed in $X$.
Moreover, for every $g\in G$, the orbit $G_{2}\cdot (g\cdot z)$ is closed in $X$ because the point $g\cdot z$ is stable for the $G_{2}$-action.
It follows from \eqref{g2b} that $b$ belongs to $\overline{G_{1}\cdot(k\cdot z)}$ for some $k\in G_{2}$.
As a result, there is an element $k\in G_{2}$ such that $b\in\overline{G_{1}\cdot(k\cdot z)}$ which means $k\cdot z\in X^{s}_{\xi}\cap W^{+}_{b}$.
Put $w=k\cdot z$.
Then $z=k^{-1}\cdot w$ and hence $V^{\xi}_{b}\subset G(X^{s}_{\xi}\cap W^{+}_{b})$.
\end{proof}

\begin{prop}\label{iso-wb}
The subgroup of $G$ which preserves $W^{+}_{b}=(\pi^{+}_{i})^{-1}(b)$ (resp. $W^{-}_{b}=(\pi^{-}_{i})^{-1}(b)$) is precisely the stabilizer $G_{b}$.
\end{prop}
\begin{proof}
Recall that
$$
W^{+}_{b}=\{x\in X\,|\,\lim_{\lambda\rightarrow0}\lambda\cdot x=b\},\quad
W^{-}_{b}=\{x\in X\,|\,\lim_{\lambda\rightarrow\infty}\lambda\cdot x=b\},
$$
where $\lambda\in G^{0}_{b}$.
We only prove the conclusion for $W^{+}_{b}$, and $W^{-}_{b}$ is treated similarly.
Let $x$ be an arbitrary point in $W^{+}_{b}$.
On the one hand, for any $g\in G_{b}$, we have
$$
\lim_{\lambda\rightarrow0}\lambda\cdot(g\cdot x)=
\lim_{\lambda\rightarrow0}g\cdot(\lambda\cdot x)=
g\cdot(\lim_{\lambda\rightarrow0}\lambda\cdot x)=g\cdot b=b.
$$
This implies $g\cdot x\in W^{+}_{b}$ and therefore $G_{b}$ preserves $W^{+}_{b}$.
On the other hand, given an element $g^{\prime}\in G$ such that $g^{\prime}\cdot W^{+}_{b}\subset W^{+}_{b}$.
Then for any $x\in W^{+}_{b}$, the identity $\lim_{\lambda\rightarrow0}\lambda\cdot(g^{\prime}\cdot x)=b$ is valid.
Because of the fact $\lim_{\lambda\rightarrow0}\lambda\cdot(g^{\prime}\cdot x)=g^{\prime}\cdot b$, we get $g^{\prime}\cdot b=b$ and hence $g^{\prime}\in G_{b}$.
\end{proof}

We are ready to prove the main result of this subsection.
\begin{thm}\label{m2}
For each $[G\cdot b]\in B$ with $b\in\Phi^{-1}(\epsilon)$, let $G^{0}_{b}$ be the identity component of the stabilizer $G_{b}$ and $F$ the fixed point set by the $G^{0}_{b}$-action on $X$.
Then we have:
\begin{itemize}
  \item [(i)] the fiber  $f_{\xi, \epsilon}^{-1}([G\cdot b])$ (resp. $f_{\zeta, \epsilon}^{-1}([G\cdot b])$) is biholomorphic to the quotient of a weighted projective spaces by the finite group $\Gamma_{b}=G_{b}/G^{0}_{b}$;
  \item [(ii)] the dimensions of the fibers satisfies
        $$
        \mathrm{dim}_{\mathbb{C}}\,f_{\zeta, \epsilon}^{-1}([G\cdot b])+
        \mathrm{dim}_{\mathbb{C}}\,f_{\xi, \epsilon}^{-1}([G\cdot b])=
        \mathrm{codim}_{\mathbb{C}}\,(B_{i}, X^{ss}_{\epsilon}/\!\!/G)-1,
         $$
         where $B_{i}$ is the stratum of $B$ in which $[G\cdot b]$ lies.
\end{itemize}
\end{thm}

\begin{proof}
Note that the fiber $f_{\xi, \epsilon}^{-1}([G\cdot b])$ is a compact connected complex analytic subvariety of $X^{s}_{\xi}/\!\!/G$.
Recall that $V^{\xi}_{b}=\{x\in X^{s}_{\xi}\,|\, G\cdot b\subset\overline{G\cdot x} \}$.
By definition, we have
\begin{equation}\label{f-xi-0}
f_{\xi, \epsilon}^{-1}([G\cdot b])=V^{\xi}_{b}/G.
\end{equation}
According to Proposition \ref{v=gw} and Proposition \ref{iso-wb}, we obtain
\begin{equation}\label{f-xi-1}
V^{\xi}_{b}/G\cong (X^{s}_{\xi}\cap W^{+}_{b})/G_{b}
\cong\bigl((X^{s}_{\xi}\cap W^{+}_{b})/G^{0}_{b}\bigr)/\Gamma_{b}.
\end{equation}
We claim that $(X^{s}_{\xi}\cap W^{+}_{b})/G^{0}_{b}$ is a weighted projective space.
We may identify $W^{+}_{b}$ with $\mathbb{C}^{m+1}$ and $G^{0}_{b}=\mathbb{C}^{\ast}$ acts linearly on $\mathbb{C}^{m+1}$ with the weight $\textbf{a}=(a_{0},\cdots,a_{m})$ which has the unique fixed point $0$.
It follows from the plus and the minus decompositions for the $G^{0}_{b}$-action on $X$ that the quotient $(W^{+}_{b}\setminus\{b\})/G^{0}_{b}$ is a weighted projective space denoted by $\widetilde{\mathbb{P}}(\textbf{a})$.
Since $f_{\xi,\epsilon}$ is surjective, the space
$\bigl((X^{s}_{\xi}\cap W^{+}_{b})/G^{0}_{b}\bigr)/\Gamma_{b}$
is nonempty.
On the one hand, due to the openness of $X^{s}_{\xi}$ in $X$, the set
$X^{s}_{\xi}\cap W^{+}_{b}
=X^{s}_{\xi}\cap (W^{+}_{b}\setminus\{b\})$
is an open subset of $W^{+}_{b}\setminus\{b\}$,
and thus
$$
(X^{s}_{\xi}\cap W^{+}_{b})/G^{0}_{b}
=\bigl(X^{s}_{\xi}\cap (W^{+}_{b}\setminus\{b\})\bigr)/G^{0}_{b}
$$
is an open subset of $\widetilde{\mathbb{P}}(\textbf{a})$.
Consequently, the quotient space $\bigl((X^{s}_{\xi}\cap W^{+}_{b})/G^{0}_{b}\bigr)/\Gamma_{b}$ is open in the quotient $\widetilde{\mathbb{P}}(\textbf{a})/\Gamma_{b}$.
On the other hand, as a fiber of a proper map, the quotient space
$\bigl((X^{s}_{\xi}\cap W^{+}_{b})/G^{0}_{b}\bigr)/\Gamma_{b}$
is compact in $X^{s}_{\xi}/G$ and so is in $\widetilde{\mathbb{P}}(\textbf{a})/\Gamma_{b}$.
It follows that
$\bigl((X^{s}_{\xi}\cap W^{+}_{b})/G^{0}_{b}\bigr)/\Gamma_{b}$
is closed in $\widetilde{\mathbb{P}}(\textbf{a})/\Gamma_{b}$. 
As a result, we get
\begin{equation}\label{f-xi-2}
\bigl((X^{s}_{\xi}\cap W^{+}_{b})/G^{0}_{b}\bigr)/\Gamma_{b}=
\widetilde{\mathbb{P}}(\textbf{a})/\Gamma_{b}.
\end{equation}
From \eqref{f-xi-0}-\eqref{f-xi-2}, we are led to the conclusion that the fiber $f_{\xi, \epsilon}^{-1}([G\cdot b])$ is biholomorphic to $\widetilde{\mathbb{P}}(\textbf{a})/\Gamma_{b}$.
Using the same argument, we can prove the conclusion for the fiber $f_{\zeta, \epsilon}^{-1}([G\cdot b])$.

We now turn to the conclusion (ii).
Denote by $F_{i}$ the connected component of the fixed point set $F$ by the $G^{0}_{b}$-action on $X$, which is a closed complex submanifold of $X$.
Note that $G^{0}_{b}\cong\mathbb{C}^{\ast}$ and $T^{0}_{b}\cong\mathbb{S}^{1}$.
Set $B_{i}=\Pi_{\epsilon}(X^{ss}_{\epsilon}\cap F_{i})$ which is a stratum of $B$.
Due to \cite[Theorem 2.10]{Sj95}, we get
$$
\mathrm{dim}_{\mathbb{C}}\,B_{i}=\mathrm{dim}_{\mathbb{C}}\,F_{i}-(d-1).
$$
On the one hand, owing to \cite[Proposition II]{CS78}, we have
$$
\mathrm{dim}_{\mathbb{C}}\,W^{+}_{b}+\mathrm{dim}_{\mathbb{C}}\,W^{-}_{b}
=\mathrm{codim}_{\mathbb{C}}\,(F_{i}, X).
$$
On the other hand, the result of (i) shows
$
\mathrm{dim}_{\mathbb{C}}\,f_{\xi, \epsilon}^{-1}([G\cdot b])=
\mathrm{dim}_{\mathbb{C}}\,W^{+}_{b}-1
$
and
$
\mathrm{dim}_{\mathbb{C}}\,f_{\zeta, \epsilon}^{-1}([G\cdot b])=
\mathrm{dim}_{\mathbb{C}}\,W^{-}_{b}-1.
$
As a direct result, we obtain
\begin{eqnarray*}
  \mathrm{dim}_{\mathbb{C}}\,f_{\zeta, \epsilon}^{-1}([G\cdot b])+
\mathrm{dim}_{\mathbb{C}}\,f_{\xi, \epsilon}^{-1}([G\cdot b])
&=& \mathrm{codim}_{\mathbb{C}}\,(F_{i}, X)-2 \\
&=& \mathrm{codim}_{\mathbb{C}}\,(B_{i}, X^{ss}_{\epsilon}/\!\!/G)-1
\end{eqnarray*}
and this completes the proof.
\end{proof}

Recall that a complex analytic space $Y$ is \emph{$\mathbb{Q}$-factorial} if every Weil divisor on $Y$ is $\mathbb{Q}$-Cartier.
Since every complex analytic spaces with finite quotient singularities only are $\mathbb{Q}$-factorial (cf. \cite[Proposition 5.15]{KM98}),
the regular  K\"{a}hler quotients $X^{s}_{\xi}/\!\!/G$ and $X^{s}_{\zeta}/\!\!/G$, as compact K\"{a}hler orbifolds, are $\mathbb{Q}$-factorial necessarily.
As corollaries of Theorem \ref{m2}, we get
\begin{cor}\label{divisorial}
Suppose the center $B$ has only one stratum,
if $f_{\xi,\epsilon}$ (resp. $f_{\zeta,\epsilon}$) is divisorial, then $f_{\zeta,\epsilon}$ (resp. $f_{\xi,\epsilon}$) is a biholomorphic map.
\end{cor}
\begin{proof}
Put $A_{\xi}=f^{-1}_{\xi,\epsilon}(B)$ and $A_{\zeta}=f^{-1}_{\zeta,\epsilon}(B)$.
If $f_{\xi,\epsilon}$ is divisorial, then it follows from the definition that $A_{\xi}$ is an irreducible and reduced complex subspace of codimension 1, i.e., $\mathrm{dim}_{\mathbb{C}}\,A_{\xi}=(n-d)-1$, see Appendix \ref{mod-bim}.
By assumption, the center $B$ has only one stratum and thus $B$ is a connected complex manifold.
Due to Theorem \ref{m2}, for each point $[G\cdot b]$ in $B$, we have
\begin{eqnarray*}
  \mathrm{dim}_{\mathbb{C}}\,A_{\xi}
  &=&
  \mathrm{dim}_{\mathbb{C}}\,f_{\xi, \epsilon}^{-1}([G\cdot b])+\mathrm{dim}_{\mathbb{C}}\,B \\
  &=&
  \mathrm{dim}_{\mathbb{C}}\,X^{ss}_{\epsilon}/\!\!/G
-\mathrm{dim}_{\mathbb{C}}\,f_{\zeta, \epsilon}^{-1}([G\cdot b])-1\\
&=&(n-d)-\mathrm{dim}_{\mathbb{C}}\,f_{\zeta, \epsilon}^{-1}([G\cdot b])-1.
\end{eqnarray*}
This implies the fiber $f_{\zeta, \epsilon}^{-1}([G\cdot b])$ has dimension zero.
Because $f_{\zeta, \epsilon}^{-1}([G\cdot b])$ is connected, the restriction of $f_{\zeta,\epsilon}$ to $A_{\zeta}$ is injective.
As a result, $f_{\zeta,\epsilon}$ is a finite modification of normal complex spaces which is a biholomorphic map due to Zariski's Main Theorem \cite[Theorem 1.11]{Ue75}.
\end{proof}
Particularly, when $G=\mathbb{C}^{\ast}$ the center $B$ is connected necessarily provided that $A_{\xi}$ is a prime divisor and therefore $B$ has only one stratum.
\begin{cor}\label{small}
The modification $f_{\xi,\epsilon}$ (resp. $f_{\zeta,\epsilon}$) is small if and only if each fiber of $f_{\zeta,\epsilon}$ (resp. $f_{\xi,\epsilon}$) over $B$ is of positive dimension.
\end{cor}
\begin{proof}
Via refining the decomposition, we may assume that each stratum $B_{i}$ of $B$ is connected.
By definition,  $f_{\xi,\epsilon}$ is small, if and only if for every $B_{i}$ the inequality
\begin{equation}\label{sml}
\mathrm{dim}_{\mathbb{C}}\,f^{-1}_{\xi,\epsilon}([G\cdot b])+\mathrm{dim}_{\mathbb{C}}\,B_{i}
\leq(n-d)-2
\end{equation}
is valid, where $[G\cdot b]$ is an arbitrary point in $B_{i}$.
By Theorem \ref{m2}, we have
$$
\mathrm{dim}_{\mathbb{C}}\,f^{-1}_{\xi,\epsilon}([G\cdot b])+\mathrm{dim}_{\mathbb{C}}\,B_{i}=
(n-d)-\mathrm{dim}_{\mathbb{C}}\,f^{-1}_{\zeta,\epsilon}([G\cdot b])-1.
$$
It follows that \eqref{sml} holds if and only if the fiber $f^{-1}_{\zeta,\epsilon}([G\cdot b])$ has dimension $\geq1$ and this completes the proof.
\end{proof}

%=============================
\section{Wall-crossing for quasi-free actions}\label{bim-equ}

Under the assumptions (\textbf{A1})-(\textbf{A3}),
from Theorem \ref{main-thm-2}, we get two natural proper modifications
$f_{\xi,\epsilon}:X^{s}_{\xi}/\!\!/G\rightarrow X^{ss}_{\epsilon}/\!\!/G$ and $f_{\zeta,\epsilon}:X^{s}_{\zeta}/\!\!/G\rightarrow X^{ss}_{\epsilon}/\!\!/G$ with the common center $B\subset X^{ss}_{\epsilon}/\!\!/G$.
For simplicity, we write $X_{\xi}=X^{s}_{\xi}/\!\!/G$, $X_{\epsilon}=X^{ss}_{\epsilon}/\!\!/G$ and $X_{\zeta}=X^{s}_{\zeta}/\!\!/G$.
Put $A_{\xi}=f^{-1}_{\xi,\epsilon}(B)$ and  $A_{\zeta}=f^{-1}_{\zeta,\epsilon}(B)$ which are closed complex analytic subsets of $X_{\xi}$ and $X_{\zeta}$ respectively.
Observe that both
$$
f_{\xi,\epsilon}|_{X_{\xi}\setminus A_{\xi}}:X_{\xi}\setminus A_{\xi}\longrightarrow X_{\epsilon}\setminus B
$$
and
$$
f_{\zeta,\epsilon}|_{X_{\zeta}\setminus A_{\zeta}}:X_{\zeta}\setminus A_{\zeta}\longrightarrow X_{\epsilon}\setminus B
$$
are biholomorphic maps.
The composition of $f_{\xi,\epsilon}|_{X_{\xi}\setminus A_{\xi}}$ and $(f_{\zeta,\epsilon}|_{X_{\zeta}\setminus A_{\zeta}})^{-1}$ gives rise to a biholomorphic map
$$
\tilde{f}:X_{\xi}\setminus A_{\xi}\longrightarrow X_{\zeta}\setminus A_{\zeta}.
$$
The graph $\mathrm{Graph}(\tilde{f})$ is a closed complex analytic subset of the product space
$(X_{\xi}\setminus A_{\xi})\times (X_{\zeta}\setminus A_{\zeta})$.
Moreover, $\mathrm{Graph}(\tilde{f})$ is reduced, normal, and irreducible, since it is biholomorphic to $X_{\xi}\setminus A_{\xi}$.
Observe that
$$
\mathrm{Graph}(\tilde{f})=
\bigl(X_{\xi}\times_{X_{\epsilon}} X_{\zeta}\bigr)\setminus\bigl(A_{\xi}\times_{B}A_{\zeta}\bigr).
$$
Let $\Gamma$ be the closure of the graph $\mathrm{Graph}(\tilde{f})$ in $X_{\xi}\times X_{\zeta}$.
Then $\Gamma$ is a complex analytic subset of $X_{\xi}\times X_{\zeta}$, see \cite[Corollary 5.4]{Dem12}.

We claim that $\Gamma$ is irreducible.
Suppose $\Gamma$ is reducible.
Without loss of generality, we may assume that $\Gamma$ has two irreducible components $\Gamma_{1}$ and $\Gamma_{2}$.
Note that $\Gamma=\overline{\mathrm{Graph}(\tilde{f})}=\Gamma_{1}\cup \Gamma_{2}$.
Since $\mathrm{Graph}(\tilde{f})$ is irreducible, we obtain $\mathrm{Graph}(\tilde{f})\subset \Gamma_{1}$ or $\mathrm{Graph}(\tilde{f})\subset \Gamma_{2}$.
Consequently, the closure $\overline{\mathrm{Graph}(\tilde{f})}$ lies in $\Gamma_{1}$ or $\Gamma_{2}$, which contradicts with the assumption $\Gamma=\Gamma_{1}\cup \Gamma_{2}$.
The projections from $X_{\xi}\times X_{\zeta}$ to $X_{\xi}$ and  $X_{\zeta}$ induce two natural holomorphic maps
$p_{\xi}:\Gamma\rightarrow  X_{\xi}$ and $p_{\zeta}:\Gamma\rightarrow  X_{\zeta}$.
Observe that $\Gamma$ is compact and $p_{\xi}$ restricts to a biholomorphism from $\mathrm{Graph}(\tilde{f})$ to $X_{\xi}\setminus A_{\xi}$ which is open and dense in $X_{\xi}$.
It follows that $p_{\xi}$ is surjective and so is the map $p_{\zeta}$.
From definition, $p_{\xi}$ and $p_{\zeta}$ are proper modifications.
As a result, we obtain a natural bimeromorphic map in the sense of Definition \ref{mero}:
\begin{eqnarray}\label{bier-map-f}
    f: X_{\xi}&\dashrightarrow&  X_{\zeta}\\
     v&\longmapsto& p_{\zeta}(p^{-1}_{\xi}(v)). \nonumber
\end{eqnarray}

Recall that the $T$-action on $X$ is \emph{quasi-free}, if for every point $x$ the stabilizer $T_{x}$ is connected, and the action is free on an open dense set, see \cite[\S 2]{GS89}.
Assume that the $T$-action is quasi-free, then by definition the regular K\"{a}hler quotients $X_{\xi}$ and $X_{\zeta}$ are compact K\"{a}hler manifolds and therefore \eqref{bier-map-f} becomes a bimeromorphic map of compact K\"{a}hler manifolds.
The \emph{Strong Factorization Conjecture} in birational geometry \cite[Conjecture 0.2.1]{AKMW02} states that any birational map (resp. bimeromorphic map) between complete nonsingular varieties (resp. compact complex manifolds) can be factored into a succession of blow-ups with smooth centers followed by a succession of blow-downs with smooth centers.
Motivated by this conjecture, we get the following
\begin{thm}\label{qusi-free-wc}
Suppose the $T$-action is quasi-free.
If the center $B$ is a finite set, then the blow-ups of $X_{\xi}$ (resp. $X_{\zeta}$) with the center $A_{\xi}$ (resp. $A_{\zeta}$), and of $X_{\epsilon}$ with the center $B$, are naturally biholomorphic to the irreducible component of the fiber product $X_{\xi}\times_{X_{\epsilon}}X_{\zeta}$ dominating $X_{\epsilon}$.
\end{thm}
\begin{proof}
Consider the proper modification $f_{\xi,\epsilon}:X_{\xi}\rightarrow X_{\epsilon}$.
Since the $T$-action is quasi-free, the singular locus of $X_{\epsilon}$, if nonempty, is equal to $B$, and therefore $X_{\epsilon}\setminus B$ is non-singular.
By assumption, $B$ is a finite set and we put $B=\{x_{1},\cdots,x_{k}\}$.
Due to Proposition \ref{iso-wb}, each fiber $A_{\xi,i}:=f^{-1}_{\xi,\epsilon}(x_{i})$ is biholomorphic to a complex projective space and $A_{\xi}$ is a disjoint union of complex projective spaces of different dimensions.
For every $1\leq i\leq k$, denote by $B_{i}=B\setminus\{x_{i}\}$ and $U_{i}=X_{\epsilon}\setminus B_{i}$.
Then $U_{i}$ is an open neighborhood of $x_{i}$ in $X_{\epsilon}$ such that $B\cap U_{i}$ contains $x_{i}$ only and $U_{i}\setminus\{x_{i}\}$ is non-singular.
Set $V_{i}=f^{-1}_{\xi,\epsilon}(U_{i})$.
The modification $f_{\xi,\epsilon}$ induces an biholomorphism of $V_{i}\setminus f^{-1}_{\xi,\epsilon}(x_{i})$ to $U_{i}\setminus\{x_{i}\}$.
Note that $X_{\xi}$ and $X_{\epsilon}$ are reduced.
Let $g_{i}:\tilde{U}_{i}\rightarrow U_{i}$ be the blow-up of $U_{i}$ with the center $\{x_{i}\}$ and $h_{i}:\tilde{V}_{i}\rightarrow V_{i}$ be the blow-up of $V_{i}$ with the center $A_{\xi, i}$.
From \cite[Theorem 1]{HR64} due to Hironaka--Rossi, $\tilde{U}_{i}$ is non-singular, and $\tilde{V}_{i}$ is biholomorphic to $\tilde{U}_{i}$.
Particularly, via identifying $\tilde{V}_{i}$ with $\tilde{U}_{i}$, the diagram
\begin{equation}\label{tri-i}
\vcenter{
\xymatrix@C=1.5cm{
                & \tilde{U}_{i}\ar[dl]_{h_{i}} \ar[dr]^{g_{i}}             \\
 V_{i}  \ar[rr]^{f_{\xi,\epsilon}|_{V_{i}}} & &     U_{i}        }
 }
\end{equation}
commutes.
Observe that $U_{i}\cap U_{j}\neq\emptyset$ for any $j\neq i$, and there exists a unique biholomorphism
$$
l_{ij}:g^{-1}_{i}(U_{i}\cap U_{j})\stackrel{\simeq}\longrightarrow g^{-1}_{j}(U_{i}\cap U_{j})
$$
such that $g_{i}=g_{j}\circ l_{ij}$ and $h_{i}=h_{j}\circ l_{ij}$ on $g^{-1}_{i}(U_{i}\cap U_{j})$.
Piecing together $\{\tilde{U}_{i}\}_{1\leq i\leq k}$ by means of the biholomorphisms $\{l_{ij}\}$ derives a compact complex manifold $\tilde{X}_{\epsilon}$ and a holomorphic map $g:\tilde{X}_{\epsilon}\rightarrow X_{\epsilon}$.
The commutativity of \eqref{tri-i} implies
$h^{-1}_{i}(V_{i}\cap V_{j})=g^{-1}_{i}(U_{i}\cap U_{j})$ and
$h^{-1}_{j}(V_{i}\cap V_{j})=g^{-1}_{j}(U_{i}\cap U_{j})$.
As a result, $\{h_{i}\}_{1\leq i\leq k}$ induces a holomorphic map $h:\tilde{X}_{\epsilon}\rightarrow X_{\xi}$ satisfying $g=f_{\xi,\epsilon}\circ h$.
Actually, $g:\tilde{X}_{\epsilon}=\mathrm{Bl}_{B}X_{\epsilon}\rightarrow X_{\epsilon}$ is the blow-up of $X_{\epsilon}$ at $B$ and $h:\tilde{X}_{\epsilon}\cong\mathrm{Bl}_{A_{\xi}}X_{\xi}\rightarrow X_{\xi}$ is the blow-up of $X_{\xi}$ at $A_{\xi}$.
Similarly, we can show that the blow-up of $X_{\zeta}$ with the center $A_{\zeta}$ is naturally biholomorphic to $\tilde{X}_{\epsilon}$ and therefore we get a commutative diagram:
\begin{equation*}
\vcenter{
\xymatrix@C=1.5cm{
  \tilde{X}_{\epsilon} \ar[d]_{h}\ar[dr]^{g} \ar[r]^{l} & X_{\zeta} \ar[d]^{f_{\zeta,\epsilon}} \\
  X_{\xi} \ar[r]^{f_{\xi,\epsilon}} & X_{\epsilon}  }
  }
\end{equation*}
where $l:\tilde{X}_{\epsilon}\cong\mathrm{Bl}_{A_{\zeta}}X_{\zeta}\rightarrow X_{\zeta}$ is the blow-up morphism.

It remains to verify that the blow-up of $X_{\xi}$ at $A_{\xi}$ is biholomorphic to the irreducible component of the fiber product $V=X_{\xi}\times_{X_{\epsilon}}X_{\zeta}$ which dominates $X_{\epsilon}$.
Due to the universal property of the fiber product, there exists a unique holomorphic map $\sigma:\tilde{X}_{\epsilon}\rightarrow V$ such that the diagram
\begin{equation}\label{univ-f-b}
\vcenter{
\xymatrix@C=1.5cm{
  \tilde{X}_{\epsilon} \ar@/_/[ddr]_{h} \ar@/^/[drr]^{l}
    \ar@{.>}[dr]|-{\sigma}                   \\
   & V \ar[d]^{\pi_{1}} \ar[r]_{\pi_{2}}
                      & X_{\zeta} \ar[d]_{f_{\zeta,\epsilon}}    \\
   & X_{\xi} \ar[r]^{f_{\xi,\epsilon}}     & X_{\epsilon}               }
}
\end{equation}
commutes.
Put $A_{\zeta,i}=f^{-1}_{\zeta,\epsilon}(x_{i})$.
Since $f_{\xi,\epsilon}$ and $f_{\zeta,\epsilon}$ are proper modifications with the common center $B$,
we get
$$
\pi^{-1}_{1}(A_{\xi})=A_{\xi}\times_{B}A_{\zeta}=\pi^{-1}_{2}(A_{\zeta})
$$
and the isomorphisms
$$
X_{\xi}\setminus A_{\xi}\cong V\setminus(A_{\xi}\times_{B}A_{\zeta})
\cong X_{\zeta}\setminus A_{\zeta}
$$
under the restrictions of $\pi_{1}$ and $\pi_{2}$ to $V\setminus(A_{\xi}\times_{B}A_{\zeta})$.
As a result, the fiber product $V$ and $X_{\epsilon}$ have the same dimension.
Note that the center $B=\{x_{1},\cdots, x_{k}\}$ is a finite set.
It follows from Theorem \ref{m2} that
$$
A_{\xi}\times_{B}A_{\zeta}=\bigcup_{i=1}^{k}\bigl(A_{\xi, i}\times A_{\zeta,i}\bigr)
$$
is a hypersurface and hence there exists a unique morphism $\rho:V\rightarrow \tilde{X}_{\epsilon}$ such that $\pi_{1}=h\circ \rho$.
By the commutativity of \eqref{univ-f-b}, we obtain
\begin{equation*}
h\circ(\rho\circ\sigma)=(h\circ\rho)\circ\sigma=\pi_{1}\circ\sigma=h.
\end{equation*}
Because of the uniqueness of a holomorphic morphism $\varkappa:\tilde{X}_{\epsilon}\rightarrow\tilde{X}_{\epsilon}$ such that $h\circ\varkappa=h$, we must have $\rho\circ\sigma=\mathrm{id}_{\tilde{X}_{\epsilon}}$.
This implies that $\sigma:\tilde{X}_{\epsilon}\rightarrow V$ is an imbedding.
Observe that $\tilde{X}_{\epsilon}$ is irreducible.
Consequently, we are led to the conclusion that $\tilde{X}_{\epsilon}$ is biholomorphic to the irreducible component of $V$ which dominates $X_{\epsilon}$ and this completes the proof.
\end{proof}

Let $\pi_{\epsilon,1}:X_{\epsilon, 1}\rightarrow X_{\epsilon, 0}=X_{\epsilon}$ be the blow-up of $X_{\epsilon}$ at $x_{1}$.
Under the biholomorphism
$$
X_{\epsilon,1}\setminus\pi^{-1}_{\epsilon,1}(x_{1})\stackrel{\simeq}\longrightarrow X_{\epsilon,0}\setminus\{x_{1}\}
$$
induced by $\pi_{\epsilon, 1}$, the point $x_{2}$ can be considered as a point of $X_{\epsilon, 1}$.
Denote by $\pi_{\epsilon, 2}:X_{\epsilon,2}\rightarrow X_{\epsilon,1}$ the blow-up of $X_{\epsilon,1}$ at the point $x_{2}$.
Likewise, for every $3\leq i\leq k$, let $\pi_{\epsilon,i}:X_{\epsilon,i}\rightarrow X_{\epsilon,i-1}$ be the blow-up of $X_{\epsilon, i-1}$ at $x_{i}$.
By definition, $g$ is equal to the composition of the blow-ups $\pi_{\epsilon, 1}\circ\cdots\circ\pi_{\epsilon, k}$.
Observe that $A_{\xi}$ is the disjoint union of complex projective spaces $A_{\xi, 1},\cdots, A_{\xi, k}$. Let
$\varpi_{1, i}:X_{\xi, 1}\rightarrow X_{\xi, 1}=X_{\xi}$ be the blow-up of $X_{\xi, 0}$ centered at $A_{\xi, 1}$ and $\varpi_{\xi, i}:X_{\xi, i}\rightarrow X_{\xi, i-1}$ be the blow-up of $X_{\xi, i-1}$ with the center $A_{\xi, i}$ for every $2\leq i\leq k$.
From the construction of $h$, we obtain $h=\varpi_{\xi,1}\circ\cdots\circ\varpi_{\xi, k}$.
By Theorem \ref{qusi-free-wc}, the bimeromorphic map $f:X_{\xi}\dashrightarrow X_{\zeta}$ is factored into a succession of $k$-blow-ups with smooth centers followed by a succession of $k$-blow-downs with smooth centers, and this gives rise to an explicit bimeromorphic map between compact K\"{a}hler manifolds for which the Strong Factorization Conjecture holds.

\begin{rem}
Comparing to \cite[Theorem 1.9]{Th96} on the variation of GIT-quotients, in some sense, Theorem \ref{qusi-free-wc} is an analytic counterpart of \cite[Theorem 1.9]{Th96} under the condition of quasi-free actions.
\end{rem}
%============================
\section{Comparison of K\"{a}hler classes}\label{var-cla}

Let $(X, ds^{2}, T^{\mathbb{C}}, \Phi)$ be a compact K\"{a}hler Hamiltonian $T$-manifold.
Denote by $\omega$ the K\"{a}hler form of $ds^{2}$ and $\check{\omega}$ the \v{C}ech cohomology class corresponding to $[\omega]$ under the de Rham--\v{C}ech isomorphism $H^{2}_{DR}(X)\cong\check{H}^{2}(X,\underline{\mathbb{R}}_{X})$.
For each $a\in\Phi(X)$, set $Z_{a}=\Phi^{-1}(a)$.
Due to Theorem \ref{main-k-q}, we shall identify $X_{a}=Z_{a}/T$ with $X^{ss}_{a}/\!\!/T^{\mathbb{C}}$ as complex analytic spaces.
Let $\imath_{a}:Z_{a}\hookrightarrow X$ be the inclusion and $\pi_{a}:Z_{a}\rightarrow X_{a}$ the quotient map.
In particular, there exists a unique K\"{a}hler metric (in the sense of Definition \ref{kahler space}) $\kappa_{a}$ on $X_{a}$ induced from $ds^{2}$, such that the corresponding K\"{a}hler class $\check{\omega}_{a}=c_{1}(\kappa_{a})\in\check{H}^{2}(X_{a}, \underline{\mathbb{R}}_{X_{a}})$ satisfies the condition $\pi^{\ast}_{a}\check{\omega}_{a}=\imath^{\ast}_{a}\check{\omega}$, see \cite[Corollary 4]{HS20} and \cite[Proposition 4.2]{Ya22}.
According to \cite[Theorem 1.1]{Ya22}, the K\"{a}hler class $\check{\omega}_{a}$ actually gives rise to a \emph{cohomologically symplectic structure} on $X_{a}$ in the following sense:
\begin{defn}[{\cite[Definition 5.1]{Ya22}}]
Let $(M,\mathscr{S}_{M},\mathscr{C}^{\infty}_{M})$ be a compact reduced differentiable stratified space of dimension $2n$.
By a cohomologically symplectic structure on $M$, we mean a \v{C}ech cohomology class $\check{\omega}\in\check{H}^{2}(M, \underline{\mathbb{R}}_{M})$ such that $\check{\omega}^{n}$ is nonzero in $\check{H}^{2n}(M, \underline{\mathbb{R}}_{M})$.
\end{defn}
This implies that the K\"{a}hler reduction of $(X, ds^{2}, T^{\mathbb{C}}, \Phi)$ produces a family of cohomologically symplectic stratified spaces:
$$
\{(X_{a}, \check{\omega}_{a})\,|\,a\in\Phi(X)\}.
$$
Inspired by \cite{DH82} and \cite{GS89}, we consider the wall-crossing problem of the cohomologically symplectic structures on K\"{a}hler quotients.
With the assumptions (\textbf{A1})--(\textbf{A3}), we get three cohomologically symplectic stratified spaces $(X_{\xi}, \check{\omega}_{\xi})$, $(X_{\epsilon}, \check{\omega}_{\epsilon})$ and $(X_{\zeta}, \check{\omega}_{\zeta})$, together with two modifications
$f_{\xi,\epsilon}:X_{\xi}\rightarrow X_{\epsilon}$ and $f_{\zeta,\epsilon}:X_{\zeta}\rightarrow X_{\epsilon}$, see Figure \ref{fig-2} below.
\begin{figure*}[h]
\centering
\begin{tikzpicture}[scale=0.4]
  \shade[ball color=black] (3,8) circle (0.1);
    \shade[ball color=black] (11,8) circle (0.1);
     \shade[ball color=black] (7,8) circle (0.1);
\draw  (1.2,7.2)node[right] {$\xi$};
\draw  (10.8,7.2)node[right] {$\zeta$};
\draw  (5.8,7.2)node[right] {$\epsilon$};
\draw  (0,14)node[right] {$(X_{\xi}, \check{\omega}_{\xi})$};
\draw  (6.5,14)node[right] {$(X_{\epsilon}, \check{\omega}_{\epsilon})$};
\draw  (12.5,14)node[right] {$(X_{\zeta}, \check{\omega}_{\zeta})$};
\draw  (3.5,5.2)node[right] {$\Delta_{-}$};
\draw  (8,5.2)node[right] {$\Delta_{+}$};
\draw[-latex,dashed,red](3,8)--(1.5,13.2);
\draw[-latex,dashed,red](11,8)--(13,13.2);
\draw[-latex,dashed,red](7,8)--(7.8,13.2);
\draw[-latex,blue] (3,8) --(7,8);
\draw[-latex,blue] (11,8)--(7,8);
\draw[black](7,2)--(7,12);
\draw (0,7) --(7,12)--(14,7)--(7,2)--(0,7);
\fill[dashed,blue, opacity=0.1] (0,7) --(7,12)--(14,7)--(7,2)--(0,7);
\end{tikzpicture}
\caption{}
\label{fig-2}
\end{figure*}
The pullback of locally constant $\mathbb{R}$-valued functions gives rise to a morphism of sheaves
\begin{equation}\label{con1-iso}
f^{\natural}_{\xi,\epsilon}:\underline{\mathbb{R}}_{X_{\epsilon}}\longrightarrow
(f_{\xi,\epsilon})_{\ast}\underline{\mathbb{R}}_{X_{\xi}}.
\end{equation}
Since the map $f_{\xi,\epsilon}$ is surjective and closed\footnote{Because $X_{\xi}$ is compact and $X_{\epsilon}$ is Hausdorff, the closed map lemma shows that $f_{\xi,\epsilon}$ is a closed map.}, $X_{\epsilon}$ has the quotient topology determined by $f_{\xi,\epsilon}$.
Observe that the fibers of $f_{\xi,\epsilon}$ is connected, see Property \ref{fiber-conn}.
It follows that the morphism \eqref{con1-iso} is an isomorphism (cf. \cite[Example 3.44]{Wed16}).
Likewise, the modification $f_{\zeta,\epsilon}:X_{\zeta}\rightarrow X_{\epsilon}$ also yields an isomorphism of sheaves
\begin{equation}\label{con2-iso}
f^{\natural}_{\zeta,\epsilon}:\underline{\mathbb{R}}_{X_{\epsilon}}
\stackrel{\simeq}\longrightarrow
(f_{\zeta,\epsilon})_{\ast}\underline{\mathbb{R}}_{X_{\zeta}}.
\end{equation}
As a result, we have the following isomorphisms of \v{C}ech cohomology groups:
\begin{equation}\label{cech-iso}
\check{H}^{\ast}(X_{\epsilon}, (f_{\xi,\epsilon})_{\ast}\underline{\mathbb{R}}_{X_{\xi}})\cong
\check{H}^{\ast}(X_{\epsilon}, \underline{\mathbb{R}}_{X_{\epsilon}})\cong
\check{H}^{\ast}(X_{\epsilon}, (f_{\zeta,\epsilon})_{\ast}\underline{\mathbb{R}}_{X_{\zeta}}).
\end{equation}

Consider the Leray spectral sequence corresponding to $f_{\xi,\epsilon}$, denoted by $\{E_{r}, d_{r}\}$, which has the second term
$$
E^{p,q}_{2}=\check{H}^{p}(X_{\epsilon}, R^{q}(f_{\xi,\epsilon})_{\ast}\underline{\mathbb{R}}_{X_{\xi}})
$$
and converges to $\check{H}^{\ast}(X_{\xi},\underline{\mathbb{R}}_{X_{\xi}})$.
The edge homomorphism
$$
E^{2,0}_{2}=\check{H}^{2}(X_{\epsilon}, (f_{\xi,\epsilon})_{\ast}\underline{\mathbb{R}}_{X_{\xi}})\twoheadrightarrow
E^{2,0}_{\infty}\hookrightarrow
\check{H}^{2}(X_{\xi},\underline{\mathbb{R}}_{X_{\xi}})
$$
and the first isomorphism in \eqref{cech-iso} derives a canonical morphism
\begin{equation}\label{cech-2-xi}
f_{\xi,\epsilon}^{\ast}:
\check{H}^{2}(X_{\epsilon},\underline{\mathbb{R}}_{X_{\epsilon}})
\longrightarrow
\check{H}^{2}(X_{\xi},\underline{\mathbb{R}}_{X_{\xi}}).
\end{equation}
Similarly, from \eqref{con2-iso} and the second isomorphism in \eqref{cech-iso}, the modification $f_{\zeta,\epsilon}$ also induces a canonical morphism
\begin{equation}\label{cech-2-zeta}
f_{\zeta,\epsilon}^{\ast}:
\check{H}^{2}(X_{\epsilon},\underline{\mathbb{R}}_{X_{\epsilon}})
\longrightarrow
\check{H}^{2}(X_{\zeta},\underline{\mathbb{R}}_{X_{\zeta}}).
\end{equation}

Particularly, we have the following:
\begin{lem}\label{key-lem}
Both the morphisms $f_{\xi,\epsilon}^{\ast}$ and $f_{\zeta,\epsilon}^{\ast}$ are injective.
\end{lem}
\begin{proof}
Consider the second term of the Leray spectral sequence $\{E_{r}, d_{r}\}$ associated to the morphism $f_{\xi,\epsilon}$ and the sheaf $\underline{\mathbb{R}}_{X_{\xi}}$.
Let $(\mathscr{A}^{[\bullet]}, d)$ be the simplicial flabby resolution of $\underline{\mathbb{R}}_{X_{\xi}}$ which yields a complex of sheaves on $X_{\epsilon}$ denoted by $\mathscr{L}^{\bullet}=(f_{\xi,\epsilon})_{\ast}\mathscr{A}^{[\bullet]}$.
Then we can associate to $\mathscr{L}^{\bullet}$ a double complex of groups
$$
K^{p,q}_{\mathscr{L}}=(\mathscr{L}^{q})^{[p]}(X_{\epsilon}),
$$
where $p,q\geq0$.
The Leray spectral sequence $\{E_{r}, d_{r}\}$ can be reformulated as the spectral sequence of the double complex $K^{\bullet,\bullet}_{\mathscr{L}}$, see \cite[\S 13.B]{Dem12}.
According to \cite[Theorem 11.8]{Dem12}, there exists an exact sequence:
\begin{equation}\label{5-ex}
\xymatrix@C=0.7cm{
  0 \ar[r] & E^{1,0}_{2} \ar[r]^{} & \check{H}^{1}(X_{\xi},\underline{\mathbb{R}}_{X_{\xi}})
   \ar[r]^{} & E^{0,1}_{2} \ar[r]^{d_{2}} & E^{2,0}_{2} \ar[r]^{} &
   \check{H}^{2}(X_{\xi},\underline{\mathbb{R}}_{X_{\xi}}) &  }
\end{equation}
where the non indicated arrows are edge homomorphisms.
From the definition of $f^{\ast}_{\xi,\epsilon}$ and the exactness of \eqref{5-ex},
$f^{\ast}_{\xi,\epsilon}$ is injective if and only if the image of
$$
E^{0,1}_{2}=
\check{H}^{0}(X_{\epsilon}, R^{1}(f_{\xi,\epsilon})_{\ast}\underline{\mathbb{R}}_{X_{\xi}})
$$
under $d_{2}$ is zero.
We claim that $R^{1}(f_{\xi,\epsilon})_{\ast}\underline{\mathbb{R}}_{X_{\xi}})$ is vanishing.
Recall that $f_{\xi,\epsilon}:X_{\xi}\rightarrow X_{\epsilon}$ is a proper modification with the center $B$.
It suffices to show that the stalk of $R^{1}(f_{\xi,\epsilon})_{\ast}\underline{\mathbb{R}}_{X_{\xi}}$ at each point $x=[G\cdot b]\in B$ is zero.
By definition, we have
$$
\bigl(R^{1}(f_{\xi,\epsilon})_{\ast}\underline{\mathbb{R}}_{X_{\xi}}\bigr)_{x}=
\varinjlim_{x\in U}\check{H}^{1}(f^{-1}_{\xi,\epsilon}(U),\underline{\mathbb{R}}_{X_{\xi}})
$$
where $U$ runs over all open neighborhoods of $x$ in $X_{\epsilon}$.
Put $X_{\xi,x}=f^{-1}_{\xi,\epsilon}(x)$.
As $f_{\xi,\epsilon}:X_{\xi}\rightarrow X_{\epsilon}$ is a proper map of compact spaces, applying \cite[Chapter IV, Theorem 9.10]{Dem12} to $f_{\xi,\epsilon}$, we obtain
$$
\bigl(R^{1}(f_{\xi,\epsilon})_{\ast}\underline{\mathbb{R}}_{X_{\xi}}\bigr)_{x}=
\check{H}^{1}(X_{\xi,x},i^{-1}_{x}\underline{\mathbb{R}}_{X_{\xi}}),
$$
where $i_{x}:X_{\xi,x}\hookrightarrow X_{\xi}$ is the inclusion map.
Note that $i^{-1}_{x}\underline{\mathbb{R}}_{X_{\xi}}=\underline{\mathbb{R}}_{X_{\xi,x}}$ (cf. \cite[Example 3.56]{Wed16}).
So we have
$$
\bigl(R^{1}(f_{\xi,\epsilon})_{\ast}\underline{\mathbb{R}}_{X_{\xi}}\bigr)_{x}=
\check{H}^{1}(X_{\xi,x},\underline{\mathbb{R}}_{X_{\xi,x}}).
$$
According to Theorem \ref{m2}, the fiber $X_{\xi,x}$ is biholomorphic to the quotient of a weighted projective space by a finite group $\widetilde{\textbf{P}}/\Gamma$,
where $\Gamma$ is determined by the stabilizer $G_{b}$.
Let $\pi:\widetilde{\textbf{P}}\rightarrow Q:=\widetilde{\textbf{P}}/\Gamma$ be the quotient map.
To end the proof, we need to verify
$
\check{H}^{1}(Q,\underline{\mathbb{R}}_{Q})=0.
$
Observe that the constant sheaf $\underline{\mathbb{R}}_{\widetilde{\textbf{P}}}$ can be regarded as an abelian $\Gamma$-sheaf on $\widetilde{\textbf{P}}$ in the sense of \cite[\S\,5.1]{Gro57}, and the multiplication by the order of $\Gamma$ is an automorphism of $\underline{\mathbb{R}}_{\widetilde{\textbf{P}}}$.
Particularly, there is a natural sheaf morphism induced by the pullback of locally constant functions
$$
\pi^{\natural}:\underline{\mathbb{R}}_{Q}\longrightarrow
\pi^{\Gamma}_{\ast}\underline{\mathbb{R}}_{\widetilde{\textbf{P}}},
$$
where $\pi^{\Gamma}_{\ast}$ is the $\Gamma$-invariant direct image functor.
A direct checking shows that $\pi^{\natural}$ is an isomorphism and therefore we have
\begin{equation}\label{iso-y-z-1}
\check{H}^{1}(Q,\underline{\mathbb{R}}_{Q})\cong
\check{H}^{1}(Q,\pi^{\Gamma}_{\ast}\underline{\mathbb{R}}_{\widetilde{\textbf{P}}}).
\end{equation}
By a result by Grothendieck \cite[Page 16, Grothendieck Theorem]{Ala97}, there exists an isomorphism
\begin{equation}\label{iso-y-z-2}
\check{H}^{1}(Q,\pi^{\Gamma}_{\ast}\underline{\mathbb{R}}_{\widetilde{\textbf{P}}})\cong
\check{H}^{1}(\widetilde{\textbf{P}},
\underline{\mathbb{R}}_{\widetilde{\textbf{P}}})^{\Gamma}.
\end{equation}
On account of \cite[Corollary 2.3.6]{Dol82} due to Dolgachev, we have the following vanishing result
\begin{equation}\label{iso-y-z-3}
\check{H}^{1}(\widetilde{\textbf{P}},
\underline{\mathbb{R}}_{\widetilde{\textbf{P}}})=0.
\end{equation}
From \eqref{iso-y-z-1}-\eqref{iso-y-z-3}, we are led to the conclusion $\check{H}^{1}(Q,\underline{\mathbb{R}}_{Q})=0$.
This implies that \eqref{cech-2-xi} is injective and so is the morphism \eqref{cech-2-zeta} by the same argument.
\end{proof}

Denote by $\Gamma_{\xi}<T$ the finite subgroup generated by stabilizers $T_{x}$ for all $x\in Z_{\xi}$.
Put $Y_{\xi}=Z_{\xi}/\Gamma_{\xi}$.
Then there exists a commutative diagram
\begin{equation*}\label{zyx}
\vcenter{
\xymatrix@C=1cm{
  Z_{\xi} \ar[dr]_{\pi_{\xi}} \ar[r]^{r_{\xi}}
                & Y_{\xi} \ar[d]^{q_{\xi}}  \\
                & X_{\xi}             }
                }
\end{equation*}
where $r_{\xi}$ is a finite brached covering and $q_{\xi}:Y_{\xi}\rightarrow X_{\xi}$ is a principal $T/\Gamma_{\xi}$-bundle.
Put $\Lambda_{\xi}=\{v\in\mathfrak{t}\,|\, \exp(v)\in\Gamma_{\xi}\}$.
Then the principal $T/\Gamma_{\xi}$-bundle $q_{\xi}:Y_{\xi}\rightarrow X_{\xi}$ is characterized by a topological invariant called the Chern class $c_{\xi}\in\check{H}^{2}(X_{\xi}, \Lambda_{\xi})$ (cf. \cite[\S\,1]{DH82}).
For any $a\in\mathrm{int}\,\Delta_{-}$, there exists a canonical identification of $X_{a}$ with $X_{\xi}$ and thus
$
\check{H}^{\ast}(X_{a}, \underline{\mathbb{R}}_{X_{a}})\cong\check{H}^{\ast}(X_{\xi}, \underline{\mathbb{R}}_{X_{\xi}}),
$
see \cite{DH82}.
In particular, we have $c_{a}=c_{\xi}$.
On account of the Duistermaat--Heckman theorem \cite[Theorem 1.1]{DH82}, for any $a\in\mathrm{int}\,\Delta_{-}$, the cohomology class
$
\check{\omega}_{a}\in\check{H}^{2}(X_{a}, \underline{\mathbb{R}}_{X_{a}})\cong\check{H}^{2}(X_{\xi}, \underline{\mathbb{R}}_{X_{\xi}})
$
satisfies the equality\footnote{In this paper, we assume that the moment map satisfies the property $d\Phi^{\xi}=\iota_{\xi_{X}}\omega$ which is different from \cite{DH82}.
}
\begin{equation}\label{dh-xi}
\check{\omega}_{\xi}=\check{\omega}_{a}-\langle c_{\xi}, \xi-a\rangle
\end{equation}
in $\check{H}^{2}(X_{\xi}, \underline{\mathbb{R}}_{X_{\xi}})$.
Similarly, we have
\begin{equation}\label{dh-zeta}
\check{\omega}_{\zeta}=\check{\omega}_{b}-\langle c_{\zeta}, \zeta-b\rangle,
\end{equation}
for any $b\in\mathrm{int}\,\Delta_{+}$.

Using the morphisms \eqref{cech-2-xi} and \eqref{cech-2-zeta}, we can extend the Duistermaat--Heckman formulae to the critical value $\epsilon$.
\begin{lem}\label{limit-dh}
Let $a$ and $b$ trend to $\epsilon$, we have the limit forms of \eqref{dh-xi} and \eqref{dh-zeta}, i.e., the formulae
$$
\check{\omega}_{\xi}=f^{\ast}_{\xi,\epsilon}(\check{\omega}_{\epsilon})-\langle c_{\xi}, \xi-\epsilon\rangle
$$
and
$$
\check{\omega}_{\zeta}=f^{\ast}_{\zeta,\epsilon}(\check{\omega}_{\epsilon})-\langle c_{\zeta}, \zeta-\epsilon\rangle
$$
hold in $\check{H}^{2}(X_{\xi}, \underline{\mathbb{R}}_{X_{\xi}})$ and $\check{H}^{2}(X_{\xi}, \underline{\mathbb{R}}_{X_{\xi}})$, respectively.
\end{lem}
\begin{proof}
We divide the proof into three steps.

$\textbf{Step 1.}$
Due to the shifting trick, we may assume that $\epsilon=0$ and thus we have the following commutative square
\begin{equation}\label{xi-0}
\vcenter{
\xymatrix@=1.3cm{
  X^{s}_{\xi} \ar[d]_{\Pi_{\xi}}
  \ar[r]^{i_{\xi,0}} & X^{ss}_{0} \ar[d]^{\Pi_{0}} \\
  X_{\xi}
  \ar[r]^{f_{\xi,0}} & X_{0}}
  }
\end{equation}
Recall the construction of the K\"{a}hler class $\check{\omega}_{0}$.
According to \cite[Theorem 3]{HS20}, there exists a $T^{\mathbb{C}}$-invariant open covering $\mathscr{U}=\{U_{i}\}_{i\in I}$ of the level set $Z_{0}=\Phi^{-1}(0)$ in $X^{ss}_{0}$ together with $T$-invariant smooth strictly plurisubharmonic functions $\rho_{i}$ on $U_{i}$ such that each $U_{i}$ is $\Pi_{0}$-saturated, and
\begin{equation}\label{local-pt}
dd^{c}\rho_{i}=\omega|_{U_{i}},\quad\quad
-\iota_{\xi_{X}}(d^{c}\rho_{i})=(\Phi^{\xi})|_{U_{i}}\,\,\,(\forall\,\xi\in\mathfrak{t})
\end{equation}
where $d^{c}=\sqrt{-1}(\bar{\partial}-\partial)$ and $\xi_{X}$ is the vector field on $X$ generated by $\xi$.
Put $W_{i}=\Pi_{0}(U_{i})$.
Then $\mathscr{W}_{0}=\{W_{i}\}_{i\in I}$ becomes an open covering of $X_{0}$, and each $\rho_{i}$ descends to a continuous strictly plurisubharmonic function $\varphi_{i}$ on $W_{i}$ which defines a K\"{a}hler metric
$$
\kappa_{0}=\{(W_{i}, \varphi_{i})\}\in \check{H}^{0}(X_{0}, \mathscr{K}^{1}_{X_{0},\mathbb{R}})
$$
in the sense of Definition \ref{kahler space}.
The K\"{a}hler class $\check{\omega}_{0}$ is defined to be $\check{\omega}_{0}=c_{1}(\kappa_{0})$ under the morphism
$$
c_{1}: \check{H}^{0}(X_{0}, \mathscr{K}^{1}_{X_{0},\mathbb{R}})\stackrel{\delta^{0}}\longrightarrow
\check{H}^{1}(X_{0}, \mathscr{PH}_{X_{0},\mathbb{R}})
\stackrel{\delta^{1}}\longrightarrow \check{H}^{2}(X_{0}, \underline{\mathbb{R}}_{X_{0}}).
$$

Note that $f_{\xi,0}:X_{\xi}\rightarrow X_{0}$ is a modification.
Set
$$
\mathscr{W}_{\xi}=\{W_{i,\xi}=f^{-1}_{\xi,0}(W_{0})\}_{i\in I}
$$
which forms an open covering of $X_{\xi}$.
The pullback $\psi_{0,i}:=(f_{\xi,0}|_{W_{i,\xi}})^{\ast}(\varphi_{i})$ is only continuous on $W_{i,\xi}$ in general.
Since $f_{\xi,0}$ is holomorphic, we get
$$
(\psi_{0,i}-\psi_{0,j})|_{W_{i,\xi}\cap W_{j,\xi}}\in
\mathscr{PH}_{X_{\xi},\mathbb{R}}(W_{i,\xi}\cap W_{j,\xi}).
$$
It follows that $\{(W_{i,\xi}, \psi_{0,i})\}$ defines an element in $\check{H}^{0}(X_{\xi}, \mathscr{K}^{1}_{X_{\xi},\mathbb{R}})$.
Similarly, by the morphism
\begin{equation}\label{c1-xi}
c_{1}: \check{H}^{0}(X_{\xi}, \mathscr{K}^{1}_{X_{\xi},\mathbb{R}})\stackrel{\delta^{0}}\longrightarrow
\check{H}^{1}(X_{\xi}, \mathscr{PH}_{X_{\xi},\mathbb{R}})
\stackrel{\delta^{1}}\longrightarrow \check{H}^{2}(X_{\xi}, \underline{\mathbb{R}}_{X_{\xi}}),
\end{equation}
we obtain a \v{C}ech cohomology class $c_{1}(\{(W_{i,\xi}, \psi_{0,i})\})$ in $\check{H}^{2}(X_{\xi}, \underline{\mathbb{R}}_{X_{\xi}})$.
By definition, the equality
\begin{equation}\label{lem6.3-equ0}
f^{\ast}_{\xi,0}(\check{\omega}_{0})=c_{1}(\{(W_{i,\xi}, \psi_{0,i})\})
\end{equation}
is valid.

$\textbf{Step 2.}$
The commutativity of \eqref{xi-0} deduces $\Pi^{-1}_{\xi}(W_{i,\xi})=U_{i}\cap X^{s}_{\xi}$ and that the quotient space $(U_{i}\cap X^{s}_{\xi}\cap\Phi^{-1}(\xi))/T$ is homeomorphic to $\Pi_{\xi}(U_{i}\cap X^{s}_{\xi})=W_{i,\xi}$.
Since $\xi$ is a regular value of $\Phi$, the restriction of $\rho_{i}$ to $U_{i}\cap X^{s}_{\xi}\cap\Phi^{-1}(\xi)$ descends to a smooth function (in the orbifold sense) on $W_{i,\xi}$ denoted by $\psi_{\xi,i}$.
We claim that $\{(W_{i,\xi}, \psi_{\xi,i})\}$ defines an element in $\check{H}^{0}(X_{\xi}, \mathscr{K}^{1}_{X_{\xi},\mathbb{R}})$.

For any point $p$ in $U_{i}\cap X^{s}_{\xi}\cap\Phi^{-1}(\xi)$, put $x=\Pi_{\xi}(p)$ and then we have
\begin{equation}\label{m-x}
\psi_{\xi,i}(x)=\rho_{i}(p).
\end{equation}
Denote by $F_{t}$ the $T$-equivariant gradient flow of the function $\mu=-\frac{1}{2}\|\Phi\|^{2}$.
For any $p\in X$, we identify $\Phi(p)\in\mathfrak{t}^{\ast}$ with a vector in $\mathfrak{t}$ using the fixed inner product, and denote by $(\Phi(p))_{X,p}$ the vector field on $X$ induced by $\Phi(p)$, evaluated at the point $p$.
Then we have
\begin{equation}\label{grad-mu}
\mathrm{grad}\,\mu(p)=-J\Phi(p)_{X,p},
\end{equation}
where $J$ is the complex structure on $X$.
It is well known that the limit
$$F_{\infty}(p):=\lim_{t\rightarrow\infty}F_{t}(p)$$
exists and $F_{\infty}(p)$ lies in $\Phi^{-1}(0)$.
From the definition of $f_{\xi,0}$, we get
$$
f_{\xi,0}(x)=f_{\xi,0}(\Pi_{\xi}(p))=\Pi_{0}(F_{\infty}(p)).
$$
As a result, for any $x\in W_{i,\xi}$, the following equality holds
\begin{equation}\label{m-x-1}
\psi_{0,i}(x)=\varphi_{i}(f_{\xi,0}(x))=\varphi_{i}(\Pi_{0}(F_{\infty}(p)))
=\rho_{i}(F_{\infty}(p)).
\end{equation}
Due to \eqref{local-pt} and \eqref{grad-mu}, we obtain
\begin{eqnarray*}\label{drho-dt}
\frac{d}{dt}\rho_{i}(F_{t}(p))
&=&\langle d\rho_{i}, \mathrm{grad}\,\mu\rangle(F_{t}(p))\\
&=&-\langle d\rho_{i}, J\Phi(F_{t}(p))_{X}\rangle(F_{t}(p))\\
&=&-\langle Jd\rho_{i}, \Phi(F_{t}(p))_{X}\rangle(F_{t}(p))\\
&=&\langle d^{c}\rho_{i}, \Phi(F_{t}(p))_{X}\rangle(F_{t}(p))\\
&=&-\|\Phi(F_{t}(p))\|^{2}.
\end{eqnarray*}
Collecting \eqref{m-x}, \eqref{m-x-1} and the equality above derives
\begin{eqnarray*}
 \psi_{\xi,i}(x)-\psi_{0,i}(x)  &=&
   \rho_{i}(F_{0}(p))- \rho_{i}(F_{\infty}(p)) \\
   &=& -\int^{\infty}_{0} \biggl(\frac{d}{dt}\rho_{i}(F_{t}(p))\biggr)dt\\
   &=& \int^{\infty}_{0}\|\Phi(F_{t}(p))\|^{2}dt
\end{eqnarray*}
which means that $\psi_{\xi,i}(x)-\psi_{0,\xi}(x)$ is independent of the index $i$ or $j$ for any $x$ in $W_{i,\xi}\cap W_{j,\xi}$.
As a result, the following equality holds
\begin{equation*}
\psi_{\xi,i}-\psi_{0,i}=\psi_{\xi,j}-\psi_{0,j}
\end{equation*}
on $W_{i,\xi}\cap W_{j,\xi}$ and thus we get
\begin{equation}\label{i-j-plu}
(\psi_{\xi,i}-\psi_{\xi,j})|_{W_{i,\xi}\cap W_{j,\xi}}
=(\psi_{0,i}-\psi_{0,j})|_{W_{i,\xi}\cap W_{j,\xi}}.
\end{equation}
Note that the right-hand side of \eqref{i-j-plu} are plurisubharmonic function on $W_{i,\xi}\cap W_{j,\xi}$.
This implies $\{(W_{i,\xi},\psi_{\xi,i})\}\in\check{H}^{0}(X_{\xi}, \mathscr{K}^{1}_{X_{\xi},\mathbb{R}})$ and
$$
\delta^{0}(\{(W_{i,\xi},\psi_{\xi,i})\})
=\delta^{0}(\{(W_{i,\xi},\psi_{0,i})\})
$$
in $\check{H}^{1}(X_{\xi}, \mathscr{PH}_{X_{\xi},\mathbb{R}})$; moreover, from \eqref{c1-xi}, we have
\begin{equation}\label{lem6.3-equ1}
c_{1}(\{(W_{i,\xi}, \psi_{\xi,i})\})=c_{1}(\{(W_{i,\xi}, \psi_{0,i})\})
\end{equation}
in $\check{H}^{2}(X_{\xi}, \underline{\mathbb{R}}_{X_{\xi}})$.

$\textbf{Step 3.}$
Observe that each $\psi_{\xi,i}$ is a smooth function (in the orbifold sense) and $\{dd^{c}\psi_{\xi,i}\}$ glues to a globally defined differential 2-form on the complex orbifold $X_{\xi}$.
Under the de Rham--\v{C}ech isomorphism $H^{2}_{DR}(X_{\xi})\cong\check{H}^{2}(X_{\xi},\underline{\mathbb{R}}_{X_{\xi}})$,
we have
$$
[\{dd^{c}\psi_{\xi,i}\}]=c_{1}(\{(W_{i,\xi}, \psi_{\xi,i})\}).
$$
Due to \eqref{lem6.3-equ0} and \eqref{lem6.3-equ1}, to finish the proof, it suffices to show that
$$[\omega_{\xi}]-[\{dd^{c}\psi_{\xi,i}\}]=-\langle c_{\xi}, \xi\rangle,$$
where $\omega_{\xi}$ is the reduced symplectic form on $X_{\xi}$.

For simplicity, we may assume that the $T$-action on $Z_{\xi}=\Phi^{-1}(\xi)$ is free, and the general case can be handled by using a suitable finite branched covering as in \cite{DH82}.
For any $p\in Z_{\xi}$, one has a canonical orthogonal decomposition
$$
T_{p}Z_{\xi}=T_{p}(T\cdot p)\oplus W_{p}\cong\mathfrak{t}\oplus W_{p},
$$
where $W_{p}$ is $J$-invariant.
The metric on $X$ defines a $\mathfrak{t}$-valued 1-form $\Theta$ on $Z_{\xi}$, which is a connection form for the principal $T$-bundle $\pi_{\xi}:Z_{\xi}\rightarrow Z_{\xi}/T\cong X_{\xi}$ satisfying the conditions
\begin{equation}\label{w-eta}
\Theta(w)=0\,\,\,(\forall\,w\in W_{p})\,\,\,\mathrm{and}\,\,\,
\Theta(\eta_{X,p})=\eta\,\,\,(\forall\,\eta\in \mathfrak{t}).
\end{equation}
Recall that $\imath_{\xi}:Z_{\xi}\hookrightarrow X$ is the inclusion map and the quotient map $\pi_{\xi}$ is identical with the restriction $\Pi_{\xi}|_{Z_{\xi}}$.
Put $x=\pi_{\xi}(p)$.
Then $W_{p}$ is isomorphic to $T_{x}X_{\xi}$ and the complex structure on $T_{x}X_{\xi}$ is induced from that on $W_{p}$.
Consequently, for any $w\in(W_{p}\otimes\mathbb{C})^{1,0}$, we have
$(\pi_{\xi})_{\ast}(w)\in T^{1,0}_{x}X_{\xi}$ and
\begin{eqnarray*}
   & &(\imath^{\ast}_{\xi}d^{c}\rho_{i})_{p}(w)-
   (\pi^{\ast}_{\xi}d^{c}\psi_{\xi,i})_{p}(w)\\
   &=& -\sqrt{-1}(\imath^{\ast}_{\xi}d\rho_{i})_{p}(w)+\sqrt{-1}
   (\pi^{\ast}_{\xi}d\psi_{\xi,i})_{p}(w)\,\,\,\,\,\,(w\,\,\mathrm{is}\,\, \mathrm{of}\,\, (1,0)-\mathrm{type}) \\
   &=&  -\sqrt{-1}(d\rho_{i})_{p}(w)+\sqrt{-1}
   (d\pi^{\ast}_{\xi}\psi_{\xi,i})_{p}(w) \\
   &=&  -\sqrt{-1}(d\rho_{i})_{p}(w)+\sqrt{-1}
   (d\rho_{i})_{p}(w)\,\,\,\,\,\,\,\,\qquad (\pi^{\ast}_{\xi}\psi_{\xi,i}=\rho_{i})\\
   &=&  0.
\end{eqnarray*}
The same argument shows
$(\imath^{\ast}_{\xi}d^{c}\rho_{i})_{p}(w)-
   (\pi^{\ast}_{\xi}d^{c}\psi_{\xi,i})_{p}(w)=0
$ for any $w\in(W_{p}\otimes\mathbb{C})^{0,1}$ and therefore the 1-form $\imath^{\ast}_{\xi}d^{c}\rho_{i}-\pi^{\ast}_{\xi}d^{c}\psi_{\xi,i}$
is vanishing on $W_{p}$.
Meanwhile, for any $\eta\in\mathfrak{t}$, since $(\pi_{\xi})_{\ast,p}(\eta_{X,p})=0$ we get
$$
(\pi^{\ast}_{\xi}d^{c}\psi_{\xi,i})_{p}(\eta_{X,p})=
(d^{c}\psi_{\xi,i})_{x}((\pi_{\xi})_{\ast,p}(\eta_{X,p}))=0
$$
and thus
\begin{eqnarray*}
 (\imath^{\ast}_{\xi}d^{c}\rho_{i})_{p}(\eta_{X,p})-
   (\pi^{\ast}_{\xi}d^{c}\psi_{\xi,i})_{p}(\eta_{X,p})
   &=& (\imath^{\ast}_{\xi}d^{c}\rho_{i})_{p}(\eta_{X,p}) \\
   &\stackrel{\eqref{local-pt}}=& -\langle\eta,\Phi(p)\rangle \\
   &=& -\langle\eta, \xi\rangle\\
   &\stackrel{\eqref{w-eta}}=& -\langle\Theta(\eta_{X,p}), \xi\rangle.
\end{eqnarray*}
Consequently, we are led to the conclusion that
\begin{equation*}
\imath^{\ast}_{\xi}d^{c}\rho_{i}-\pi^{\ast}_{\xi}d^{c}\psi_{\xi,i}
=-\langle\Theta, \xi\rangle
\end{equation*}
is valid on $U_{i}\cap X^{s}_{\xi}\cap Z_{\xi}\cap
=\Pi^{-1}_{\xi}(W_{i,\xi})\cap Z_{\xi}$ for every $i\in I$.
From definition, the reduced symplectic form $\omega_{\xi}$ satisfies the property
$\pi^{\ast}_{\xi}\omega_{\xi}=\imath^{\ast}_{\xi}\omega
=\imath^{\ast}_{\xi}(\{dd^{c}\rho_{i}\})$ and hence we get
\begin{eqnarray*}
 \pi^{\ast}_{\xi}(\omega_{\xi}-\{dd^{c}\psi_{\xi,i}\})
 &=&\{\imath^{\ast}_{\xi}dd^{c}\rho_{i}-
 \pi^{\ast}_{\xi}dd^{c}\psi_{\xi,i}\} \\
 &=& \{d(\imath^{\ast}_{\xi}d^{c}\rho_{i}-
 \pi^{\ast}_{\xi}d^{c}\psi_{\xi,i})\} \\
 &=& -\langle d\Theta, \xi\rangle \\
 &=& -\pi^{\ast}_{\xi}\langle \Omega, \xi\rangle
\end{eqnarray*}
where $\Omega$ is the curvature form of the connection $\Theta$.
Taking the de Rham cohomology classes deduces an equality
$$
\pi^{\ast}_{\xi}([\omega_{\xi}]-[\{dd^{c}\psi_{\xi,i}\}])
=-\pi^{\ast}_{\xi}\langle c_{\xi}, \xi\rangle
$$
in $H^{2}_{DR}(Z_{\xi})$.
Since $\pi_{\xi}:Z_{\xi}\rightarrow X_{\xi}$ is a submersion the morphism $\pi^{\ast}_{\xi}$ is injective and therefore we obtain
$$
[\omega_{\xi}]-[\{dd^{c}\psi_{\xi,i}\}]
=-\langle c_{\xi}, \xi\rangle.
$$
This completes the proof.
\end{proof}

Consider the fiber product over $f_{\xi,\epsilon}$ and $f_{\zeta,\epsilon}$.
Then we get a cartesian square:
\begin{equation}\label{fiber-pro-0}
\vcenter{
\xymatrix@C=1cm{
  V=X_{\xi}\times_{X_{\epsilon}}X_{\zeta}\ar[dr]^{\phi}\ar[d]_{\pi_{1}} \ar[r]^{\quad\pi_{2}} & X_{\zeta} \ar[d]^{f_{\zeta,\epsilon}} \\
  X_{\xi} \ar[r]_{f_{\xi,\epsilon}} & X_{\epsilon}   }
  }
\end{equation}
where $\phi= f_{\xi,\epsilon}\circ\pi_{1}= f_{\zeta,\epsilon}\circ\pi_{2}$.
Due to Lemma \ref{lem-xi-zeta}, we identify all the K\"{a}hler quotient $X_{a}$ for $a\in[\xi, \epsilon)$ with $X_{\xi}$ and for $a\in(\epsilon, \zeta]$ with $X_{\zeta}$.
Note that $\pi_{1}$, $\pi_{2}$, and $\phi$ are quotient maps which have connected fibers.
Akin to the construction of \eqref{cech-2-xi}, we obtain the following natural morphisms of cohomology groups:
\begin{eqnarray*}
  \pi^{\ast}_{1}:\check{H}^{2}(X_{\xi}, \underline{\mathbb{R}}_{X_{\xi}})
  &\longrightarrow&
  \check{H}^{2}(V, \underline{\mathbb{R}}_{V}), \\
  \phi^{\ast}:\check{H}^{2}(X_{\epsilon}, \underline{\mathbb{R}}_{X_{\epsilon}})
  &\longrightarrow&
  \check{H}^{2}(V, \underline{\mathbb{R}}_{V}), \\
  \pi^{\ast}_{2}:\check{H}^{2}(X_{\zeta}, \underline{\mathbb{R}}_{X_{\zeta}})
  &\longrightarrow&
  \check{H}^{2}(V, \underline{\mathbb{R}}_{V}).
\end{eqnarray*}
This enables us to compare the K\"{a}hler classes $\check{\omega}_{a}$ in $\check{H}^{2}(V, \underline{\mathbb{R}}_{V})$ as $a$ crosses the wall.
Consider the curve in $\check{H}^{2}(V, \underline{\mathbb{R}}_{V})$ given by
\begin{equation}\label{curve}
\gamma(a):=
\begin{cases}
\pi^{\ast}_{1}(\check{\omega}_{a}),      &  a\in [\xi, \epsilon); \\
\phi^{\ast}(\check{\omega}_{\epsilon}),     & a=\epsilon; \\
\pi^{\ast}_{2}(\check{\omega}_{a}),      &  a\in (\epsilon, \zeta].
\end{cases}
\end{equation}
Then we have
\begin{thm}\label{comp-cla}
As $a$ varies along the line segment $l=[\xi, \zeta]$ (see Figure \ref{fig-2}), the curve defined by \eqref{curve} is continuous which is a broken line segment at the point $\gamma(\epsilon)$.
\end{thm}
\begin{proof}
According to Lemma \ref{key-lem}, the \v{C}ech cohomology group $\check{H}^{2}(X_{\epsilon},\underline{\mathbb{R}}_{X_{\epsilon}})$ can be thought of as a subgroup of
$\check{H}^{2}(X_{\xi},\underline{\mathbb{R}}_{X_{\xi}})$.
Let $\xi$ tend to $\epsilon$ along $l$.
Then the space $X_{\xi}$ contracts to $X_{\epsilon}$ under the map $f_{\xi,\epsilon}$ and the cohomologically symplectic structure $\check{\omega}_{\xi}$ on $X_{\xi}$ degenerates to $\check{\omega}_{\epsilon}$ on $X_{\epsilon}$.
%Define the curve $\sigma:[\xi, \epsilon]\rightarrow \check{H}^{2}(Z_{\xi}, \underline{\mathbb{R}}_{Z_{\xi}})$ by setting
%\begin{equation}\label{curve-sigma}
%\sigma(a):=
%\begin{cases}
%j^{\ast}_{a}(\eta)
%+\pi^{\ast}_{\xi}\circ f^{\ast}_{\xi,\epsilon}(\check{\omega}_{\epsilon}),
%&  a\in [\xi, \epsilon); \\
%\pi^{\ast}_{\xi}\circ f^{\ast}_{\xi,\epsilon}(\check{\omega}_{\epsilon}),     & a=\epsilon.
%\end{cases}
%\end{equation}
%From Lemma \ref{lem-6-13}, for any $a$ in $[\xi, \epsilon)$, we have
%$\pi^{\ast}_{\xi}(\check{\omega}_{a})=\pi^{\ast}_{\xi}\circ f^{\ast}_{\xi,\epsilon}(\check{\omega}_{\epsilon})$.
%This implies
%$
%\check{\omega}_{a}-f^{\ast}_{\xi,\epsilon}(\check{\omega}_{\epsilon})
%$
%belongs to
%$$
%\ker\,(\pi^{\ast}_{\xi})=\{\langle c_{\xi}, \lambda\rangle\,|\,\forall \,\lambda\in\mathfrak{t}^{\ast}\}
%$$
%due to \eqref{ker-pi}.
Due to Lemma \ref{limit-dh},  we obtain
$$
\check{\omega}_{a}= f^{\ast}_{\xi,\epsilon}(\check{\omega}_{\epsilon})+\langle c_{\xi}, a-\epsilon \rangle,
$$
for any $a\in[\xi, \epsilon)$.
It follows that the limit
$\lim_{a\rightarrow\epsilon}\check{\omega}_{a}$
in $\check{H}^{2}(X_{\xi},\underline{\mathbb{R}}_{X_{\xi}})$ is equal to
$ f^{\ast}_{\xi,\epsilon}(\check{\omega}_{\epsilon})$.
%As $[\xi, \epsilon)\ni a\rightarrow\epsilon$, the equality \eqref{dh-xi} derives
%\begin{equation*}\label{dh-xi-1}
%\check{\omega}_{\xi}=f^{\ast}_{\xi,\epsilon}(\check{\omega}_{\epsilon})+\langle c_{\xi}, \xi-\epsilon\rangle
%\end{equation*}
%in $\check{H}^{2}(X_{\xi}, \underline{\mathbb{R}}_{X_{\xi}})$.
Similarly, for any $b$ in $(\epsilon, \zeta]$, we have
\begin{equation*}\label{dh-zeta-1}
\check{\omega}_{b}=f^{\ast}_{\zeta,\epsilon}(\check{\omega}_{\epsilon})+\langle c_{\zeta}, b-\epsilon\rangle,
\end{equation*}
and the limit
$\lim_{b\rightarrow\epsilon}\check{\omega}_{b}$
in $\check{H}^{2}(X_{\xi},\underline{\mathbb{R}}_{X_{\zeta}})$ is also equal to
$ f^{\ast}_{\xi,\epsilon}(\check{\omega}_{\epsilon})$.
Due to the commutativity of \eqref{fiber-pro-0}, we have
\begin{equation*}\label{dh-xi-1}
\pi_{1}^{\ast}(\check{\omega}_{a})=\phi^{\ast}(\check{\omega}_{\epsilon})+
\pi_{1}^{\ast}(\langle c_{\xi}, a-\epsilon\rangle)\in
\check{H}^{2}(V, \underline{\mathbb{R}}_{V})
\end{equation*}
for any $a\in[\xi, \epsilon)$, and
\begin{equation*}\label{dh-zeta-1}
\pi_{2}^{\ast}(\check{\omega}_{b})=\phi^{\ast}(\check{\omega}_{\epsilon})+
\pi_{2}^{\ast}(\langle c_{\zeta}, b-\epsilon\rangle)\in
\check{H}^{2}(V, \underline{\mathbb{R}}_{V})
\end{equation*}
for any $b\in(\epsilon, \zeta]$.
This implies that the curve $\gamma(a)$ is a broken line segment at the point $\gamma(\epsilon)$.
\end{proof}
\begin{rem}
Consider the Hamiltonian circle action on a compact symplectic manifold.
Godinho \cite[Theroem A.1]{Go01} generalized the Duistermaat--Heckman theorem to interval of values of the moment map containing a critical value.
The theorem above can be considered as a slightly generalization of Godinho's result in the K\"{a}hler setting.
\end{rem}
%=============================
\section{Applications of the main theorem}\label{apps}
With the same setting as above, we assume that $(X, ds^{2}, T^{\mathbb{C}}, \Phi)$ is a compact K\"{a}hler Hamiltonian $T$-manifold.
Put $n=\mathrm{dim}_{\mathbb{C}}\,X$, $T=(\mathbb{S}^{1})^{d}$ and $G=T^{\mathbb{C}}$.
We shall state several applications of Theorem \ref{main-thm-1} to the bimeromorphic geometry of K\"{a}hler quotients.
Following \cite{Hu92}, we call $X_{\alpha}=X^{ss}(\Phi_{\alpha})/\!\!/G$ a \emph{nondegenerate K\"{a}hler quotient}, if $\alpha$ lies in the interior $\mathrm{int}\,\Delta$.
Otherwise, if $\alpha$ is in the boundary $\partial\,\Delta$ then we call $X_{\alpha}$ a \emph{degenerate K\"{a}hler quotient}.
It is noteworthy that a nondegenerate K\"{a}hler quotient has the expected pure dimension $n-d$.
%and the pure dimension of a degenerate K\"{a}hler quotient is strictly less than $n-d$.
%==================================
%==================================
\subsection{Algebraic dimensions of quotients}
The algebraic dimension is an important bimeromorphic invariant in complex geometry,
which distinguishes the algebraic-like and non-algebraic complex spaces based on their meromorphic function fields.
As a direct consequence of Theorem \ref{main-thm-1}, we get
\begin{thm}\label{inv-alg-dim}
Let $\alpha$ vary in $\mathrm{int}\,\Delta$, then the algebraic dimension of the K\"{a}hler quotient $X^{ss}_{\alpha}/\!\!/G$ is independent on the choice of $\alpha$.
\end{thm}
\begin{proof}
Recall that $\Delta$ is a union of several $d$-dimensional subpolytopes $\Delta_{1},\cdots,\Delta_{k}$ whose interiors are disjoint by walls.
If $\alpha$ varies in the interior of a subpolytope or in the relative interior of a wall, due to Lemma \ref{lem-xi-zeta} and the second assertion of Lemma \ref{sta}, the K\"{a}hler quotient $X^{ss}_{\alpha}/\!\!/G$ is independent on the choice of $\alpha$ and so is the algebraic dimension $a(X^{ss}_{\alpha}/\!\!/G)$.

To end the proof, it suffices to verify that if $\alpha\in\mathrm{int}\,\Delta$ lies in the boundary of a subpolytope $\Delta_{i}$ then we have
$a(X^{ss}_{\alpha}/\!\!/G)=a(X^{ss}_{\beta}/\!\!/G)$, where $\beta$ is an arbitrary point in the interior $\mathrm{int}\,\Delta_{i}$.
Due to Theorem \ref{main-thm-1}, we can construct a natural proper modification of K\"{a}hler quotients
$$
f_{\beta,\alpha}:X^{s}_{\beta}/\!\!/G\longrightarrow X^{ss}_{\alpha}/\!\!/G.
$$
By lifting of meromorphic functions, the holomorphic map $f_{\beta,\alpha}$
induces homomorphism of the fields of meromorphic functions:
\begin{equation}\label{mero-iso}
f_{\beta,\alpha}^{\ast}:K(X^{ss}_{\alpha}/\!\!/G)\longrightarrow
K(X^{s}_{\beta}/\!\!/G).
\end{equation}
In particular, the homomorphism \eqref{mero-iso} is an isomorphism of fields, since $f_{\beta,\alpha}$ is a modification (cf. \cite[Chapter VII, Proposition 6.7]{GPR94}).
It follows from Definition \ref{alg-dim} that both the quotients $X^{ss}_{\alpha}/\!\!/G$ and $X^{s}_{\beta}/\!\!/G$ have the same algebraic dimension, i.e., $a(X^{ss}_{\alpha}/\!\!/G)=a(X^{s}_{\beta}/\!\!/G)$.
Consequently, the argument above implies that, for any $\alpha$ in $\mathrm{int}\,\Delta$, the algebraic dimension of the K\"{a}hler quotient $a(X^{ss}_{\alpha}/\!\!/G)$ is constant.
\end{proof}

Note that a proper modification is a special example of bimeromorphic map.
It follows from Theorem \ref{main-thm-1} that, up to bimeromorphic equivalence, the K\"{a}hler quotient $\Phi^{-1}(\alpha)/T$, as a reduced normal complex space, is independent with the choice of $\alpha$ in $\mathrm{int}\,\Delta$.
In the bimeromorphic classification of compact K\"{a}hler manifolds, an interesting class are those with algebraic dimensions zero, i.e., compact K\"{a}hler manifolds having no non-constant meromorphic functions.
It is natural to raise the following
\begin{prob}
Does the condition of algebraic dimension zero descends to the K\"{a}hler quotients?
\end{prob}
If this problem has an affirmative answer, then we can use the technique of K\"{a}hler reduction to construct examples of compact K\"{a}hler manifolds of algebraic dimensions zero.
%==================================
\subsection{Desingularisations of singular quotients}
In order to extend the quantization commutes with reduction theorem to singular symplectic reduction, one has to resolve the non-orbifold singularities of the singular symplectic reduced spaces.
The reason lies in the fact that there is no obvious way to define a Riemann--Roch number as the index of an elliptic operator on a singular symplectic reduced space.

In the algebro-geometric setting, Kirwan \cite{Ki85} showed that there is a canonical \emph{partial desingularisation} of the non-orbifold singularities of the singular GIT-quotient as long as the set of stable points is nonempty.
Particularly, for quotient varieties of torus actions, via perturbing the value of the moment map, Hu \cite{Hu92} proved that there is a more explicit small resolution for a singular quotient.
In the symplectic setting, Meinrenken--Sjamaar \cite{MS99} showed that Kirwan's partial desingularisation procedure still works for symplectic reductions in a slightly less canonical sense;
moreover, the canonical partial desingularisations give the same Riemann--Roch numbers as  the shift desingularisations, via shifting the critical value of the moment map to a nearby regular value.

The purpose of this subsection is to show that there exists an analytic counterpart of \cite[Theorem 2.4]{Hu92} for singular K\"{a}hler quotient.
We start with some basic notions on singularities in the complex analytic setting.
\begin{defn}(\cite[Definition 5.8]{KM98})\label{rational-sin}
Let $V$ be a reduced complex space and $f:W\rightarrow V$ a resolution of singularities of $V$.
We say that $f:W\rightarrow V$ is a \emph{rational resolution} if
\begin{itemize}
  \item [(i)] $f_{\ast}\mathscr{O}_{W}=\mathscr{O}_{V}$ (equivalently, $V$ is normal), and
  \item [(ii)]$R^{i}f_{\ast}\mathscr{O}_{W}=0$ for $i>0$.
\end{itemize}
If every resolution of $V$ is rational, then we say that $V$ has \emph{rational singularities} or $V$ is of \emph{rational type}.
\end{defn}
%Suppose $V$ is a complex space admitting only quotient singularities.
By a \emph{partial desingularisation} of $V$, we mean a proper modification $\pi:Z\rightarrow V$ such that $Z$ is a complex orbifold.
In addition, if $\pi_{\ast}\mathscr{O}_{Z}=\mathscr{O}_{V}$ and $R^{i}\pi_{\ast}\mathscr{O}_{Z}=0$ for $i>0$, we call $\pi:Z\rightarrow V$ a \emph{rational} partial desingularisation of $V$.
In the K\"{a}hler case, as an application of Theorem \ref{main-thm-1}, we obtain the following result.
\begin{thm}\label{par-des}
For each singular nondegenerate K\"{a}hler quotient $X_{\alpha}$, there exists a canonical shift desingularization $f_{\beta,\alpha}:X_{\beta}\rightarrow X_{\alpha}$ which is partial and rational.
\end{thm}
\begin{proof}
Since $X_{\alpha}$ is singular and nondegenerate, there exists a $d$-dimensional subpolytope $\Delta_{i}$ of $\Delta=\Phi(X)$ such that $\alpha$ lies in the intersection of $\mathrm{int}\,\Delta$ with the boundary $\partial\Delta_{i}$.
Let $\beta\in\mathrm{int}\,\Delta_{i}$, then $\beta$ is a regular value of $\Phi$ and therefore the K\"{a}hler quotient $X_{\beta}$ has at most finite quotient singularities.
According to Theorem \ref{main-thm-1}, there is a canonical modification
\begin{equation}\label{p-desi-1}
f_{\beta,\alpha}:X_{\beta}\longrightarrow X_{\alpha}
\end{equation}
which gives rise to a partial desingularisation of $X_{\alpha}$.

Now, we turn to prove that \eqref{p-desi-1} is rational.
As a direct consequence of the holomorphic slice theorem (cf. \cite[(2.7) Theorem]{HL94} or \cite[Theorem 1.12]{Sj95}) and the result by Boutot \cite[Corollaire]{Bou87} on the singularities of reductive quotients, we obtain that K\"{a}hler quotients are of rational type, see \cite[Page 125]{Sj95} or \cite[Theorem 3]{BGLM24}.
Since $f_{\beta,\alpha}$ is a proper modification of normal complex spaces, we have
$(f_{\beta,\alpha})_{\ast}\mathscr{O}_{X_{\beta}}=
\mathscr{O}_{X_{\alpha}}$ (cf. \cite[Chapter I, Corollary 1.14]{Ue75}).
Let $g:\widetilde{X}_{\beta}\rightarrow X_{\beta}$ be a resolution of singularities of $X_{\beta}$.
Then we get a commutative diagram
\begin{equation*}
\vcenter{
\xymatrix@=1cm{
  \widetilde{X}_{\beta} \ar[d]_{g} \ar[dr]^{h}        \\
  X_{\beta} \ar[r]^{f_{\beta,\alpha}}  & X_{\alpha}             }
  }
\end{equation*}
where $h=f_{\beta,\alpha}\circ g$.
Note that both $g$ and $f_{\beta,\alpha}$ are proper modifications.
Consequently, the holomorphic map $h$ is proper and surjective; moreover, it is a modification.
Consider the higher direct images $R^{i}(f_{\beta,\alpha}\circ g)_{\ast}\mathscr{O}_{\widetilde{X}_{\beta}}$.
The Grothendieck spectral sequence of $f_{\beta,\alpha}\circ g$ having the second term
$$
E^{s,t}_{2}=
R^{s}(f_{\beta,\alpha})_{\ast}\circ R^{t}g_{\ast}(\mathscr{O}_{\widetilde{X}_{\beta}})
$$
converges to
$
R^{s+t}(f_{\beta,\alpha}\circ g)_{\ast}\mathscr{O}_{\widetilde{X}_{\beta}}=
R^{i}h_{\ast}\mathscr{O}_{\widetilde{X}_{\beta}}
$
where $s+t=i$ (cf. \cite[Theorem 5.18]{GPR94}).
Since $X_{\beta}$ is of rational type, we get $R^{0}g_{\ast}\mathscr{O}_{\widetilde{X}_{\beta}}=\mathscr{O}_{X_{\beta}}$ and $R^{i}g_{\ast}\mathscr{O}_{\widetilde{X}_{\beta}}=0$ for all $i>0$.
This implies that the spectral sequence above degenerates at $E_{2}$-term and thus
\begin{eqnarray*}
  R^{i}h_{\ast}\mathscr{O}_{\widetilde{X}_{\beta}}&\cong& E^{i,0}_{\infty}
  = E^{i,0}_{2}\\
  &=&R^{i}(f_{\beta,\alpha})_{\ast}\bigl( R^{0}g_{\ast}(\mathscr{O}_{\widetilde{X}_{\beta}})\bigr)\\
  &=&R^{i}(f_{\beta,\alpha})_{\ast}\mathscr{O}_{X_{\beta}}.
\end{eqnarray*}

We will prove $R^{i}h_{\ast}\mathscr{O}_{\widetilde{X}_{\beta}}=0$ for $i>0$.
Let $q:\widetilde{X}_{\alpha}\rightarrow X_{\alpha}$ be a resolution of singularities of $X_{\alpha}$ and $C\subset X_{\alpha}$ the center of the modification $h:\widetilde{X}_{\beta}\rightarrow X_{\alpha}$.
Put $V=C\cup(X_{\alpha})_{\mathrm{sing}}$ which is a closed complex analytic subset of $X_{\alpha}$.
Denote by $\psi$ the composition of $h|_{\widetilde{X}_{\beta}\setminus h^{-1}(V)}$ and $(q|_{\widetilde{X}_{\alpha}\setminus q^{-1}(V)})^{-1}$ which gives rise to a biholomorphism
$$
\psi:\widetilde{X}_{\beta}\setminus h^{-1}(V)\longrightarrow \widetilde{X}_{\alpha}\setminus q^{-1}(V).
$$
Let $Y$ be the closure of the graph $\Gamma(\psi)$ in $\widetilde{X}_{\alpha}\times\widetilde{X}_{\beta}$.
Then $Y$ is an irreducible complex analytic subset of $\widetilde{X}_{\alpha}\times\widetilde{X}_{\beta}$, and the projection of  $\widetilde{X}_{\alpha}\times\widetilde{X}_{\beta}$ onto $\widetilde{X}_{\alpha}$ induces a proper holomorphic map $p_{1}:Y\rightarrow\widetilde{X}_{\alpha}$.
Since $E:=q^{-1}(V)$ is a nowhere dense complex analytic subset of $\widetilde{X}_{\alpha}$ and $Y\setminus p^{-1}_{1}(E)=\Gamma(\psi)$ is biholomorphic to $\widetilde{X}_{\alpha}\setminus E$ under $p_{1}$, it follows that $p_{1}:Y\rightarrow\widetilde{X}_{\alpha}$ is a modification.
Likewise, we can show that $p_{2}:Y\rightarrow\widetilde{X}_{\beta}$ is also a modification.
Let $l:\widetilde{Y}\rightarrow Y$ be a resolution of singularities of $Y$.
Then we get a commutative diagram
\begin{equation}\label{resol-y}
\vcenter{
\xymatrix@C=1cm{
  \widetilde{Y} \ar@/_/[ddr]_{r_{\beta}} \ar@/^/[drr]^{r_{\alpha}}
    \ar@{>}[dr]|-{l}                   \\
   & Y \ar[d]_{p_{2}} \ar[r]^{p_{1}}
                      & \widetilde{X}_{\alpha} \ar[d]^{q}    \\
   & \widetilde{X}_{\beta} \ar[r]^{h}     & X_{\alpha}              }
}
\end{equation}
where $r_{\alpha}=p_{1}\circ l$ and $r_{\beta}=p_{2}\circ l$.
The commutativity of \eqref{resol-y} deduces
\begin{equation}\label{equ-0}
R^{i}(q\circ r_{\alpha})_{\ast}\mathscr{O}_{\widetilde{Y}}=
R^{i}(h\circ r_{\beta})_{\ast}\mathscr{O}_{\widetilde{Y}}.
\end{equation}
Note that $r_{\alpha}$ and $r_{\beta}$ in \eqref{resol-y} are modifications of compact complex manifolds.
By a result of Hironaka \cite[Proposition 2.14]{Ue75}, we have the vanishing of the higher direct images $R^{i}(r_{\alpha})_{\ast}\mathscr{O}_{\widetilde{Y}}$
and $R^{i}(r_{\beta})_{\ast}\mathscr{O}_{\widetilde{Y}}$ for all $i>0$, thereby the Grothendiek spectral sequences of $q\circ r_{\alpha}$ and $h\circ r_{\beta}$ degenerate at $E_{2}$-terms.
On the one hand, since $X_{\alpha}$ is of rational type we get
\begin{equation}\label{equ-1}
  R^{i}(q\circ r_{\alpha})_{\ast}\mathscr{O}_{\widetilde{Y}}
  \cong R^{i}q_{\ast}\bigl( R^{0}(r_{\alpha})_{\ast}\mathscr{O}_{\widetilde{Y}}\bigr)
  \cong R^{i}q_{\ast}\mathscr{O}_{\widetilde{X}_{\alpha}}
  =0. \quad\quad(\forall\,\,i>0)
\end{equation}
On the other hand, we have
\begin{equation}\label{equ-2}
R^{i}(h\circ r_{\beta})_{\ast}\mathscr{O}_{\widetilde{Y}}
=R^{i}h_{\ast}\mathscr{O}_{\widetilde{X}_{\beta}}.
\end{equation}
Combining \eqref{equ-0}-\eqref{equ-2} derives
$R^{i}h_{\ast}\mathscr{O}_{\widetilde{X}_{\beta}}=0$ and therefore
$R^{i}(f_{\beta,\alpha})_{\ast}\mathscr{O}_{X_{\beta}}=0$ for all $i>0$.
The proof is complete.
\end{proof}

In general, suppose $(V, \mathscr{O}_{V})$ is a compact reduced irreducible complex space having rational singularities only.
Note that the finite quotient singularities are rational, see \cite[Proposition 5.15]{KM98}.
Without essential changes, using the same argument in the proof of Theorem \ref{par-des}, we can prove the following result
\begin{prop}
Every partial desingularisation of $(V, \mathscr{O}_{V})$ is rational.
\end{prop}

Let $(V, \mathscr{O}_{V})$ be a pure dimensional compact complex space and $\mathscr{F}$ a coherent sheaf of $\mathscr{O}_{V}$-modules on $V$.
As $\mathbb{C}$-vector spaces, the sheaf cohomology $H^{i}(V, \mathscr{F})$ are finite dimensional and $H^{i}(V, \mathscr{F})=0$ for all $i>\mathrm{dim}_{\mathbb{C}}\,V$.
The \emph{Riemann--Roch number} (or \emph{holomorphic Euler characteristic}) of $\mathscr{F}$ is defined to be
$$
RR(V, \mathscr{F})=\sum_{i\geq0}(-1)^{i}\mathrm{dim}_{\mathbb{C}}\,H^{i}(V, \mathscr{F}).
$$
Particularly, if $\mathscr{F}=\mathscr{O}_{V}$ we call $RR(V):=RR(V, \mathscr{O}_{V})$ the Riemann--Roch number of $V$.
Consider the Riemann--Roch numbers of K\"{a}hler quotients.
As a corollary of the above theorem, we have
\begin{cor}\label{inv-euler}
All nondegenerate K\"{a}hler quotients of $(X, ds^{2}, T^{\mathbb{C}}, \Phi)$ have the same Riemann--Roch number.
\end{cor}
\begin{proof}
It suffices to verify that a singular nondegenrate K\"{a}hler quotient $X_{\alpha}$ and its shift desingularisations have the same Riemann--Roch number.
Let $f_{\beta,\alpha}:X_{\beta}\rightarrow X_{\alpha}$ be a shift desingularisation.
According to the Leray spectral sequence theorem \cite[Theorem 13.8]{Dem12}, there exists a spectral sequence $\{E_{r}, d_{r}\}$ with the second term
$$
E^{s,t}_{2}=H^{s}(X_{\alpha}, R^{t}(f_{\beta,\alpha})_{\ast}\mathscr{O}_{X_{\beta}})
$$
which converges to $H^{\ast}(X_{\beta}, \mathscr{O}_{X_{\beta}})$.
Due to Theorem \ref{par-des}, we have
$(f_{\beta,\alpha})_{\ast}\mathscr{O}_{X_{\beta}}=\mathscr{O}_{X_{\alpha}}$ and
$R^{i}(f_{\beta,\alpha})_{\ast}\mathscr{O}_{X_{\beta}}=0$ for $i>0$.
It follows that the Leray spectral sequence $\{E_{r}, d_{r}\}$ degenerates at $E_{2}$.
Consequently, as $\mathbb{C}$-vector spaces,
$$
H^{i}(X_{\alpha}, \mathscr{O}_{X_{\alpha}})=
H^{i}(X_{\alpha}, (f_{\beta,\alpha})_{\ast}\mathscr{O}_{X_{\beta}})
$$
is isomorphic to
$H^{i}(X_{\beta}, \mathscr{O}_{X_{\beta}})$
and hence $RR(X_{\beta})=RR(X_{\alpha})$ is valid.
\end{proof}

%=================================
\subsection{Geometric quantization}
Assume that $L$ is a holomorphic line bundle over $X$ admitting a Hermitian metric $h$ with the property that the K\"{a}hler form $\omega$ is equal to $-(2\pi\sqrt{-1})^{-1}\Theta$, where $\Theta$ is the curvature form of the Hermitian connection $\nabla$ of $h$.
From definition, $L$ is positive and therefore due to the Kodaira Embedding Theorem \cite[Page 181]{GH94} $X$ is embeddable in a projective space.

Next suppose that the tours $T$ acts on $L$ by line bundle automorphisms preserving the Hermitian metric $h$.
Then the Hermitian connection $\nabla$ is invariant and $\omega$ is invariant under the induced $T$-action on $X$.
For each $\xi$ in $\mathfrak{t}$, denote by $\xi_{L}$ the vector field on $L$ generated by $\xi$.
Let $\xi_{X,\mathrm{hor}}$ be the horizontal lift of $\xi_{X}$ to $TL$ with respect to the connection $\nabla$ and $\nu_{L}$ be the vector field on $L$ generating the circle action defined by fibrewise multiplication by complex numbers of length 1.
By Kostant's formula, there exists a unique smooth function $\Phi^{\xi}$ on $X$ such that
\begin{equation}\label{k}
\xi_{L}=\xi_{X,\mathrm{hor}}+2\pi\Phi^{\xi}\nu_{L}
\end{equation}
and $d\Phi^{\xi}=\imath_{\xi_{X}}\omega$.
This implies that $(X, \omega)$ is a K\"{a}hler  Hamiltonian $T$-manifold with the moment map $\Phi:X\rightarrow\mathfrak{t}^{\ast}$ given by \eqref{k}.

Let $s$ be a holomorphic section of $L$ and $\langle s,s\rangle$ the norm of $s$ with respect to the Hermitian metric $h$.
For any integer $q>0$,  consider the $q$-th tensor product $L^{q}:=L^{\otimes q}$.
Then $s^{q}:=s^{\otimes q}$ becomes a holomorphic section of $L^{q}$.
Moreover, the metric $h$ induces a unique Hermitian metric on $L^{q}$, denoted by $h^{(q)}$, such that the norm of $s^{q}$ with respect to $h^{(q)}$ satisfies
$$
\langle s^{q}, s^{q}\rangle^{(q)}=(\langle s,s\rangle)^{q}.
$$

Put $\mathscr{L}=\mathscr{O}(L)$ the sheaf of holomorphic sections of $L$.
The Dolbeault theorem states that the sheaf cohomology $H^{i}(X, \mathscr{L})$ is isomorphic to $H^{0,i}_{\bar{\partial}}(X, L)$.
For any $\alpha\in\Delta=\Phi(X)$, consider the K\"{a}hler quotient $\Pi_{\alpha}:X^{ss}_{\alpha}\rightarrow X_{\alpha}$.
Note that $\mathscr{L}$ is a coherent $G$-equivariant sheaf on $X$, see \cite{Rob86}.
The invariant direct image $\mathscr{L}_{\alpha}:=(\Pi_{\alpha})^{G}_{\ast}(\mathscr{L}|_{X^{ss}_{\alpha}})$
is defined by setting
$$
(\Pi_{\alpha})^{G}_{\ast}(\mathscr{L}|_{X^{ss}_{\alpha}})(U)=
\mathscr{L}(\Pi^{-1}_{\alpha}(U))^{G}
$$
for any open subset $U$ of $X_{\alpha}$.
By a result of Roberts \cite{Rob86}, $\mathscr{L}_{\alpha}$ is a coherent $\mathscr{O}_{X_{\alpha}}$-modules.
In general, $\mathscr{L}_{\alpha}$ is not locally free necessarily.
We say that $L$ \emph{descends to} $X_{\alpha}$ if $\mathscr{L}_{\alpha}$ is locally free, i.e., $L|_{X^{ss}_{\alpha}}$ descends to a holomorphic line bundle $L_{\alpha}$ on $X_{\alpha}$ such that $\mathscr{L}_{\alpha}=\mathscr{O}(L_{\alpha})$.
If some power of $L$ descends to $X_{\alpha}$, we say that $L$ \emph{descends fractionally to} $X_{\alpha}$.
According to Kempf's descend lemma \cite[Th\'{e}or\`{e}me  2.3]{DN89}, $L$ descends fractionally to $X_{\alpha}$, provided that $\alpha=0$ or $\alpha$ is a regular value of $\Phi$.
By the quantizations $\mathscr{H}$ and $\mathscr{H}_{0}$, we mean the virtual vector spaces
$$
\mathscr{H}=\bigoplus_{j\geq0}(-1)^{j}H^{j}(X, \mathscr{L})
$$
and
$$
\mathscr{H}_{0}=\bigoplus_{j\geq0}(-1)^{j}H^{j}(X_{0}, \mathscr{L}_{0}).
$$
Note that, for any integer $i\geq0$, the $i$-th Dolbeault cohomology $H^{0,i}_{\bar{\partial}}(X, L)$ is a $T$-representation,
and we denote its $T$-invariant part by $H^{0,i}_{\bar{\partial}}(X, L)^{T}\cong H^{i}(X, \mathscr{L})^{T}$.
Then the Riemann--Roch numbers $RR^{T}(X, \mathscr{L})$ and $RR(X_{0}, \mathscr{L}_{0})$ are defined by
$$
RR(X, \mathscr{L})^{T}
=\sum_{j\geq0}(-1)^{j}\mathrm{dim}_{\mathbb{C}}H^{j}(X, \mathscr{L})^{T}
$$
and
$$
RR(X_{0}, \mathscr{L}_{0})
=\sum_{j\geq0}(-1)^{j}\mathrm{dim}_{\mathbb{C}}H^{j}(X_{0}, \mathscr{L}_{0}).
$$
If $0$ is a regular value of $\Phi$, the \emph{quantization commutes with reduction} theorem \cite{GS82b,JK97,Mer98,TZ98,Br01} states that
$$
RR(X_{0}, \mathscr{L}_{0})=RR(X, \mathscr{L})^{T}.
$$

Originally, Guillemin--Sternberg \cite{GS82b} asserted the conjecture of quantization commutes with reduction in a more general setting: $K$ is a compact Lie group and $(X, \omega)$ is a compact symplectic manifold with a prequantum data.
Under the guiding principle of quantization commutes with reduction, extensive works have been done by many authors, for detailed results and comments we refer to the survey \cite{Ver02} and references therein.
Particularly, the theorem of quantization commutes with reduction has been generalized to various non-compact settings \cite{MZ14,Par15,HM17,Par18} and non-symplectic settings \cite{Mer12,GMW18,GMW21,BLS21,LLSS22,LLSS24}.

Henceforth, we assume that $X_{0}$ is a singular nondegenerate K\"{a}hler quotient.
Due to Theorem \ref{par-des}, there exists a canonical shift desigularization $f_{\beta,0}:X_{\beta}\rightarrow X_{0}$.
Consider the commutative diagram:
\begin{equation}\label{a-b}
\vcenter{
\xymatrix@=1cm{
  X^{s}_{\beta} \ar[d]_{\Pi_{\beta}}
  \ar[r]^{i_{\beta, 0}} & X^{ss}_{0} \ar[d]^{\Pi_{0}} \\
  X_{\beta}
  \ar[r]^{f_{\beta, 0}} & X_{0}}
  }
\end{equation}
For any open subset $V$ of $X_{0}$, by definition, we have
$\mathscr{L}_{0}(V)=\mathscr{L}(\Pi_{0}^{-1}(V))^{G}$ and
$(f_{\beta,0})_{\ast}\mathscr{L}_{\beta}(V)=
\mathscr{L}(\Pi_{\beta}^{-1}((f_{\beta,0})^{-1}(V)))^{G}$.
Observe that
$$
\Pi_{\beta}^{-1}((f_{\beta,0})^{-1}(V))=
i^{-1}_{\beta,0}(\Pi_{0}^{-1}(V))=
X^{s}_{\beta}\cap\Pi_{0}^{-1}(V).
$$
The restriction map
\begin{equation}\label{r-v}
\mathfrak{r}(V):\mathscr{L}(\Pi_{0}^{-1}(V))^{G}\longrightarrow \mathscr{L}(X^{s}_{\beta}\cap\Pi_{0}^{-1}(V))^{G}
\end{equation}
gives rise to a morphism of $\mathscr{O}_{X_{0}}(V)$-modules denoted by
$$
f^{\natural}_{\beta,0}(V):\mathscr{L}_{0}(V)\longrightarrow
(f_{\beta,0})_{\ast}\mathscr{L}_{\beta}(V)
$$
and therefore we obtain a sheaf morphism
\begin{equation}\label{la=lb}
f^{\natural}_{\beta,0}:\mathscr{L}_{0}\longrightarrow
(f_{\beta,0})_{\ast}\mathscr{L}_{\beta}.
\end{equation}

Let $q$ be the positive integer in \cite[Proposition 2.15]{Sj95} such that $L^{q}|_{X^{ss}_{0}}$ descends to a holomorphic line bundle $\mathbb{L}_{0}$ over the quotient $X_{0}$.
Note that $\Pi_{0}^{\ast}\mathbb{L}_{0}=L^{q}|_{X^{ss}_{0}}$ and then
$i_{\beta,0}^{\ast}\circ\Pi_{0}^{\ast}\mathbb{L}_{0}
=L^{q}|_{X^{s}_{\beta}}$.
By the commutativity of \eqref{a-b}, we get $\Pi_{0}\circ i_{\beta,0}=f_{\beta,0}\circ\Pi_{\beta}$, and thus
\begin{equation}\label{p-f-0}
\Pi_{\beta}^{\ast}\circ f_{\beta,0}^{\ast}\mathbb{L}_{0}=L^{q}|_{X^{s}_{\beta}}.
\end{equation}
From \cite[Lemma 2.13, (ii)]{Sj95}, the following isomorphism is valid
\begin{equation}\label{p-f-f}
(\Pi_{\beta})^{G}_{\ast}\mathscr{O}\bigl(\Pi_{\beta}^{\ast}\circ f_{\beta,0}^{\ast}
\mathbb{L}_{0}\bigr)\cong\mathscr{O}\bigl(f_{\beta,0}^{\ast}
\mathbb{L}_{0}\bigr).
\end{equation}
Combining \eqref{p-f-0} with \eqref{p-f-f} derives the isomorphism
$$
(\Pi_{\beta})^{G}_{\ast}\bigl(\mathscr{O}(L^{q}|_{X^{s}_{\beta}})\bigr)
\cong\mathscr{O}\bigl(f_{\beta,0}^{\ast}
\mathbb{L}_{0}\bigr).
$$
This implies that $L^{q}|_{X^{s}_{\beta}}$ descends to a holomorphic line bundle $\mathbb{L}_{\beta}\rightarrow X_{\beta}$ satisfying $f_{\beta,0}^{\ast}\mathbb{L}_{0}\cong\mathbb{L}_{\beta}$.

As $0$ is a critical value of $\Phi$, it is a rather restrictive condition that the sheaf $\mathscr{L}_{0}$ is locally free.
Naturally, it is reasonable to ask whether the Riemann--Roch numbers of $\mathscr{L}_{0}$ and $\mathscr{L}_{\beta}$ are equal, see also \cite[Page 126, (i)]{Sj95}.
Inspired by \cite[Lemma 5.8]{Tel00}, using an algebro-geometric argument, we will prove the following:
\begin{thm}\label{L}
For any integer $i\geq0$, we have a canonical isomorphism
$$
H^{i}(X_{0}, \mathscr{L}_{0})\cong H^{i}(X_{\beta}, \mathscr{L}_{\beta}).
$$
\end{thm}

The idea of the proof is to show that $\mathscr{L}_{0}$ is isomorphic to $(f_{\beta,0})_{\ast}\mathscr{L}_{\beta}$ and the higher direct images $R^{>0}(f_{\beta,0})_{\ast}\mathscr{L}_{\beta}$ are vanishing, and then we can use the Leray spectral sequence theorem.
The main difficulty lies in the fact that $\mathscr{L}_{0}$ is not locally free and thus the projection formula does not works.

We shall prove the following:
\begin{lem}\label{lem-iso}
The sheaf morphism \eqref{la=lb} is an isomorphism.
\end{lem}
\begin{proof}
It suffices to show that the morphism of the stalks
\begin{equation}\label{stalk-iso}
(f^{\natural}_{\beta,0})_{x}:(\mathscr{L}_{0})_{x}\longrightarrow
\bigl((f_{\beta,0})_{\ast}\mathscr{L}_{\beta}\bigr)_{x}
\end{equation}
is an isomorphism at each point $x$ of $X_{0}$.
Let $\sigma_{x}$ be a germ of $(\mathscr{L}_{0})_{x}$.
By definition, the germ $\sigma_{x}$ can be represented by a pair $(\Pi^{-1}_{0}(V), \sigma)$,
where $V$ is an open neighborhood of $x$ in $X_{0}$ and $\sigma$ is a $G$-invariant section of $L$ over $\Pi^{-1}_{0}(V)$.
Consider the restriction map \eqref{r-v}.
Suppose $\sigma|_{X^{s}_{\beta}\cap\Pi^{-1}_{0}(V)}=0$, we claim $\sigma=0$.
The local holomorphic trivialization of $L$ implies that, for each $p$ in $\Pi^{-1}_{0}(V)$, there exists an open neighborhood $W_{p}$ of $p$ in $X^{ss}_{0}$ such that
$L|_{W_{p}}$ is biholomorphic to the product space $W_{p}\times\mathbb{C}$.
Put $W^{\prime}_{p}=W_{p}\cap\Pi^{-1}_{0}(V)$, which is open in $X^{ss}_{0}$.
Then we have
$$
\Pi^{-1}_{0}(V)=\bigcup_{p\in\Pi^{-1}_{0}(V)}W^{\prime}_{p},
$$
and we can find a nowhere vanishing holomorphic section $e:W^{\prime}_{p}\rightarrow L$ such that $\sigma|_{W^{\prime}_{p}}=g\cdot e$ for some holomorphic function $g$ on $W^{\prime}_{p}$.
For any $p\in\Pi^{-1}_{0}(V)\setminus X^{s}_{\beta}$, since $X^{s}_{\beta}$ is open and dense in $X$ we have $W^{\prime}_{p}\cap X^{s}_{\beta}\neq\emptyset$, and thus $g|_{W^{\prime}_{p}}=0$ due to the identity theorem.
This implies $\sigma|_{W^{\prime}_{p}}=0$ for any $p\in\Pi^{-1}_{0}(V)$ and therefore $\sigma=0$.

Now we prove the surjectivity of \eqref{stalk-iso}.
Note that $X$ is a projective manifold.
On account of a result by Heinzner--Migliorini \cite[Page 168, Semistability Theorem]{HM01}, there exists a very ample line bundle $\tilde{L}$ over $X$ together with a lifting of the $G$-action such that $X^{ss}_{\beta}=X^{s}_{\beta}$ is identical with the set of $\tilde{L}$-semistable points $X^{ss}(\tilde{L})$ in the sense of Mumford.
%Here a point $x$ of $X$ is $\tilde{L}$-semistable if and only if there exist $m\in\mathbb{N}$ and a $G$-invariant holomorphic section $s\in\Gamma(X,\tilde{L}^{\otimes m})$ such that $s(x)\neq0$.
This implies that $X^{s}_{\beta}=X^{ss}(\tilde{L})$ is a Zariski open subset of $X$, and therefore its complement is Zariski closed.
Consequently, the $\Phi_{\beta}$-unstable set $X^{us}_{\beta}=X\setminus X^{s}_{\beta}$ is a complex analytic subvariety $X$.
Observe that
$$X^{ss}_{0}\setminus X^{s}_{\beta}=X^{us}_{\beta}\cap X^{ss}_{0}$$
which is a closed subset of $X^{ss}_{0}$.
Since $B$ is a closed complex analytic subset of $X_{0}$ and the strictly semistable set $X^{sss}_{0}=X^{ss}_{0}\setminus X^{s}_{0}$ is identical with the inverse image $\Pi^{-1}_{0}(B)$, we get that $X^{sss}_{0}$ is a closed complex analytic subset of $X^{ss}_{0}$.
Note that $X^{ss}_{0}$ is irreducible, we get $\mathrm{codim}_{\mathbb{C}}(X^{sss}_{0}, X^{ss}_{0})\geq1$ and therefore $X^{sss}_{0}$ is nowhere dense in $X^{ss}_{0}$.
Since $X^{s}_{0}\subset X^{s}_{\beta}$ the subset $X^{ss}_{0}\setminus X^{s}_{\beta}$ is contained in $X^{sss}_{0}$, and hence $X^{ss}_{0}\setminus X^{s}_{\beta}$ is thin\footnote{Recall that a closed subset $A$ of a complex space $X$ is called \emph{thin} in $X$, if every point $p$ has an open neighborhood $U$ such that $A\cap U$ is contained in a nowhere dense analytic subset of $U$.} in $X^{ss}_{0}$.

Let $q$ be the positive integer in \cite[Proposition 2.15]{Sj95}.
Then $L^{q}|_{X^{ss}_{0}}$ and $L^{q}|_{X^{s}_{\beta}}$ descend to the holomorphic line bundles $\mathbb{L}_{0}\rightarrow X_{0}$ and $\mathbb{L}_{\beta}\rightarrow X_{\beta}$ satisfying $\mathbb{L}_{\beta}\cong f^{\ast}_{\beta,0}\mathbb{L}_{0}$.
Consider the natural morphism of sheaves
\begin{equation}\label{L-q-b-0}
f^{\natural}_{\beta,0}:\mathscr{O}(\mathbb{L}_{0})\longrightarrow
(f_{\beta,0})_{\ast}\mathscr{O}(\mathbb{L}_{\beta}).
\end{equation}
Let $W$ be an open subset in $X_{0}$, if $\mathbb{L}_{0}|_{W}$ admits a holomorphic trivialization then so is the bundle $\mathbb{L}_{\beta}|_{f^{-1}_{\beta,0}(W)}$.
From definition, for any $x\in X_{0}$, we have
$$
\mathscr{O}(\mathbb{L}_{0})_{x}=\varinjlim_{x\in U}\Gamma(U, \mathbb{L}_{0})
$$
and
$$
((f_{\beta,0})_{\ast}\mathscr{O}(\mathbb{L}_{\beta}))_{x}=\varinjlim_{x\in U}
\Gamma(f^{-1}_{\beta, 0}(U), \mathbb{L}_{\beta}).
$$
Here $U$ runs over all open neighborhoods of $x$ in $X_{0}$.
Let $s_{x}$ be a germ of $((f_{\beta,0})_{\ast}\mathscr{O}(\mathbb{L}_{\beta}))_{x}$ with a representative $(V, s)$, where $V=f^{-1}_{\beta, 0}(U)$ and $U$ is an open neighborhood of $x$ in $X_{0}$ such that $\mathbb{L}_{0}|_{U}$ is biholomorphic to the product space $U\times\mathbb{C}$.
It follows that $\mathbb{L}_{\beta}|_{V}$ is biholomorphic to $V\times\mathbb{C}$.
Let $e:U\rightarrow\mathbb{L}_{0}$ be a nowhere vanishing holomorphic section and then $\tilde{e}=e\circ(f_{\beta, 0}|_{V})$ is a nowhere vanishing holomorphic section of $\mathbb{L}_{\beta}$ on $V$.
Moreover, there exists a holomorphic function $g$ on $V$ such that $s=g\cdot\tilde{e}$.
Particularly, we can choose $U$  so small such that $g$ is bounded on $V$.
Observe that $f_{\beta,0}:X_{\beta}\rightarrow X_{0}$ is a proper modification of normal complex spaces.
The argument in the proof of \cite[Proposition 1.13]{Ue75} shows that there exists a holomorphic function $h$ on $U$ such that $g=h\circ(f_{\beta, 0}|_{V})$.
Set $t=h\cdot e$ which is holomorphic section of $\mathbb{L}_{0}$ on $U$.
Denote by $t_{x}=[(U, t)]$ the germ of $t$ at $x$, then we get $(f^{\natural}_{\beta, 0})_{x}(t_{x})=s_{x}$, i.e., \eqref{L-q-b-0} is surjective.

Let $\hat{s}_{x}$ be a germ of $((f_{\beta,0})_{\ast}\mathscr{L}_{\beta})_{x}$ with a representative $(U^{\prime}, \hat{s})$.
Observe that
$$
(f_{\beta,0})_{\ast}\mathscr{L}_{\beta}(U^{\prime})=
\mathscr{L}(\Pi^{-1}_{\beta}(f^{-1}_{\beta,0}(U^{\prime})))^{G}
=
\mathscr{L}(X^{s}_{\beta}\cap\Pi_{0}^{-1}(U^{\prime}))^{G}
$$
and $\hat{s}^{q}$ is a $G$-invariant section of $L^{q}|_{X^{s}_{\beta}\cap\Pi^{-1}_{0}(U^{\prime})}$
which gives rise to a germ $(\hat{s}^{q})_{x}$ in $((f_{\beta,0})_{\ast}\mathscr{O}(\mathbb{L}_{\beta}))_{x}$.
Because the morphism of the stalks
\begin{equation*}
(f^{\natural}_{\beta,0})_{x}:\mathscr{O}(\mathbb{L}_{0})_{x}\longrightarrow ((f_{\beta,0})_{\ast}\mathscr{O}(\mathbb{L}_{\beta}))_{x}
\end{equation*}
is surjective, there exists a germ $\hat{t}_{x}$ in $\mathscr{O}(\mathbb{L}_{0})_{x}$ such that $(f^{\natural}_{\beta,0})_{x}(\hat{t}_{x})=(\hat{s}^{q})_{x}$.
Let $(U^{\prime\prime}, \hat{t})$ be a representative of $\hat{t}_{x}$.
Here $U^{\prime\prime}$ is an open neighborhood of $x$ in $X_{0}$ and $\hat{t}$ is a $G$-invariant holomorphic section of $L^{q}|_{\Pi^{-1}_{0}(U^{\prime\prime})}$ satisfying the condition
$$
\hat{t}|_{X^{s}_{\beta}\cap\Pi^{-1}_{0}(W)}=\hat{s}^{q}|_{X^{s}_{\beta}\cap\Pi^{-1}_{0}(W)}
$$
for some open neighborhood $W$ of $x$ with $W\subset U^{\prime}\cap U^{\prime\prime}$.
This implies that the function
$$
\langle \hat{s}^{q}, \hat{s}^{q}\rangle^{(q)}|_{X^{s}_{\beta}\cap\Pi^{-1}_{0}(W)}
=(\langle \hat{s}, \hat{s}\rangle)^{q}|_{X^{s}_{\beta}\cap\Pi^{-1}_{0}(W)}
$$
can be extended to
$\langle \hat{t}, \hat{t}\rangle^{(q)}|_{\Pi^{-1}_{0}(W)}.$
Put
$$
C=\Pi^{-1}_{0}(W)\setminus(X^{s}_{\beta}\cap\Pi^{-1}_{0}(W))
=\Pi^{-1}_{0}(W)\cap(X^{ss}_{0}\setminus X^{s}_{\beta}),
$$
which is a thin subset in $\Pi^{-1}_{0}(W)$.
As a result, the function $\langle \hat{s}^{q}, \hat{s}^{q}\rangle^{(q)}|_{X^{s}_{\beta}\cap\Pi^{-1}_{0}(W)}$ is bounded near $C$, i.e., for any $\tilde{x}\in C$ there exists an open neighborhood $\tilde{V}$ of $\tilde{x}$ in $\Pi^{-1}_{0}(W)$ such that $\langle \hat{s}^{q}, \hat{s}^{q}\rangle^{(q)}$ is bounded on $\tilde{V}\setminus C$; and so is the function $\langle \hat{s}, \hat{s}\rangle|_{X^{s}_{\beta}\cap\Pi^{-1}_{0}(W)}$.

Combining the Riemann Extension Theorem on complex manifolds (cf. \cite[Page 132]{GR84}) with the local holomorphic trivialization of $L|_{\Pi^{-1}_{0}(W)}$ derives that $\hat{s}|_{X^{s}_{\beta}\cap\Pi^{-1}_{0}(W)}$ has a unique holomorphic extension $\tilde{s}$ on $\Pi^{-1}_{0}(W)$.
It is not a priori clear that $\tilde{s}$ is $G=T^{\mathbb{C}}$-invariant.
However, the following argument shows the $G$-invariance of $\tilde{s}$.
Note that a holomorphic local section of $L$ on $\Pi^{-1}_{0}(W)$ is $G=T^{\mathbb{C}}$-invariant if and only if it is $T$-invariant.
Using an averaging argument, we can define
$$
\bar{s}(\tilde{x})=\int_{T}t\cdot\tilde{s}(x)dt
$$
for any $\tilde{x}\in\Pi^{-1}_{0}(W)$.
Then $\bar{s}$ becomes a $T$-invariant holomorphic local section of $L$ on $\Pi^{-1}_{0}(W)$ satisfying $\bar{s}|_{X^{s}_{\beta}\cap\Pi^{-1}_{0}(W)}=\tilde{s}$.
The fact $\bar{s}|_{X^{s}_{\beta}\cap\Pi^{-1}_{0}(W)}=
\tilde{s}|_{X^{s}_{\beta}\cap\Pi^{-1}_{0}(W)}$ derives $\bar{s}=\tilde{s}$ since the identity theorem.
The pair $(W, \tilde{s})$ gives rise to a germ $\tilde{s}_{x}$ of $(\mathscr{L}_{0})_{x}$ such that $(f^{\natural}_{\beta,0})_{x}(\tilde{s}_{x})=\hat{s}_{x}$ since $(W, \hat{s}|_{W})$ is a representative of $\hat{s}_{x}$ .
This implies that the morphism of stalks \eqref{stalk-iso} is surjective and therefore we are led to the conclusion that the sheaf morphism \eqref{la=lb} is an isomorphism.
\end{proof}

Consider the following commutative square of holomorphic maps
\begin{equation}\label{com-b-0}
\vcenter{
\xymatrix@=1cm{
  X^{s}_{\beta} \ar[d]_{\Pi_{\beta}}
  \ar[r]^{i_{\beta, 0}} & X^{ss}_{0} \ar[d]^{\Pi_{0}} \\
  X_{\beta}
  \ar[r]^{f_{\beta, 0}} & X_{0}}
  }
\end{equation}
Denote by $\mathfrak{S}^{G}_{\mathrm{coh}}(X^{s}_{\beta})$ (resp. $\mathfrak{S}^{G}_{\mathrm{coh}}(X^{ss}_{0})$) the category of coherent analytic $G$-sheaves on $X^{s}_{\beta}$ (resp. $X^{ss}_{0}$).
Since $X^{s}_{\beta}$ is a $G$-invariant open subset of $X^{ss}_{0}$, the direct image functor
\begin{equation}\label{i-functor}
(i_{\beta,0})_{\ast}:\mathfrak{S}^{G}_{\mathrm{coh}}(X^{s}_{\beta})
\longrightarrow \mathfrak{S}^{G}_{\mathrm{coh}}(X^{ss}_{0})
\end{equation}
is exact.
Let $\mathfrak{S}_{\mathrm{coh}}(X_{\beta})$ and $\mathfrak{S}_{\mathrm{coh}}(X_{0})$ be the categories of coherent analytic sheaves on $X_{\beta}$ and $X_{0}$ respectively.
Then we have the invariant direct image functors
$$
(\Pi_{\beta})^{G}_{\ast}:\mathfrak{S}^{G}_{\mathrm{coh}}(X^{s}_{\beta})
\longrightarrow
\mathfrak{S}_{\mathrm{coh}}(X_{\beta})
$$
and
$$
(\Pi_{0})^{G}_{\ast}:\mathfrak{S}^{G}_{\mathrm{coh}}(X^{ss}_{0})
\longrightarrow
\mathfrak{S}_{\mathrm{coh}}(X_{0}).
$$

\begin{prop}\label{pi-exact}
The functors $(\Pi_{\beta})^{G}_{\ast}$ and $(\Pi_{0})^{G}_{\ast}$ are exact.
\end{prop}
\begin{proof}
By the definition of analytic Hilbert quotient, the morphism $$\Pi_{\beta}:X^{s}_{\beta}\longrightarrow X_{\beta}=X^{s}_{\beta}/G$$
is locally Stein, which means that there exists an open covering $\mathscr{U}=\{U_{\lambda}\}$ of $X_{\beta}$ by Stein subspaces such that $V_{\lambda}:=\Pi^{-1}_{\beta}(U_{\lambda})$ is a Stein subspace of $X^{s}_{\beta}$ for all $\lambda$.
Let
$$
\xymatrix@C=1cm{
  0 \ar[r] & \mathscr{F} \ar[r]^{\phi} & \mathscr{G}\ar[r]^{\psi} & \mathscr{H} \ar[r] & 0 }
$$
be a short exact sequence of coherent analytic $G$-sheaves $X^{s}_{\beta}$.
As the functor $(\Pi_{\beta})^{G}_{\ast}$ is left exact, for any $x\in X_{\beta}$, we obtain an exact sequence of stalks
\begin{equation}\label{ex-pi-beta}
\xymatrix@C=1.4cm{
  0 \ar[r] & ((\Pi_{\beta})^{G}_{\ast}\mathscr{F})_{x} \ar[r]^{(\Pi_{\beta})^{G}_{\ast}(\phi)_{x}} & ((\Pi_{\beta})^{G}_{\ast}\mathscr{G})_{x}\ar[r]^{(\Pi_{\beta})^{G}_{\ast}(\psi)_{x}} & ((\Pi_{\beta})^{G}_{\ast}\mathscr{H})_{x}. }
\end{equation}
To show the exactness of $(\Pi_{\beta})^{G}_{\ast}$, we only need to verify that $(\Pi_{\beta})^{G}_{\ast}(\psi)_{x}$ in \eqref{ex-pi-beta} is surjective.
Observe that $x\in U_{\lambda}$ for some open Stein subspace $U_{\lambda}\in\mathscr{U}$ and the sheaf $(\Pi_{\beta})^{G}_{\ast}\mathscr{H}|_{U_{\lambda}}$ is a coherent $\mathscr{O}_{U_{\lambda}}$-modules.
It follows from the Theorem A for Stein spaces \cite[\S\,0.37, Theorem]{Fi76} that the stalk $((\Pi_{\beta})^{G}_{\ast}\mathscr{H})_{x}$ is generated by the germs of the sections over $U_{\lambda}$.
As a result, for any element $s_{x}$ in $((\Pi_{\beta})^{G}_{\ast}\mathscr{H})_{x}$,
there exists a section
$$
\tilde{s}\in\Gamma(U_{\lambda}, (\Pi_{\beta})^{G}_{\ast}\mathscr{H})=
\Gamma(V_{\lambda}, \mathscr{H})^{G}
$$
satisfying $\tilde{s}_{x}=s_{x}$.
Since $V_{\lambda}$ is a Stein subspace of $X^{s}_{\beta}$, the following sequence is exact
$$
\xymatrix@C=1cm{
  0 \ar[r] & \Gamma(V_{\lambda}, \mathscr{F}) \ar[r]^{\phi(V_{\lambda})} & \Gamma(V_{\lambda}, \mathscr{G})\ar[r]^{\psi(V_{\lambda})} & \Gamma(V_{\lambda}, \mathscr{H}) \ar[r] & 0. }
$$
This implies $\tilde{s}=\psi(V_{\lambda})(\hat{s})$ for some $\hat{s}\in\Gamma(V_{\lambda}, \mathscr{G})$.
Recall that $G=T^{\mathbb{C}}$.
Averaging $\hat{s}$ over $T$ defines a $G$-invariant section $\bar{s}\in\Gamma(V_{\lambda}, \mathscr{G})^{G}$ such that $\psi(V_{\lambda})(\bar{s})=\tilde{s}$, and therefore $\bar{s}$ defines a germ $\bar{s}_{x}$ in $((\Pi_{\beta})^{G}_{\ast}\mathscr{H})_{x}$ such that
$(\Pi_{\beta})^{G}_{\ast}(\psi)_{x}(\bar{s}_{x})=\tilde{s}_{x}=s_{x}$.
Consequently, we are led to the conclusion that the functor $(\Pi_{\beta})^{G}_{\ast}$ is exact and we can prove the exactness of the functor $(\Pi_{0})^{G}_{\ast}$ using the same argument.
\end{proof}

To finish the proof of Theorem \ref{L}, we need the following lemma.
\begin{lem}\label{hd=0}
The higher direct image of $\mathscr{L}_{\beta}$ along the morphism $f_{\beta,0}$ is vanishing, i.e., we have
$R^{j}(f_{\beta,0})_{\ast}\mathscr{L}_{\beta}=0$, for all $j>0$.
\end{lem}
\begin{proof}
Consider the composition of the functors
$$
(f_{\beta})_{\ast}\circ(\Pi_{\beta})^{G}_{\ast}:
\mathfrak{S}^{G}_{\mathrm{coh}}(X^{s}_{\beta})
\longrightarrow
\mathfrak{S}_{\mathrm{coh}}(X_{0})
$$
and
$$
(\Pi_{0})^{G}_{\ast}\circ(i_{\beta,0})_{\ast}:
\mathfrak{S}^{G}_{\mathrm{coh}}(X^{s}_{\beta})
\longrightarrow
\mathfrak{S}_{\mathrm{coh}}(X_{0}).
$$
For any open subset $U$ in $X_{0}$, by the commutativity of \eqref{com-b-0}, we obtain
\begin{eqnarray*}
  \Gamma\bigl(U, (f_{\beta,0})_{\ast}\circ(\Pi_{\beta})^{G}_{\ast}\mathscr{F}\bigr)
  &=& \Gamma\bigl(f^{-1}_{\beta,0}(U), (\Pi_{\beta})^{G}_{\ast}\mathscr{F}\bigr)\\
  &=& \Gamma\bigl(\Pi^{-1}_{\beta}(f^{-1}_{\beta,0}(U)), \mathscr{F}\bigr)^{G} \\
  &=& \Gamma\bigl(i^{-1}_{\beta,0}(\Pi^{-1}_{0}(U)), \mathscr{F}\bigr)^{G}
\end{eqnarray*}
and
\begin{eqnarray*}
  \Gamma\bigl(U, (\Pi_{0})^{G}_{\ast}\circ(i_{\beta,0})_{\ast}\mathscr{F}\bigr)
  &=& \Gamma\bigl(\Pi^{-1}_{0}(U), (i_{\beta,0})_{\ast}\mathscr{F}\bigr)^{G}\\
  &=& \Gamma\bigl(i^{-1}_{\beta,0}(\Pi^{-1}_{0}(U)), \mathscr{F}\bigr)^{G},
\end{eqnarray*}
where $\mathscr{F}$ is an arbitrary coherent analytic $G$-sheaf on $X^{s}_{\beta}$.
This implies that the functor $(f_{\beta})_{\ast}\circ(\Pi_{\beta})^{G}_{\ast}$ is identical with the functor $(\Pi_{0})^{G}_{\ast}\circ(i_{\beta,0})_{\ast}$.
Consequently, we obtain
\begin{equation}\label{7-eq-0}
R^{j}\bigl((f_{\beta,0})_{\ast}\circ(\Pi_{\beta})^{G}_{\ast}\bigr)
\mathscr{O}(L|_{X^{s}_{\beta}})
=
R^{j}\bigl((\Pi_{0})^{G}_{\ast}\circ(i_{\beta,0})_{\ast}\bigr)
\mathscr{O}(L|_{X^{s}_{\beta}})
\end{equation}
for all $j>0$.
Due to Proposition \ref{pi-exact}, the invariant direct image functor $(\Pi_{\beta})^{G}_{\ast}$ is exact, and so is the functor $(\Pi_{0})^{G}_{\ast}$.
The exactness of the functor $(\Pi_{\beta})^{G}_{\ast}$ shows
$$
R^{t}(\Pi_{\beta})^{G}_{\ast}\mathscr{O}(L|_{X^{s}_{\beta}})=0
$$
for all $t>0$, and thus the Grothendieck spectral sequence for the left hand side of \eqref{7-eq-0} degenerates at the $E_{2}$-term.
More precisely, we get
$$
R^{j}(f_{\beta,0})_{\ast}\circ(\Pi_{\beta})^{G}_{\ast}\mathscr{O}(L|_{X^{s}_{\beta}})=
E^{j,0}_{2}=E^{j,0}_{\infty}\cong R^{j}\bigl((f_{\beta,0})_{\ast}\circ(\Pi_{\beta})^{G}_{\ast}\bigr)
\mathscr{O}(L|_{X^{s}_{\beta}})
$$
which means
\begin{equation}\label{7-equ-1}
R^{j}\bigl((f_{\beta,0})_{\ast}\circ(\Pi_{\beta})^{G}_{\ast}\bigr)
\mathscr{O}(L|_{X^{s}_{\beta}})\cong
R^{j}(f_{\beta,0})_{\ast}\mathscr{L}_{\beta}.
\end{equation}
On account of the exactness of the functors \eqref{i-functor} and $(\Pi_{0})^{G}_{\ast}$, the composition functor $(\Pi_{0})^{G}_{\ast}\circ(i_{\beta,0})_{\ast}$ is exact.
This implies
\begin{equation}\label{7-eq-2}
R^{j}\bigl((\Pi_{0})^{G}_{\ast}\circ(i_{\beta,0})_{\ast}\bigr)
\mathscr{O}(L|_{X^{s}_{\beta}})=0
\end{equation}
for all $j>0$.
Collecting \eqref{7-eq-0}-\eqref{7-eq-2} derives
$R^{j}(f_{\beta,0})_{\ast}\mathscr{L}_{\beta}=0$ for all $j>0$.
\end{proof}
\begin{rem}
Taking $j=0$ in the proof of Lemma \ref{hd=0}, it seems possible to prove the isomorphism $(f_{\beta,0})_{\ast}\mathscr{L}_{\beta}\cong\mathscr{L}_{0}$.
Using the projection formula, it suffices to verify $(i_{\beta,0})_{\ast}\mathscr{O}_{X^{s}_{\beta}}\cong\mathscr{O}_{X^{ss}_{0}}$ which is an extension problem of holomorphic functions.
Note that we can not use the classical Riemann Extension Theorem directly to solve this problem.
However, using the properties of the modification $f_{\beta,0}$, we can obtain Lemma \ref{lem-iso} which gives a canonical isomorphism from $\mathscr{L}_{0}$ to $(f_{\beta,0})_{\ast}\mathscr{L}_{\beta}$.
\end{rem}

We are now in a position to give the proof of Theorem \ref{L}.
\begin{proof}[Proof of Theorem \ref{L}]
Due to Lemma \ref{lem-iso} and Lemma \ref{hd=0}, there exists a natural sheaf isomorphism
\begin{equation*}
f^{\natural}_{\beta,0}:\mathscr{L}_{0}\stackrel{\simeq}\longrightarrow
(f_{\beta,0})_{\ast}\mathscr{L}_{\beta},
\end{equation*}
and the higher direct images of $\mathscr{L}_{\beta}$ along $f_{\beta,0}$ vanishes for all $j>0$.
Using the Leray spectral sequence theorem \cite[Theorem 13.8]{Dem12} again, we get
a canonical isomorphism $H^{i}(X_{0}, \mathscr{L}_{0})\cong H^{i}(X_{\beta}, \mathscr{L}_{\beta})$ for all $i\geq0$.
As a direct consequence, the Riemann--Roch numbers of $\mathscr{L}_{0}$ and $\mathscr{L}_{\beta}$ are equal.
Since $\beta$ is a regular value, $X_{\beta}$ is a K\"{a}hler  orbifold and $\mathscr{L}_{\beta}$ is the sheaf of holomorphic sections of the holomorphic orbifold line bundle $L_{\beta}$.
As a result, the Riemann--Roch number $RR(X_{\beta}, \mathscr{L}_{\beta})$ can be computed by Kawasaki's Riemann--Roch theorem in \cite{Kaw79}.
\end{proof}

Assume that $T$ acts on $\Phi^{-1}(\beta)$ freely, then $X_{\beta}=X^{s}_{\beta}/T^{\mathbb{C}}$ is a K\"{a}hler manifold and $L|_{X^{s}_{\beta}}$ descends to a holomorphic line bundle $L_{\beta}$ over $X_{\beta}$.
According to \cite[Theorem 0.1]{Zh99}, the following refined quantization formula for bundle-valued Dolbeault cohomologies holds
\begin{equation}\label{r-q-f}
\mathrm{dim}_{\mathbb{C}}\,H^{0,i}_{\bar{\partial}}(X, L)^{T}=
\mathrm{dim}_{\mathbb{C}}\,H^{0,i}_{\bar{\partial}}(X_{\beta}, L_{\beta})
\end{equation}
for any integer $i\geq0$.
By the Dolbeault theorem, $H^{0,i}_{\bar{\partial}}(X_{\beta}, L_{\beta})$ is isomorphic to the sheaf cohomology $H^{i}(X_{\beta}, \mathscr{L}_{\beta})$.
From the identity \eqref{r-q-f} and Theorem \ref{L}, we obtain
\begin{cor}
If the $T$-action on $\Phi^{-1}(\beta)$ is free, then we have
$$
\mathrm{dim}_{\mathbb{C}}\,H^{i}(X, L)^{T}=
\mathrm{dim}_{\mathbb{C}}\,H^{i}(X_{0}, \mathscr{L}_{0})
$$
for any integer $i\geq0$.
\end{cor}

Without the freeness condition of the $T$-action on $\Phi^{-1}(\beta)$, we obtain a holomorphic orbifold line bundle $L_{\beta}\rightarrow X_{\beta}$.
If we replace the cohomology with the Dolbeault cohomology of orbifold line bundle, the equality \eqref{r-q-f} still holds and so does the equality in the corollary.

\begin{rem}
Comparing the assertion in \cite[Lemma 5.8]{Tel00}, it seems reasonable to expect that Theorem \ref{L} holds for the case that 0 lies in the boundary $\partial\,\Delta$.
However, our argument does not work if $0\in\partial\,\Delta$ since $f_{\beta,0}$ is not a proper modification in this case.
Furthermore, a natural problem is to study the comparison problem of K\"{a}hler quotients by general actions of a compact connected Lie group $K$ and its complexification.
\end{rem}
%==================================
\appendix
%========================
\section{K\"{a}hler structures on complex spaces}\label{kah-spa}
Let $(X,\mathscr{O}_{X})$ be a reduced complex space.
Denote by $\mathscr{C}_{X,\mathbb{R}}$ the sheaf of continuous real-valued functions on $X$.
By a \emph{pluriharmonic function} on $X$, we mean a function which locally is the real part of some holomorphic function, and we define the sheaf $\mathscr{PH}_{X,\mathbb{R}}$ of pluriharmonic functions.
A continuous (resp. smooth) function $f: X\rightarrow\mathbb{R}$ is a continuous (resp. smooth) \emph{plurisubharmonic} (p.s.h.) function on $X$, if locally it is the restriction of a continuous (resp. smooth) plurisubharmonic function on an open subset in $\mathbb{C}^{N}$ for a local embedding of $X$ into $\mathbb{C}^{N}$.
Given a continuous or smooth p.s.h. function $f$, if for each smooth function $h$ and each relatively compact open subset $U$ there exists $\epsilon>0$ such that $f+t\cdot h$ is p.s.h. on $U$ for all $|t|<\epsilon$, then we call $f$ a \emph{strictly plurisubharmonic function} (in the sense of perturbations).
We denote by sheaves $\mathscr{P}^{0}_{X}$ (resp. $\mathscr{P}^{\infty}_{X}$) of continuous (resp. smooth) p.s.h.  functions, and $\mathscr{SP}^{0}_{X}$ (resp. $\mathscr{SP}^{\infty}_{X}$) of continuous (resp. smooth) strictly plurisubharmonic functions.
Put $\mathscr{K}^{1}_{X,\mathbb{R}}:
=\mathscr{C}_{X,\mathbb{R}}/\mathscr{PH}_{X,\mathbb{R}}$.
\begin{defn}({cf. \cite[Definition 55]{HS20}})\label{kahler space}
A \emph{K\"{a}hler metric} on $X$ is a global section $\kappa\in \mathscr{K}^{1}_{X,\mathbb{R}}(X)$ which can be represented by an open covering $\{U_{\alpha}\}$ of $X$ together with a family of strictly plurisubharmonic functions $\{\varphi_{\alpha}:U_{\alpha}\rightarrow\mathbb{R}\}$ such that $(\varphi_{\alpha}-\varphi_{\beta})|_{U_{\alpha}\cap U_{\beta}}$ is pluriharmonic.
We write $\kappa=\{(U_{\alpha}, \varphi_{\alpha})\}$ and a \emph{K\"{a}hler space} is a reduced complex space equipped with a K\"{a}hler metric.
\end{defn}

\begin{rem}\label{rem-ka}
Comparing to \cite[\S\,II, 1.2]{Va89}, a K\"{a}hler metric on $X$ is given by a global section of the sheaf $\mathscr{C}^{\infty}_{X,\mathbb{R}}/\mathscr{PH}_{X,\mathbb{R}}$ represented by a system of sections of $\mathscr{SP}^{\infty}_{X}$.
According to \cite[\S\,II, Theorem 1]{Va89}, if $X$ admits a K\"{a}hler metric in the sense of Definition \ref{kahler space} then it is a K\"{a}hler space in the sense of \cite[\S\,II, 1.2]{Va89}.
\end{rem}

Given a K\"{a}hler metric $\kappa=\{(U_{\alpha}, \varphi_{\alpha})\}$ on $X$, then $\kappa$ determines a \v{C}ech cohomology class in $\check{H}^{0}(X, \mathscr{K}^{1}_{X,\mathbb{R}})$.
Let $\underline{\mathbb{R}}_{X}$ be the sheaf of locally constant real-valued functions on $X$ and $\underline{\mathbb{R}}_{X}\hookrightarrow\mathscr{O}_{X}$ the embedding via multiplication by $\textbf{i}=\sqrt{-1}$.
By definition, there exist two natural short exact sequences of sheaves on $X$:
\begin{equation}\label{s-e-ph}
\xymatrix@C=0.5cm{
  0 \ar[r] & \mathscr{PH}_{X,\mathbb{R}} \ar[r]^{} & \mathscr{C}_{X,\mathbb{R}}
   \ar[r]^{} & \mathscr{K}^{1}_{X,\mathbb{R}} \ar[r] & 0 }
\end{equation}
and
\begin{equation}\label{re-part}
\xymatrix@C=0.5cm{
  0 \ar[r] & \underline{\mathbb{R}}_{X} \ar[r]^{\textbf{i}\cdot} & \mathscr{O}_{X}
   \ar[r]^{\mathrm{Re}\quad} & \mathscr{PH}_{X,\mathbb{R}} \ar[r] & 0,}
\end{equation}
where $\mathrm{Re}$ is the real part map.
By \eqref{s-e-ph} and \eqref{re-part}, we deduce a canonical connecting homomorphism in degree 0:
\begin{equation}\label{delta1}
\delta^{0}: \check{H}^{0}(X, \mathscr{K}^{1}_{X,\mathbb{R}})\longrightarrow \check{H}^{1}(X, \mathscr{PH}_{X,\mathbb{R}}).
\end{equation}
and a canonical connecting homomorphism in degree 1:
\begin{equation}\label{delta2}
\delta^{1}: \check{H}^{1}(X, \mathscr{PH}_{X,\mathbb{R}})\longrightarrow \check{H}^{2}(X, \underline{\mathbb{R}}_{X}).
\end{equation}
Composing \eqref{delta1} with \eqref{delta2} gives rise to a canonical morphism:
\begin{equation*}\label{c}
c_{1}=\delta^{1}\circ\delta^{0}: \check{H}^{0}(X, \mathscr{K}^{1}_{X,\mathbb{R}})\longrightarrow \check{H}^{2}(X, \underline{\mathbb{R}}_{X}).
\end{equation*}
For any K\"{a}hler metric $\kappa$ in $X$, we call the cohomology class $c_{1}(\kappa)\in \check{H}^{2}(X, \underline{\mathbb{R}}_{X})$ the \emph{K\"{a}hler class} of the metric $\kappa$.
%=============================
\section{Modifications and bimeromorphic maps}\label{mod-bim}
In this appendix, we give a rapid review of basic notions in bimeromorphic geometry, and for a nice reference we refer to \cite[Chapter I]{Ue75} or \cite[\S\,2.1]{MM07}.
All complex analytic spaces are assumed to be \emph{irreducible} and \emph{reduced}.

Given a complex analytic space $(X, \mathscr{O}_{X})$, the regular locus $X_{\mathrm{reg}}$ is a connected complex manifold since $X$ is irreducible.
The \emph{dimension} of $X$ is defined to be the dimension of  $X_{\mathrm{reg}}$.
Let $A$ be an irreducible analytic subset of $X$, the \emph{dimension} of $A$ is defined by
$
\mathrm{dim}\,A=\mathrm{dim}\,A_{\mathrm{reg}}.
$
If $A$ has several irreducible components $\{A_{l}\}$, then we set
$
\mathrm{dim}\,A=\max\,\{\mathrm{dim}\,A_{l}\}.
$
The \emph{codimension} of $A$ is defined by
$
\mathrm{codim}\,A=\mathrm{dim}\,X-\mathrm{dim}\,A.
$
Every analytic subset $A$ of $X$ is equipped with a natural structure of a reduced complex analytic subspace of $(X, \mathscr{O}_{X})$ with the structure sheaf $\mathscr{O}_{A}:=(\mathscr{O}_{X}/\mathscr{I}_{A})|_{A}$, where $\mathscr{I}_{A}$ is the coherent ideal sheaf of $A$.
\begin{defn}\label{mod}
A morphism of complex analytic spaces $f: (X, \mathscr{O}_{X})\rightarrow(Y, \mathscr{O}_{Y})$ is
called a \emph{proper modification}, if it is proper and surjective and if there exists an analytic subset $A\subset Y$ such that
\begin{itemize}
  \item [(i)] $A\subset Y$ and $f^{-1}(A)\subset X$ are nowhere dense analytic subsets;
  \item [(ii)] the restriction $f:X\setminus f^{-1}(A)\rightarrow Y\setminus A$ is biholomorphic.
\end{itemize}
\end{defn}

The analytic subset $A$ is called the \emph{center }of the modification and the set $\mathrm{Ex}(f):=f^{-1}(A)$ is called the \emph{exceptional locus}.
By a \emph{prime divisor} on $(X, \mathscr{O}_{X})$, we mean an irreducible and reduced complex subspace of codimension 1.
We say that a proper modification $f:X\rightarrow Y$ is \emph{divisorial} if the exceptional locus $\mathrm{Ex}(f)$  is a primer divisor,  and is \emph{small} if $\mathrm{Ex}(f)$ has codimension $\geq$2.
A special example of proper modification is the blow-up, for definition see \cite[Proposition 2.1.14]{MM07}; meanwhile, modifications are bimeromorphic maps in the following sense.

\begin{defn}\label{mero}
Let $(X, \mathscr{O}_{X})$ and $(Y, \mathscr{O}_{Y})$ be two complex analytic spaces.
A \emph{meromorphic map} $f: X\dashrightarrow Y$ is defined to be a map from $X$ into the
power set of $Y$ satisfies the following conditions:
\begin{itemize}
  \item [(i)] The graph $\mathrm{Graph}(f)=\{(x,y)\in X\times Y\,|\,y\in f(x)\}$ is an irreducible analytic
      subset in $X\times Y$.
  \item [(ii)] The projection $p_{X}: \mathrm{Graph}(f)\rightarrow X$ is a proper modification.
\end{itemize}
Additionally, if the projection $p_{Y}: \mathrm{Graph}(f)\rightarrow Y$ is also a proper modification, we call $f:
X\dashrightarrow Y$ a \emph{bimeromorphic map} and we say that $X$ and $Y$ are
\emph{bimeromorphically equivalent}.
\end{defn}

Let $(X, \mathscr{O}_{X})$ be a compact complex analytic space of pure dimension $n$.
Denote by $K(X)$ the field of meromorphic functions on $X$.
\begin{defn}\label{alg-dim}
The \emph{algebraic dimension} of $(X, \mathscr{O}_{X})$, denoted by $a(X)$, is defined to be transcendence degree of $K(X)$ over $\mathbb{C}$, i.e.,
$$
a(X)=\mathrm{trdeg}_{\mathbb{C}}\,K(X).
$$
\end{defn}
A theorem of Remmert \cite{Re56} shows that the dimension of $X$ bounds the algebraic dimension of $X$, i.e., $a(X)\leq n=\mathrm{dim}\,X$.
A compact complex analytic space $(X, \mathscr{O}_{X})$ is said to be a \emph{Moishezon space} if the algebraic dimension of $X$ is equal to its dimension.

%==============================================================

\end{document}